\documentclass[10pt]{amsart}
\usepackage{amsmath}
\usepackage{amsfonts}
\usepackage{amsthm}
\usepackage{amssymb}
\usepackage{url}
\usepackage{verbatim}
\usepackage[usenames,dvips]{color}
\usepackage{graphicx}
\usepackage{xcolor}
\usepackage{mathrsfs}
\usepackage{lmodern}
\usepackage{bm}
\usepackage{mathtools}

\numberwithin{equation}{section}
\numberwithin{subsection}{section}
\newtheorem{theorem}{Theorem}
\newtheorem{lemma}{Lemma}[section]
\newtheorem{corollary}[lemma]{Corollary}
\newtheorem{definition}[lemma]{Definition}
\newtheorem{remark}[lemma]{Remark}

\newtheorem{proposition}[lemma]{Proposition}

\newcommand{\di}{\mathrm{d}}

\newcommand{\bme}{\bm{e}}

\newcommand{\bmF}{\bm{F}}

\newcommand{\bmN}{\bm{N}}

\newcommand{\bmu}{\bm{u}}

\DeclareMathOperator{\Div}{\mathrm{div}}

\DeclareMathOperator{\pd}{\partial}

\DeclareMathOperator{\sech}{\mathrm{sech}}

\newcommand{\In}{\qquad \text{in} \;}
\newcommand{\at}{\qquad \text{at} \; }

\usepackage[colorlinks=true]{hyperref}
\title[]{  Controllability for 2D water waves: effects of bottom topography and constant vorticity }

\author{E. Haus$^{(\star)}$}
\address[$\star$]{Dipartimento di Matematica e Fisica, Universit\`a degli Studi Roma Tre, Largo San Leonardo Murialdo 1, 00146, Rome, Italy}
\email[$\star$]{emanuele.haus@uniroma3.it}

\author{S. Pasquali$^{(\ast)}$}
\address[$\ast$]{International School for Advanced Studied (SISSA), via Bonomea 265, 34136, Trieste, Italy}
\email[$\ast$]{stefano.pasquali@sissa.it}

\begin{document}

\begin{abstract}
In this paper we consider two-dimensional water waves, under the action of gravity and surface tension. We prove a controllability result for irrotational waves in a fluid domain with finite depth and general bottom topography. The result holds for an open and dense set of bottom topographies in $H^{s+1/2}(\mathbb{T})$ (where $s$ is sufficiently large) that do not touch the free surface. We point out that the bottom topographies that we allow for are not necessarily small perturbations of the flat bottom case: this leads to many technical difficulties, since the eigenvalues of the Dirichlet--Neumann operator at a still free surface with general bottom topography are not explicit, nor are they necessarily close (for low frequencies) to the eigenvalues of the corresponding operator for the flat bottom case. In turn, this leads to a more involved argument to prove Ingham-type estimates, which are needed to prove observability, and motivates the restriction mentioned above on the admissible bottom topographies.

We also prove a controllability result for waves with constant vorticity in a fluid domain with flat bottom topography.  
\end{abstract}

\maketitle

\tableofcontents

\section{Introduction} \label{sec:intro}

The controllability of PDEs is a fairly well-established topic (see \cite{micu2004introduction}). The first results were obtained for linear PDEs, including the water waves system (see \cite{reid1985boundary,reid1986open,reid1995control}). Subsequently, controllability results have been extended to semilinear PDEs, also for models arising from fluid dynamics, such as the KdV equation (see \cite{rosier1997exact,laurent2010control}) and the Benjamin-Ono equation \cite{laurent2015control}. More recently, results have also been obtained for quasilinear PDEs (see, for instance, \cite{baldi2017exact,baldi2018controllability,iandoli2024controllability}). 

In the context of quasi-linear PDEs arising in fluid dynamics, controllability results are important for justifying physical experiments in water tanks; an important controllability result has been proved by Alazard, Baldi and Han-Kwan \cite{alazard2018control} for irrotational two-dimensional gravity-capillary water waves in a domain with flat bottom. More recently, this result has been extended to three-dimensional irrotational waves \cite{zhu2020control}, and also to two-dimensional hydroelastic waves \cite{wan2026exact}.\\

In this paper we extend the exact controllability result of \cite{alazard2018control} to irrotational two-dimensional gravity-capillary water waves in a domain with non-constant bottom topography and finite depth. Our controllability result holds true for ``many" choices of bottom topographies which are sufficiently regular (namely, on an open and dense subset of a given ball $B \subset H^{s+1/2}(\mathbb{T})$, where $s$ is sufficiently large) and do not touch the free surface. 

We point out that we consider the case in which the bottom topography is not necessarily a small perturbation of the flat bottom case. This is a very important aspect which leads to many technical difficulties, since the eigenvalues of the Dirichlet--Neumann operator at a still free surface with general bottom topography are not explicit, nor are they necessarily close (at low frequencies) to the eigenvalues of the Dirichlet--Neumann operator for the flat bottom case. The explicit description of the spectrum of the Dirichlet--Neumann operator for general bottom topography is a very interesting and difficult problem (see the recent partial results \cite{craig2018bloch,lacave2025bloch} for non-periodic waves, based on Bloch-Floquet theory).

In order to circumvent this issue, we prove that for any $N>0$ sufficiently large the following holds true: for an open and dense subset of a given ball $B \subset H^{s+1/2}(\mathbb{T})$, the eigenvalues $( \mu(n) )_{|n| \leq N}$ of the operator $A_{0,\beta}$ (see \eqref{eq:uSymmLin}; this operator is associated to the Dirichlet--Neumann operator at still free surface and general bottom topography) are distinct, see Lemma \ref{lem:LowFreq}. This allows us to prove Ingham-type estimates in Sec. \ref{sec:Ingham}, which are needed in order to prove an observability property in Sec. \ref{sec:Obs}. 

The result that we prove is a generalization of the main result of \cite{alazard2018control}, namely an exact controllability result for small and sufficiently regular data, for an arbitrary time $T>0$, using an external pressure $\mathscr{P}_{ext} \in C([0,T], H^s(\mathbb{T}))$ supported for any time $t \in [0,T]$ on a fixed non-empty open subset $\Omega \subset \mathbb{T}$. \\

We also prove an exact controllability result for two-dimensional gravity-capillary water waves with constant vorticity in a domain with flat bottom topography. Two-dimensional constant vorticity waves have received a lot of attention in recent years, due to results which exploit the Hamiltonian structure of the model \cite{constantin2007nearly,wahlen2007hamiltonian,berti2021traveling,berti2024pure,barbieri2026bifurcation} (see also \cite{pasquali2026two} for constant vorticity waves with general bottom topography). Our controllability result holds true for any value of the vorticity, extending the main result of \cite{alazard2018control}. Due to the explicit form of the dispersion relation, the proof for the constant vorticity case with flat bottom is simpler than the one for the irrotational case with general bottom topography; hence, we will just mention the main differences in the proof compared to the first result. \\

\emph{Structure of the paper.} In Sec. \ref{subsec:model} and \ref{subsec:modelGamma} we present the models for irrotational water waves and general bottom topography and for water waves with constant vorticity and flat bottom topography. In Sec. \ref{sec:main} we present the main results; the proof of Theorem \ref{thm:main} is divided into several parts; see Secs. \ref{sec:DNO}--\ref{sec:NashMoser}. In Sec. \ref{sec:DNO} we study the analytical properties of the classical Dirichlet--Neumann operator. In Sec \ref{sec:symm} we perform the symmetrization of the water waves system, reducing it to the single paradifferential equation \eqref{eq:uEqNew}. Next, in Sec. \ref{sec:Red} we reduce Eq. \eqref{eq:uEqNew} to a classical PDE, and we prove that it suffices to prove a $L^2$-controllability result in order to deduce the result in higher-regularity Sobolev spaces (see Corollary \ref{cor:Reduction}). In Sec.\ref{sec:Ingham} we prove a variation of the classical Ingham inequality for the reduced equation: here the fact that the bottom topography is not necessarily a small perturbation of the flat bottom turns out to be important, since we need to study the eigenvalues of the operator $A_{0,\beta}$. In particular, in order to get a lower bound for solutions supported on low-frequency modes we prove that for any $N>0$ sufficiently large the eigenvalues $(\mu(n))_{|n| \leq N}$ of the operator $A_{0,\beta}$ are distinct, for ``many" choices of bottom topographies (see Lemma \ref{lem:LowFreq}). Next, in Sec. \ref{sec:Obs} we prove the observability property for the reduced system, while in Sec. \ref{sec:Control} we prove the controllability of the reduced system. In Sec. \ref{sec:NashMoser} we prove Theorem \ref{thm:main}. In Sec. \ref{sec:constvort} we prove Theorem \ref{thm:mainConstVort}, describing the main differences with respect to the proof of Theorem \ref{thm:main}. Finally, in Appendix \ref{sec:Paradiff} we recall some standard rules of paradifferential calculus, and in Appendix \ref{sec:NM} we recall an abstract Nash--Moser theorem proved in \cite{baldi2017nash}.

\subsection{The irrotational model with general bottom topography} \label{subsec:model}

In this subsection we consider small-amplitude two-dimensional gravity-capillary waves on the free surface of a perfect, incompressible, inviscid fluid with constant density equal to $1$. We assume that the flow is irrotational. We assume that the fluid fills a domain with variable bottom topography and with space-periodic boundary conditions; writing $\mathbb{T} := \mathbb{R}/(2\pi\mathbb{Z})$, we denote the fluid domain in the following way,
\begin{align} \label{eq:FluidDom}
D_{\eta,\beta}(t) &:= \{ (x,y) \in \mathbb{T} \times \mathbb{R} : -h + \beta(x) < y < \eta(t,x) \},
\end{align}
where $\beta: \mathbb{T} \to \mathbb{R}$ denotes the variation with respect to the flat bottom of finite depth $h>0$, and where the free surface is given by the graph of an unknown function $\eta:\mathbb{R} \times \mathbb{T} \to \mathbb{R}$. We study the controllability problem for this model, namely the generation of small amplitude water waves through the application of an external pressure $\mathscr{P}_{ext}$ on a portion of the free surface. \\

The equations describing the flow are given by
\begin{alignat}{2}
\Div \bmu &= 0, & & \In D_{\eta,\beta}(t), \label{eq:Incompr} \\
\bmu_t + (\bmu \cdot \nabla)  \bmu &= - \nabla \mathscr{P} - g \bme_2, & & \In D_{\eta,\beta}(t), \label{eq:Euler} \\
\partial_x u_{2} - \partial_y u_1 &= 0, & & \In D_{\eta,\beta}(t), \label{eq:Irr} 
\end{alignat}
\begin{alignat}{2}
\bmu \cdot \bmN_{b} &= 0 , & & \at y=-h+\beta, \label{eq:Imperm} \\
\eta_t - \bmu \cdot \bmN &= 0, & & \at y=\eta, \label{eq:KinFree} \\
\mathscr{P} - \mathscr{P}_{ext} &=  - \kappa \, \left(  \frac{\eta_x}{ \sqrt{1+\eta_x^2}} \right)_x , & & \at y=\eta, \label{eq:DynBC}
\end{alignat}
where $\bmu = 
\begin{pmatrix}
u_1 \\ u_2
\end{pmatrix}
: \overline{D_{\eta,\beta}(t)} \to \mathbb{R}^2$ is the velocity field, $\bme_2=(0,1)^T$, $\bmN=(-\eta_x,1)^T$ denotes the outward normal vector at the free surface, $\bmN_{b}=(\beta_x,-1)^T$ denotes the outward normal vector at the bottom, $g>0$ is the gravity constant, $\mathscr{P}_{ext}$ is an external source term, and $\kappa>0$ is the surface tension coefficient.  The divergence operator used above is defined as $\Div \bmu := \nabla \cdot \bmu$. 

Whenever it is not important, we omit the time-dependence of $D_{\eta,\beta}(t)$ and we will denote it by $D_{\eta,\beta}$. We also denote by $D_0$ the fixed domain $\mathbb{T} \times (-h,0)$.

Since the fluid is incompressible and since the flow is irrotational, the velocity field can be expressed as the gradient of a function $\varphi$, called velocity potential. Defining
\begin{align*}
\psi(t,x) &:= \varphi(t,x,\eta(t,x))
\end{align*}
as the trace of the velocity potential at the free surface, we recover $\varphi$ by solving the following elliptic boundary-value problem,
\begin{equation} \label{eq:EllBVP}
\begin{cases}
\Delta \varphi = 0 & \In D_{\eta,\beta}, \\
\varphi = \psi & \at y=\eta, \\
- \varphi_x \, \beta_x + \varphi_y = 0  &\at y=-h+\beta .
\end{cases}
\end{equation}

For $s \in \mathbb{R}$, we define the Sobolev space 
\begin{align*}
H^s(\mathbb{T}) &\coloneqq \left\{ f(x) = \sum_{j \in \mathbb{Z}} f_j e^{\mathrm{i} jx} : \|f\|_{H^s(\mathbb{T})} \coloneqq \left( \sum_{j \in \mathbb{Z}} |f_j|^2 \langle j \rangle^{2s} \right)^{1/2} < +\infty \right\}, \\
\langle j \rangle &\coloneqq (1+j^2)^{1/2} ,
\end{align*}
and the corresponding Sobolev space of functions with zero average
\begin{align*}
H^s_0(\mathbb{T}) &:= \left\{ f \in H^s(\mathbb{T}) : \int_{\mathbb{T}} f(x) \di x = 0 \right\},
\end{align*}
with the associated norm
\begin{align*}
\|f\|_{H^s_0(\mathbb{T})} &\coloneqq \left( \sum_{j \in \mathbb{Z}\setminus \{0\} } |f_j|^2 \langle j \rangle^{2s} \right)^{1/2} , \;\; \forall f \in H^s_0(\mathbb{T}),
\end{align*}
while we denote by $\dot{H}^s(\mathbb{T})$ the homogeneous Sobolev space of order $s$, namely the quotient space $H^s(\mathbb{T})/\mathbb{R}$; for simplicity we denote the equivalent classes $[f] = \{ f+ c, c \in \mathbb{R} \}$ just by $f$.

Given $s \in \mathbb{R}$ and $\eta, \beta , \psi \in H^s(\mathbb{T})$, we introduce the \emph{Dirichlet--Neumann operator} $G(\eta,\beta)$ given by 
\begin{align}
G(\eta,\beta)\psi &\coloneqq \nabla\varphi \cdot \bmN |_{y=\eta(x)}, \label{eq:DNO} 
\end{align}
where $\varphi$ is the solution of the boundary-value problem \eqref{eq:EllBVP}.
We defer to Sec.2 and Appendix A in \cite{craig2005hamiltonian} and to Chap. 3 of \cite{lannes2013water} for more detailed studies of the Dirichlet--Neumann operator.

In the case of a general bottom topography, and allowing for the presence of an external pressure, the system \eqref{eq:Incompr}-\eqref{eq:DynBC} can be rewritten as
\begin{equation} \label{eq:WWsys}
\begin{cases}
\eta_t &= G(\eta,\beta)\psi  \\
\psi_t &= - \frac{1}{2} \, \psi_x^2 + \frac{ 1 }{ 2(1+\eta_x^2)}   \left[ G(\eta,\beta)\psi  + \eta_x \psi_x \right]^2  - g \, \eta + \kappa \, \left(  \frac{\eta_x}{ \sqrt{1+\eta_x^2}} \right)_x + \mathscr{P}_{ext} .
\end{cases}
\end{equation}

We briefly discuss the phase space of \eqref{eq:WWsys} when $\mathscr{P}_{ext} \equiv 0$. Recalling the boundary value problem \eqref{eq:EllBVP} (more precisely, the interior equation and the condition at the bottom) and using the divergence theorem, we find that also $G(\eta,\beta)\psi$ has zero average, which  implies that $\int_{\mathbb{T}} \eta(x) \mathrm{d}x$ is a constant of motion for \eqref{eq:WWsys}. Hence, we assume that the variables $(\eta,\psi)$ in \eqref{eq:WWsys} belong to the space $H^1_0(\mathbb{T}) \times \dot H^1(\mathbb{T})$.

Finally, if we denote by $D_{\eta,\beta}$ the domain
\begin{align*} 
D_{\eta,\beta} &= \{ (x,y) \in \mathbb{T} \times \mathbb{R}: -h+\beta(x) < y < \eta(x) \} ,
\end{align*}
we say that $D_{\eta,\beta}$ is \emph{strictly connected} if there exists $h_0>0$ such that
\begin{align} \label{eq:StrConnected}
h - \beta(x) + \eta(x) &\geq h_0, \; \; \forall \; x \in \mathbb{T} .
\end{align}

\subsection{The case with constant vorticity and flat bottom} \label{subsec:modelGamma}

In this subsection we consider two-dimensional gravity-capillary waves on the free surface of a perfect, incompressible and inviscid fluid with constant density equal to $1$ and with \emph{constant vorticity} $\gamma$. We assume that the fluid fills a domain with \emph{finite depth} and \emph{flat bottom topography}, with space-periodic boundary conditions; writing $\mathbb{T} := \mathbb{R}/(2\pi\mathbb{Z})$, we denote the fluid domain in the following way,
\begin{align} \label{eq:FluidDomF}
	D_{\eta}(t) &:= \{ (x,y) \in \mathbb{T} \times \mathbb{R} : -h < y < \eta(t,x) \},
\end{align}
where $h>0$, and where the free surface is given by the graph of an unknown function $\eta:\mathbb{R} \times \mathbb{T} \to \mathbb{R}$. \\

The equations describing the flow are given by
\begin{alignat}{2}
	\Div \bmu &= 0, & & \In D_{\eta}(t), \label{eq:IncomprV} \\
	\bmu_t + (\bmu \cdot \nabla)  \bmu &= - \nabla \mathscr{P} - g \bme_2, & & \In D_{\eta}(t), \label{eq:EulerV} \\
	\partial_x u_{2} - \partial_y u_1 &= \gamma, & & \In D_{\eta}(t), \label{eq:ConstVort} \\
	\bmu \cdot \bme_2 &= 0 , & & \at y=-h, \label{eq:ImpermV} \\
	\eta_t - \bmu \cdot \bmN &= 0, & & \at y=\eta, \label{eq:KinFreeV} \\
	\mathscr{P} - \mathscr{P}_{ext} &=  - \kappa \, \left(  \frac{\eta_x}{ \sqrt{1+\eta_x^2}} \right)_x , & & \at y=\eta, \label{eq:DynBCV}
\end{alignat}
where $\bmu = 
\begin{pmatrix}
	u_1 \\ u_2
\end{pmatrix}
: \overline{D_{\eta}(t)} \to \mathbb{R}^2$ is the velocity field, $\bme_2=(0,1)^T$, $\bmN=(-\eta_x,1)^T$ denotes the outward normal vector at the free surface, $g>0$ is the gravity constant, $\gamma \in \mathbb{R}$ is the value of the vorticity (assumed to be constant), $\mathscr{P}_{ext}$ is an external source term, and $\kappa>0$ is the surface tension coefficient.  The divergence operator used above is defined as $\Div \bmu := \nabla \cdot \bmu$. 

Due to the constant value of the vorticity, the velocity field $\bmu$ is given by the sum of the Couette flow 
$ \begin{pmatrix}
	-\gamma y \\ 0
\end{pmatrix} $
and an irrotational velocity field, that can be expressed as the gradient of a function $\varphi$, called velocity potential. Denoting by
\begin{align*}
	\psi(t,x) &:= \varphi(t,x,\eta(t,x))
\end{align*}
the evaluation of the velocity potential at the free surface, we can recover $\varphi$ by solving the following elliptic boundary-value problem,
\begin{equation} \label{eq:EllBVPF}
	\begin{cases}
		\Delta \varphi = 0 & \In D_{\eta}, \\
		\varphi = \psi & \at y=\eta, \\
		\varphi_y = 0  & \at y=-h .
	\end{cases}
\end{equation}

Using a formulation introduced by Constantin et al. in \cite{constantin2007nearly} for the system \eqref{eq:IncomprV}-\eqref{eq:DynBCV}, the time evolution of the fluid can be described by 
\begin{equation} \label{eq:WWsysV}
	\begin{cases}
		\eta_t = G^{DN}(\eta)\psi + \gamma \, \eta \, \eta_x \\
		\psi_t = - g \, \eta - \frac{1}{2} \, \psi_x^2 + \frac{ ( G^{DN}(\eta)\psi + \eta_x \psi_x )^2 }{ 2(1+\eta_x^2)} + \kappa \, \left(  \frac{\eta_x}{ \sqrt{1+\eta_x^2}} \right)_x  + \gamma \, \left( \eta \, \psi_x +  \partial_x^{-1}G^{DN}(\eta)\psi \right) \\
		\qquad \quad + \mathscr{P}_{ext} ,
	\end{cases}
\end{equation}
where $G^{DN}(\eta)$ is the classical \emph{Dirichlet--Neumann operator}
\begin{align} \label{eq:DNOF}
	G^{DN}(\eta)\psi &:= \nabla\varphi \cdot \bmN |_{y=\eta(x)}.
\end{align}

\begin{remark} \label{rem:ZCS}
	
	For $\gamma=0$ the system \eqref{eq:WWsysV} reduces to the classical Zakharov--Craig--Sulem formulation for irrotational fluids, namely to \eqref{eq:WWsys} with $\beta \equiv 0$. Moreover, recalling the Dirichlet--Neumann operator defined in \eqref{eq:DNO}, we have that  $G^{DN}(\eta)=G(\eta,0)$.
	
\end{remark}

We briefly discuss some properties of \eqref{eq:WWsysV} in the case $\mathscr{P}_{ext} \equiv 0$:
since the bottom of the fluid domain is flat, the system \eqref{eq:WWsysV} is invariant by space translation; notice that since $G^{DN}(\eta)\psi$ has zero average, then the quantity $\int_{\mathbb{T}} \eta(x) \mathrm{d}x$ is a constant of motion of \eqref{eq:WWsysV}; moreover, since $G^{DN}(\eta)1 =0$, then the vector field on the right-hand side of \eqref{eq:WWsysV} depends only on $\eta$ and on $\psi - \int_{\mathbb{T}} \psi \frac{ \di x }{2\pi}$. Finally, the system \eqref{eq:WWsysV} with $\mathscr{P}_{ext} \equiv 0$ is time reversible, and it also admits a Hamiltonian structure (see \cite{wahlen2007hamiltonian} and Sec. 2 of \cite{berti2021traveling}).

%We assume that the variables $(\eta,\psi)$ in \eqref{eq:WWsysV} belong to the space $H^1_0(\mathbb{T}) \times \dot{H}^1(\mathbb{T})$.

%Let us consider the space $H^1_0(\mathbb{T}) \times \dot{H}^1(\mathbb{T})$, endowed with the non canonical Poisson tensor
%\begin{align} \label{eq:NonCPoisson}
%	J_{\gamma} &:=
%	\begin{pmatrix}
%		0 & \mathrm{Id} \\ - \mathrm{Id} & \gamma \, \partial_x^{-1}
%	\end{pmatrix}
%	,
%\end{align}
%and let us consider the Hamiltonian
%\begin{align} \label{eq:HamF}
%	H_{0,\gamma}(\eta,\psi) &= \frac{1}{2} \int_{\mathbb{T}} \psi \, G^{DN}(\eta)\psi + g \, \eta^2  + \gamma \, \left( -\psi_x \, \eta^2 + \frac{\gamma}{3} \, \eta^3 \right) \di x + \kappa \, \int_{\mathbb{T}} \sqrt{1+\eta_x^2} \, \di x ,
%\end{align}
%which is well defined on $H^1_0(\mathbb{T}) \times \dot{H}^1(\mathbb{T})$. The Hamilton's equations associated to the Hamiltonian \eqref{eq:HamF} with respect to the Poisson tensor \eqref{eq:NonCPoisson}, namely
%\begin{align} \label{eq:HamEqNC}
%	\pd_t \, 
%	\begin{pmatrix}
%		\eta \\ \psi
%	\end{pmatrix}
%	&=  J_{\gamma}
%	\begin{pmatrix}
%		\nabla_{\eta} H_{0,\gamma} \\ \nabla_{\psi} H_{0,\gamma}
%	\end{pmatrix}
%	,
%\end{align}
%where $(\nabla_{\eta} H_{0,\gamma} , \nabla_{\psi} H_{0,\gamma} ) \in \dot{L}^2(\mathbb{T}) \times L^2_0(\mathbb{T})$ denote the $L^2$-gradients.

%We mention that Wahl\'en introduced another Hamiltonian formulation, in which the Poisson tensor takes the canonical form (see \cite{wahlen2007hamiltonian} and Sec. 2 of \cite{berti2021traveling}).

\section{Main results} \label{sec:main}

We now state the main result of the paper for \eqref{eq:WWsys}. In the following we denote by $B_{H^m}(R)$ the open ball of radius $R>0$ in $H^m(\mathbb{T})$.

\begin{theorem} \label{thm:main}

Let $T,R>0$, and let $\Omega$ be a non-empty open subset of $\mathbb{T}$. There exists $s_0 >0$ such that for any $s \geq s_0$, there exists an open and dense subset $\mathcal{G}$ of $B_{H^{s+1/2}}(R) \subset H^{s+1/2}(\mathbb{T})$ such that the following holds true: for any $\beta \in \mathcal{G}$ such that there exists $h_0>0$ for which
\begin{equation} \label{eq:StrConnAss}
-h+\beta(x) \leq - 2 h_0, \quad \forall x \in \mathbb{T},
\end{equation}
there exists $\epsilon_0>0$ such that for any positive $\epsilon \leq \epsilon_0$,
and for any $(\eta_0,\psi_0)$, $(\eta_f,\psi_f) \in H^{s+1/2}_0(\mathbb{T}) \times H^s(\mathbb{T})$ such that
\begin{align*}
\| (\eta_0,\psi_0) \|_{ H^{s+1/2}_0(\mathbb{T}) \times H^s(\mathbb{T}) } \leq \epsilon , &\quad 
\| (\eta_f,\psi_f) \|_{ H^{s+1/2}_0(\mathbb{T}) \times H^s(\mathbb{T}) } \leq \epsilon , 
\end{align*}
%and such that both $(\eta_0,\beta)$ and $(\eta_f,\beta)$ satisfy \eqref{eq:StrConnected}, 
there exists 
\begin{equation*}
\mathscr{P}_{ext} \in C( [0,T], H^s(\mathbb{T}) )
\end{equation*}
with
\begin{equation*}
\mathrm{supp}( \mathscr{P}_{ext}(t,\cdot) ) \subset \Omega, \quad \forall t \in [0,T] ,
\end{equation*}
such that the system \eqref{eq:WWsys} with initial datum $(\eta_0,\psi_0)$ admits a unique solution 
\begin{equation*}
(\eta,\psi) \in C([0,T], H^{s+1/2}_0(\mathbb{T}) \times H^s(\mathbb{T}) ) ,
\end{equation*}
where $( \eta(T,\cdot),\psi(T,\cdot) ) = (\eta_f,\psi_f)$.

\end{theorem}

The above Theorem \ref{thm:main} is proved in Secs. \ref{sec:DNO}--\ref{sec:NashMoser}.

\begin{remark} \label{rem:BottomSet}
	
	Our Theorem \ref{thm:main} extends Theorem 1.1 of \cite{alazard2018control} to the case in which the fluid domain has a fixed non-constant bottom topography $\beta$, for ``many" choices of $\beta$ (in a topological sense). The main obstacle to proving a controllability result for all $\beta$ with a sufficiently high regularity comes from Ingham-type estimates and observability estimates, see Secs. \ref{sec:Ingham}-\ref{sec:Obs}.
	
	Moreover, we point out that the smallness threshold $\epsilon_0$ depends on the precise choice of $\beta$; however, if we fix $\beta^{\ast} \in \mathcal{G}$, there exists a neighbourhood $\mathcal{U}_{\beta^{\ast}} \subset H^{s+1/2}(\mathbb{T})$ such that $\epsilon_0$ can be chosen uniformly for all $\beta \in \mathcal{U}_{\beta^{\ast}} \cap \mathcal{G}$ which satisfy \eqref{eq:StrConnAss}.
\end{remark}

\begin{remark} \label{rem:Cavitation}
	
	If $(\eta,\psi)$ is a sufficiently small solution of \eqref{eq:WWsys} and if $\beta$ satisfies \eqref{eq:StrConnAss}, then the domain $D_{\eta,\beta}$ is strictly connected, namely condition \eqref{eq:StrConnected} holds true (see Theorem 9.6 in \cite{lannes2013water}).
	
\end{remark}

The controllability result that we can prove for \eqref{eq:WWsysV} is the following.

\begin{theorem} \label{thm:mainConstVort}
	
	Let $T>0$, and let $\Omega$ be a non-empty open subset of $\mathbb{T}$. There exist $s_0 >0$ and $\epsilon_0>0$ such that the following holds true: for any $s \geq s_0$, for any positive $\epsilon \leq \epsilon_0$ and for any $(\eta_0,\psi_0)$, $(\eta_f,\psi_f) \in H^{s+1/2}_0(\mathbb{T}) \times H^s(\mathbb{T})$ such that
	\begin{align*}
		\| (\eta_0,\psi_0) \|_{ H^{s+1/2}_0(\mathbb{T}) \times H^s(\mathbb{T}) } \leq \epsilon , &\quad 
		\| (\eta_f,\psi_f) \|_{ H^{s+1/2}_0(\mathbb{T}) \times H^s(\mathbb{T}) } \leq \epsilon , 
	\end{align*}
	there exists 
	\begin{equation*}
		\mathscr{P}_{ext} \in C( [0,T], H^s(\mathbb{T}) )
	\end{equation*}
	with
	\begin{equation*}
		\mathrm{supp}( \mathscr{P}_{ext}(t,\cdot) ) \subset \Omega, \quad \forall t \in [0,T] ,
	\end{equation*}
	such that the system \eqref{eq:WWsysV} with initial datum $(\eta_0,\psi_0)$ admits a unique solution 
	\begin{equation*}
		(\eta,\psi) \in C([0,T], H^{s+1/2}_0(\mathbb{T}) \times H^s(\mathbb{T}) ) ,
	\end{equation*}
	where $( \eta(T,\cdot),\psi(T,\cdot) ) = (\eta_f,\psi_f)$.
	
\end{theorem}

We defer the proof of Theorem \ref{thm:mainConstVort} to Sec. \ref{sec:constvort}.

\begin{remark} \label{rem:CavitationFlat}
	
	If $(\eta,\psi)$ is a sufficiently small solution of \eqref{eq:WWsysV}, then the domain $D_{\eta}$ is strictly connected, namely there exists $h_0>0$ such that
	\begin{equation} \label{eq:StrConnectedFlat}
		-h \leq \eta(x)-h_0, \quad \forall x \in \mathbb{T},
	\end{equation}
	holds true.
	
\end{remark}

\section{The Dirichlet--Neumann operator} \label{sec:DNO}

In this section we recall some properties of the operator $G(\eta,\beta)$ defined in \eqref{eq:DNO}, namely
\begin{align*}
G(\eta,\beta)\psi &= \nabla\varphi \cdot \bmN |_{y=\eta(x)},
\end{align*}
where $\eta$ and $\beta$ satisfy \eqref{eq:StrConnected}, and where $\varphi$ is the solution of the boundary value problem \eqref{eq:EllBVP}, namely
\begin{equation*} 
\begin{cases}
\Delta \varphi = 0 & \In D_{\eta,\beta}, \\
\varphi = \psi & \at y=\eta, \\
- \varphi_x \, \beta_x + \varphi_y = 0  & \at y=-h+\beta .
\end{cases}
\end{equation*}

We refer the interested reader to \cite{craig1993numerical,craig2005hamiltonian} and to Chap. 3 of \cite{lannes2013water} for more details. In the following we use the notation $D \coloneqq -\mathrm{i}\partial_{x}$.

\subsection{Straightening Diffeomorphisms} \label{subsec:straight}

Observe that systems \eqref{eq:Incompr}-\eqref{eq:DynBC} and \eqref{eq:EllBVP} involve functions defined on the time-dependent fluid domain $D_{\eta,\beta}(t)$. In order to deal with this difficulty we introduce the following straightening diffeomorphism (also called admissible diffeomorphism in the literature, see Definition 2.13 in \cite{lannes2013water}) ,
\begin{align}
	\Sigma(t) &: D_0 \to D_{\eta,\beta}(t) , \nonumber \\
	\Sigma(t)(x,w) &:= ( x, w+ \sigma(t,x,w) ), \nonumber \\
	\sigma(t,x,-h) = \beta(x)\in H^{s}(\mathbb{T}) , &\; \; \sigma(t,x,0)=\eta(t,x)\in H^{s}(\mathbb{T}), \; \; s > 3/2,  \label{eq:straight} 
\end{align}
where $\eta$ and $\beta$ satisfy condition \eqref{eq:StrConnected}. We require the coefficients of the Jacobian matrix $J_{\Sigma} = \nabla_{x,w}\Sigma$ to belong to $L^\infty(D_0)$, and assume that there exists a constant 
\begin{align*}
	M_0 &= M_0( h_0^{-1}, \|\eta\|_{ H^{s}(\mathbb{T})} , \|\beta\|_{ H^{s}(\mathbb{T})},  \|\eta\|_{ C^{1}(\mathbb{T})} , \|\beta\|_{ C^{1}(\mathbb{T})} )>0
\end{align*}
such that
\begin{align*}
	\| J_{\Sigma,i,j} \|_{L^\infty(D_0)} &\leq M_0 , \; \; \forall i,j=1,2,
\end{align*}
and such that the determinant of the Jacobian matrix $J_{\Sigma}$ is uniformly bounded from below on $D_0$ by a constant $c_0>0$ satisfying $M_0 \geq 1/c_0$. 

We now make some remarks about straightening diffeomorphisms, and we discuss two common examples of them.

\begin{remark} \label{rem:Straight}
	
	From the above definition of the straightening diffeomorphism, and in particular from the fact that the determinant of the Jacobian matrix $J_{\Sigma}$ is bounded from below, we deduce that the function $\sigma$ satisfies
	\begin{align*}
		1 + \partial_w \sigma(t) &\neq 0, \; \; \text{in} \; \; D_0.
	\end{align*}
\end{remark}

\begin{remark} \label{rem:TrivialDiffeo}
	
	A so-called \emph{trivial diffeomorphism} between $D_0$ and $D_{\eta,\beta}(t)$ is given by choosing
	\begin{align} \label{eq:TrivialDiffeo}
		\sigma(t,x,w) &:= \left( 1 + \frac{w}{h} \right) \eta(t,x) - \frac{w}{h} \beta(x) 
	\end{align}
	in the formula \eqref{eq:straight} (see also Remark 2.4 in  \cite{lannes2005well}).
	
\end{remark}

We now introduce another example of straightening diffeomorphism, the so-called \emph{regularizing diffeomorphism} (see Proposition 2.16 and Proposition 2.18 of \cite{lannes2013water}; see also Proposition 2.13 of \cite{lannes2005well}). We recall that $D=-\mathrm{i}\partial_x$, and we write $\langle D \rangle \coloneqq (1-\partial_{xx})^{1/2}$; we also define, for any $s \in \mathbb{R}$ and $k\in \mathbb{N}$
\begin{align*}
	H^{s,k}(D_0) &\coloneqq \bigcap_{j=0}^k H^{j}( (-h,0) ;H^{s-j}(\mathbb{T}) ), \;\; \| f \|_{ H^{s,k}(D_0) } \coloneqq \sum_{j=0}^k \| \langle D \rangle \, \partial_w^j f \|_{ L^2(\mathbb{T}) } .
\end{align*}

\begin{proposition} \label{prop:RegDiffeo}
	
	Let $s>3/2$ and assume that $\eta \in H^{s}(\mathbb{T})$ and $\beta \in H^{s}(\mathbb{T})$ satisfy condition \eqref{eq:StrConnected}. Let $\chi \in C^{\infty}_c(\mathbb{R})$ be a positive function equal to $1$ in a neighborhood of the origin, and consider the diffeomorphism \eqref{eq:straight}, where
	\begin{align}
		\sigma(t,x,w) &\coloneqq \left( 1 + \frac{w}{h} \right) \eta^{(\delta)}(t,x,w) - \frac{w}{h} \beta_{(\delta)}(x) ,  \label{eq:straightReg} \\
		\eta^{(\delta)}(t,x,w) &\coloneqq \chi(\delta \, w \, |D|)\eta(t,x)  , \;\; \beta_{(\delta)}(x,w) \coloneqq \chi(\delta \, (w+h) \, |D|)\beta(x) . \nonumber
	\end{align}
	Then, if $\delta >0$ in \eqref{eq:straightReg} satisfies
	\begin{align*}
		\delta &< \frac{ h_0 }{C(\chi) \left[ \|\eta\|_{ H^{s}(\mathbb{T})} + \|\beta\|_{ H^{s}(\mathbb{T})} \right] } , \nonumber \\
		C(\chi) &:= \|\chi'\|_{L^\infty(\mathbb{R})} \, \left[ \int_{\mathbb{R}^2} (1+|\xi|^2)^{-(s-1)} \, \mathrm{d}\xi \right]^{1/2} < +\infty ,
	\end{align*}
	then the diffeomorphism defined in \eqref{eq:straightReg} is such that all entries of $J_{\Sigma} = \nabla_{x,w}\Sigma$ belongs to $( L^\infty(D_0) )^0$, and there exists $M_0>0$ such that
	\begin{align*}
		\| J_{\Sigma,i,j} \|_{L^\infty(D_0)} &\leq M_0 , \; \; \forall i,j=1,2.
	\end{align*}
	Moreover, the determinant of the Jacobian matrix $J_{\Sigma}$ is uniformly bounded from below on $D_0$ so that $M_0 \geq 1/c_0$, where $c_0 = h_0 - \delta \, C(\chi) \, \left[ \|\eta\|_{ H^{s}(\mathbb{T}) } + \|\beta\|_{ H^{s}(\mathbb{T})} \right]$.
	
\end{proposition}

We also mention the following smoothing property (see Lemma 2.20 in \cite{lannes2013water}).

\begin{lemma} \label{lem:smoothing}
	
	Let $s \in \mathbb{R}$ and $\lambda_1>0$, then for all $\chi \in C^\infty_c(\mathbb{R})$ there exists a constant $C=C(\chi)>0$ such that
	\begin{align} \label{eq:smoothing}
		\left\| \chi(\lambda_1^{1/2} w |D|) f \right\|_{H^s(D_0)}^2 &\leq C \, \left\| (1+\lambda_1^{1/2} |D|)^{-1/2} f \right\|_{H^s(\mathbb{T})}^2 .
	\end{align}
\end{lemma}

Observe that by Lemma \ref{lem:smoothing} (with $\lambda_1=\delta^2$) we have that 
\begin{align*}
	\| \eta^{(\delta)} \|_{H^{s+1/2,2}(D_0)} &\leq C(1/\delta) \|\eta\|_{H^{s}(\mathbb{T})} . 
\end{align*}
As a comparison, recall that for the trivial diffeomorphism of Remark \ref{rem:TrivialDiffeo} one needs to control $H^s(\mathbb{T})$-norm of the surface parameterization $\eta$ in order to control the $H^s(D_0)$-norm of $\sigma$ in \eqref{eq:TrivialDiffeo}. \\

For simplicity, in the following discussion we assume that $\eta$, $\beta$, the function $\sigma$ in \eqref{eq:straight}, the function $f: D_{\eta,\beta} \to \mathbb{R}$ and the vector field $\bmF: D_{\eta,\beta} \to \mathbb{R}^2$ are sufficiently smooth, so that the quantities we introduce below are well-defined.

We define $\tilde{\bmF}(x,w) := \bmF \circ \Sigma(x,w)$ and $\tilde{f}(x,w)=f \circ \Sigma(x,w)$; we also write
\begin{align*}
	\partial_i^\Sigma \tilde{\bmF} &:= ( \partial_i \bmF ) \circ \Sigma, \;\; i=t,x,w ,
\end{align*}
and similarly for $\partial_i^\Sigma \tilde{f}$. From \eqref{eq:straight} we obtain that
\begin{align*}
	\partial_j^\Sigma = \partial_j - \frac{ \partial_j \sigma }{ 1+\partial_w \sigma} \partial_w , \; \; j =x,t, &\; \; \partial_w^\Sigma = \frac{ \partial_w }{ 1+\partial_w \sigma } .
\end{align*}

\begin{remark} \label{rem:flatOp}
	
	Similarly, we can define
	\begin{align*}
		\Div^{\Sigma} \tilde{\bmF} = \nabla^\Sigma \cdot \tilde{\bmF} := (\Div \bmF) \circ \Sigma , &\;\; \Delta^\Sigma  \tilde{f} := (\Delta f) \circ \Sigma .
	\end{align*}
	
	One can check that the flattened version of the operators $\Div$ and $\Delta$ acting on $\tilde{\bmF}$ and to $\tilde{f}$ are given by
	\begin{align*}
		\Div^\Sigma \tilde{\bmF} &= \partial_x^\Sigma \tilde{F}_1 + \partial_w^\Sigma \tilde{F}_2 = \Div \tilde{\bmF} - \frac{1}{1+\partial_w \sigma} \, \partial_x \sigma \; \partial_w \tilde{F}_1  - \frac{ \partial_w \sigma }{ 1+\partial_w \sigma}  \; \partial_w \tilde{F}_2 , 
	\end{align*}
	\begin{align*}
		\Delta^\Sigma \tilde{f} &= (\partial_x^\Sigma)^2 \tilde{f}  + (\partial_w^\Sigma)^2 \tilde{f} = \left[ \mathtt{a} \, \partial_w^2 + \partial_x^2 + \mathtt{b} \, \partial_x \partial_w - \mathtt{c} \, \partial_w \right] \tilde{f},
	\end{align*}
	where
	\begin{align} 
		\mathtt{a} := \frac{ 1+ (\partial_{x} \sigma)^2 }{ (1+\partial_w \sigma)^2 } , \; \; \mathtt{b} &:= -2 \frac{ \partial_{x} \sigma }{ 1+\partial_w \sigma } , \;\; \mathtt{c} := \frac{1}{ 1+\partial_w \sigma } \left[  \partial_x^2 \sigma  + \mathtt{b} \,  \partial_{x} \, \partial_w \sigma  +  \mathtt{a} \,  \partial_w^2 \sigma \right] . \label{eq:CoeffFlatLap}
	\end{align}
	
	Moreover, if $J_\Sigma$ is the Jacobian matrix of $\Sigma$, we denote by $P(\Sigma)$ the symmetric matrix given by
	\begin{align}
		P(\Sigma) &:= |\det(J_\Sigma)| \, J_{\Sigma}^{-1}  J_\Sigma^{-T}  = (1+\partial_w \sigma) \,
		\begin{pmatrix}
			1 & \frac{- \partial_{x} \sigma }{ 1+ \partial_w \sigma} \\
			& \\
			\frac{- ( \partial_{x} \sigma )^T }{ 1+\partial_w \sigma} & \frac{ 1+ (\partial_{x} \sigma)^2 }{ (1+\partial_w \sigma)^2 } 
		\end{pmatrix}
		, \label{eq:PSigma} 
	\end{align}
	so that 
	\begin{align*}
		\Delta^{\Sigma} \tilde{f} &= (1 + \partial_w\sigma)^{-1} \, \nabla_{x,w} \cdot P(\Sigma)\nabla_{x,w} \tilde{f} ,
	\end{align*}
	see Lemma 2.5 in \cite{lannes2005well}.
	
\end{remark}

\subsection{Properties of the Dirichlet--Neumann operator} \label{subsec:propDNO}

First we state the following analyticity result (see Theorem A.11 in \cite{lannes2013water}):

\begin{proposition} \label{prop:AnalOpG}

Let $s > 3/2$ and $\psi \in H^s(\mathbb{T})$. Choose $\sigma$ of the form  \eqref{eq:straightReg} in the diffeomorphism \eqref{eq:straight}. Then the formula
\begin{align*}
G(\eta,\beta)\psi &= \nabla^{\Sigma}\tilde{\varphi}(\eta,\beta,\psi) \cdot \bmN |_{w=0} 
\end{align*}
defines an analytic map
\begin{align*}
G(\cdot,\cdot)\psi &: \{ (\eta,\beta) \in H^{s}(\mathbb{T}) \times H^{s}(\mathbb{T}) : \eqref{eq:StrConnected} \; \text{holds true} \} \to H^{s-1}(\mathbb{T}) 
\end{align*}

\end{proposition}

Recalling \eqref{eq:DNO}, we have that if $s > 3/2$, $\beta \in H^s(\mathbb{T})$, $(\eta,\psi) \in H^s(\mathbb{T}) \times H^s(\mathbb{T})$ and if \eqref{eq:StrConnected} holds true, then by Proposition \ref{prop:AnalOpG} and by Proposition 2.3 in \cite{lannes2013water} we can deduce the estimate
\begin{align} \label{eq:EstDNO}
\| G(\eta,\beta)\psi \|_{ H^{s-1}(\mathbb{T}) } &\leq C \left( \| \eta \|_{ H^s(\mathbb{T}) } , \| \beta \|_{ H^s(\mathbb{T}) }  \right) \, \| \psi \|_{ H^s(\mathbb{T}) } .
\end{align}

Now we consider the expansion of the operator $G(\eta,\beta)$ around $\eta=0$,
\begin{align*}
G(\eta,\beta) &= \sum_{j=0}^{\infty} G_j[\eta](\beta)
\end{align*}
where $G_j[\eta](\beta)$ is homogeneous of degree $j$ in $\eta$. By formula \emph{(A.1)} in \cite{craig2005hamiltonian}, we have
\begin{align}
 G_{0}(\beta)\psi &= ( D \, \tanh(h D) + D \, L(\beta) ) \psi , \label{eq:G0Irr}
\end{align}
where 
\begin{align} 
	L(\beta) &\coloneqq -L_{II}(\beta) L_{I}(\beta) , \label{eq:Lformula} \\
	L_{I}(\beta) &\coloneqq \mathrm{Op}\left( \sech(h \xi) \sinh( \beta \; \xi )  \right) , \nonumber \\
	L_{II}(\beta) &\coloneqq L_{III}(\beta)^{-1} , \quad L_{III}(\beta) \coloneqq \mathrm{Op}\left( \cosh((-h+\beta)\xi) \right) , \nonumber
\end{align}
see (2.14) in \cite{craig2005hamiltonian}.

We also mention that if we expand the operator $L(\beta)$ in \eqref{eq:G0Irr} around $\beta = 0$,
\begin{align*}
L(\beta) &= \sum_{k=1}^{\infty} L_k(\beta) 
\end{align*}
with $L_k(\beta)$ homogeneous of degree $k$ in $\beta$, we have that 
\begin{align*}
L_1 &= - \sech(h D) \beta \, \sech(h D) D,
\end{align*}
see Appendix A in \cite{craig2005hamiltonian}.

We also have the following results (see Lemma 3.3-Lemma 3.4 \cite{pasquali2026two}).

\begin{lemma} \label{lem:EstGU}

Let $s > 3/2$, $\beta \in H^s(\mathbb{T})$ and $(\eta,\psi) \in H^s(\mathbb{T}) \times H^s(\mathbb{T})$. Assume that \eqref{eq:StrConnected} holds true, then 
\begin{align}
B(\eta,\beta)\psi &= \frac{ G(\eta,\beta)\psi + \eta_x \psi_x }{1+\eta_x^2} \in H^{s-1}(\mathbb{T}) , \nonumber \\
V(\eta,\beta)\psi &= \psi_x - B(\eta,\beta)\psi \, \eta_x \in H^{s-1}(\mathbb{T}), \nonumber \\
\omega(\eta,\beta)\psi &= \psi - T_{B(\eta,\beta)\psi} \eta \in H^s(\mathbb{T}) , \label{eq:defGU}
\end{align}
and
\begin{align} 
& \| B(\eta,\beta)\psi \|_{ H^{s-1}(\mathbb{T}) } + \| V(\eta,\beta)\psi \|_{ H^{s-1}(\mathbb{T}) } + \| \omega(\eta,\beta)\psi \|_{H^s(\mathbb{T})} \nonumber \\
&\leq C \left( \| \eta \|_{ H^s(\mathbb{T}) } , \| \beta \|_{ H^s(\mathbb{T}) }  \right) \, \| \psi \|_{ H^s(\mathbb{T}) } . \label{eq:EstGUSob}
\end{align}

\end{lemma}

\begin{lemma} \label{lem:InvGU}

Let $s > 3/2$, $\eta,\beta \in H^s(\mathbb{T})$. Assume that \eqref{eq:StrConnected} holds true. Then there exists $\epsilon_0>0$ such that if 
\begin{align} \label{eq:SmallAss}
\| \eta \|_{ H^s(\mathbb{T}) }  &< \epsilon_0 ,
\end{align}
then there exists an operator $\Psi(\eta,\beta)$ such that 
\begin{align*}
\Psi(\eta,\beta) \left( \omega(\eta,\beta)\psi \right) &= \psi, \;\; \forall \psi \in H^{1/2}(\mathbb{T}) .
\end{align*}
Moreover, if $\omega \in H^s(\mathbb{T})$, then $\Psi(\eta,\beta)(\omega) \in H^s(\mathbb{T})$, and
\begin{align*}
\| \Psi(\eta,\beta)(\omega) \|_{H^s(\mathbb{T})} &\leq C \left( \| \eta \|_{ H^s(\mathbb{T}) } , \| \beta \|_{ H^s(\mathbb{T}) }  \right) \, \| \omega \|_{ H^s(\mathbb{T}) } .
\end{align*}

\end{lemma}

For the sake of simplicity, in the following we write $B$, $V$ and $\omega$ instead of $B(\eta,\beta)\psi$, $V(\eta,\beta)\psi$ and $\omega(\eta,\beta)\psi$. By Lemma \ref{lem:InvGU} we have that under the smallness assumption \eqref{eq:SmallAss}, then $B$ and $V$ can be rewritten in terms of $\eta$, $\beta$ and $\omega$,
\begin{align*}
B = B(\eta,\beta) \left( \Psi(\eta,\beta)\omega \right), &\;\; V = V(\eta,\beta) \left( \Psi(\eta,\beta)\omega \right) .
\end{align*}

Next, we show a paralinearization formula for the Dirichlet--Neumann operator.
 The paradifferential calculus was introduced by Bony in order to deal with the quantization of symbols $a(x,\xi)$ of degree $m$ with respect to $\xi$ and limited regularity in the $x$-variable, to which are associated operators of order $m$ denoted by $T_a$.

For any tempered distribution $u \in \mathcal{S}'(\mathbb{T})$ we denote its Fourier transform by $\mathscr{F}u$. 

\begin{definition} \label{def:symbolsPara}
	Let $\varrho \geq 0$ and $m \in \mathbb{R}$, we denote by $\Gamma^m_\varrho(\mathbb{T})$ the space of functions $a(x,\xi)$ on $\mathbb{T} \times \mathbb{R}$ which are of class $C^\infty$ with respect to $\xi$ and such that for all $\alpha \in \mathbb{N}$ and all $\xi$ the map $x \mapsto \pd_{\xi}^\alpha a(x,\xi)$ belongs to $W^{\varrho,\infty}(\mathbb{T})$ and there exists $C_\alpha >0$ such that
	\begin{align}
		\|\pd_{\xi}^\alpha a(\cdot,\xi)\|_{ W^{\varrho,\infty}(\mathbb{T}) } &\leq C_\alpha (1 + |\xi|)^{m-|\alpha|}. \label{eq:symbolsEst}
	\end{align}
\end{definition}

We now introduce the space of pluri-homogeneous symbols.

\begin{definition} \label{def:symbolsHomPara}
	Let $\varrho \geq 1$ and $m \in \mathbb{R}$, we denote by $\Sigma^m_\varrho(\mathbb{T})$ the space of symbols of the form
	\begin{align*}
		a(x,\xi) &= \sum_{0 \leq j < \varrho} a_{m-j}(x,\xi),
	\end{align*}
	where $a_{m-j} \in \Gamma^{m-j}_{\varrho-j}(\mathbb{T})$ is homogeneous of degree $m-j$ in the variable $\xi$ and is of class $C^\infty$ in the variable $\xi$, and with regularity $W^{\varrho -j,\infty}$ in the variable $x$. We say that $a_m$ is the principal symbol of $a$.
\end{definition}

Let $m \in \mathbb{R}$, $\varrho \in [0,1]$, $a \in \Gamma^m_{\varrho}(\mathbb{T})$, we also define the semi-norm
\begin{align} \label{eq:NormPara}
	M^m_{\varrho}(a) &:= \sup_{|\alpha| \leq 6+\varrho} \sup_{\xi \in \mathbb{R}} \| (1+|\xi|)^{\alpha-|m|} \, a(\cdot, \xi) \|_{W^{\varrho,\infty}(\mathbb{T})} .
\end{align}

In order to define rigorously paradifferential operators we first introduce a fixed cutoff function $\chi$. Fix sufficiently small $\varepsilon_1$, $\varepsilon_2$ such that $0 <  \varepsilon_1 < \varepsilon_2 $ and choose a function $\chi \in C^\infty$ homogeneous function of degree $0$ such that:
\begin{itemize}
	\item[i.] the following conditions are satisfied,
	\begin{align*}
		\chi(\xi_1 , \xi_2 ) &= 1, \; \; \text{if} \; \; |\xi_1 | \leq \varepsilon_1 |\xi_2 | , \\
		\chi(\xi_1 , \xi_2 ) &= 0, \; \; \text{if} \; \; |\xi_1 | \geq \varepsilon_2 |\xi_2 | ;
	\end{align*}
	\item[ii.] the following symmetry property holds true,
	\begin{align} \label{eq:chiSymm}
		\chi(\xi_1,\xi_2) &= \chi(-\xi_1,-\xi_2) = \chi(-\xi_1,\xi_2) .
	\end{align}
\end{itemize}

Given a symbol $a \in \Sigma^m_\varrho(\mathbb{T})$, we define the paradifferential operator $T_a$ by
\begin{align}
	\mathscr{F}(T_au)(\xi) &:= (2\pi)^{-1} \sum_{\eta \in \mathbb{Z}} \chi(\xi-\eta,\eta) \; \; \mathscr{F}a(\xi-\eta,\eta) \; \; \mathscr{F}u(\eta) , \label{eq:ParadiffOp}
\end{align}
where $\mathscr{F}a$ is the Fourier transform of $a$ with respect to the first variable.  

Combining Proposition 3.13 and Proposition 3.14 in \cite{alazard2011water} with Proposition 3.5 in \cite{pasquali2026two} we obtain

\begin{proposition} \label{prop:DNOpara}

Let $s > 5/2$. If $\eta, \beta \in H^{s+1/2}(\mathbb{T})$ are such that \eqref{eq:StrConnected} holds true and if $\psi \in H^s(\mathbb{T}) $, then  
\begin{align} \label{eq:DNOpara}
G(\eta,\beta)\psi &= T_{\lambda} \omega  - T_{ V(\eta,\beta)\psi } \eta_x + \mathcal{R}(\eta,\psi) ,
\end{align}
where the function $\omega$ defined in \eqref{eq:defGU} is Alinhac's good unknown, $\lambda$ is a symbol given by
\begin{align}
\lambda(x,\xi) &= |\xi|, \label{eq:lambdaOp} 
%\lambda^{(1)}(\xi) &= |\xi| , \nonumber \\
%\lambda^{(0)}(x,\xi) &= \frac{1+\eta_x^2}{2\lambda^{(1)}} \left[ ( \alpha_1 \eta_x )_x + \mathrm{i} ( \partial_\xi \lambda^{(1)} ) \, (\partial_x \alpha_1 ) \right] , \quad \alpha_1 = \frac{ \lambda^{(1)} + \mathrm{i} \xi \eta_x }{1+\eta_x^2} ,
\end{align}
and $\mathcal{R}(\eta,\psi) \in H^{s+1/2}(\mathbb{T})$ satisfies
\begin{align*}
\| \mathcal{R}(\eta,\psi) \|_{ H^{s+1/2}(\mathbb{T}) } &\leq C \left( \| \eta \|_{ H^{s+1/2}(\mathbb{T}) } , \| \beta \|_{ H^{s+1/2}(\mathbb{T}) }    \right) \, \| \psi \|_{ H^{s}(\mathbb{T}) } ,
\end{align*}
for some non-decreasing function $C$.
\end{proposition}

Finally, we prove the following corollary of Proposition \ref{prop:DNOpara}.

\begin{corollary} \label{cor:ExpOpG}

Let $s > 5/2$. If $\eta, \beta \in H^{s+1/2}(\mathbb{T})$ are such that \eqref{eq:StrConnected} holds true and if $\psi \in H^s(\mathbb{T}) $, then there exists $\theta \in (0,1]$ such that
\begin{align*}
G(\eta,\beta)\psi &= G_0(\beta)\omega - T_{ V(\eta,\beta)\psi } \eta_x + \mathcal{F}(\eta,\beta)\psi ,
\end{align*}
where
\begin{align*}
\| \mathcal{F}(\eta,\beta)\psi \|_{ H^{s+1/2}(\mathbb{T}) } &\leq C \left( \| \eta \|_{ H^{s+1/2}(\mathbb{T}) } , \| \beta \|_{ H^{s+1/2}(\mathbb{T}) }   \right) \, \| \eta \|_{ H^{s+1/2}(\mathbb{T}) }^{\theta} \,\| \psi \|_{ H^{s}(\mathbb{T}) } .
\end{align*}
\end{corollary}

\begin{proof}

Recall that by \eqref{eq:DNOpara} we have
\begin{align*}
G(\eta,\beta)\psi &= T_{\lambda} \omega  - T_{ V(\eta,\beta)\psi } \eta_x + \mathcal{R}(\eta,\psi) ,
\end{align*}
where
\begin{align*}
\| \mathcal{R}(\eta,\psi) \|_{ H^{s+1/2}(\mathbb{T}) } &\leq C \left( \| \eta \|_{ H^{s+1/2}(\mathbb{T}) } , \| \beta \|_{ H^{s+1/2}(\mathbb{T}) }  \right) \, \| \psi \|_{ H^{s}(\mathbb{T}) } .
\end{align*}

Moreover, recall that by \eqref{eq:DNOpara} we have that 
\begin{align*}
\| T_{ \lambda^{(1)} }\omega - G_0(\beta)\omega \|_{ H^{s+1}(\mathbb{T}) } &\leq C \left( \|\eta\|_{ H^{s+1/2}(\mathbb{T}) } , \|\beta\|_{ H^{s+1/2}(\mathbb{T}) }  \right) \| \omega \|_{H^s(\mathbb{T})} ,
\end{align*}
while by \eqref{eq:EstGUSob}
\begin{align*} 
\| \omega \|_{H^s(\mathbb{T})} &\leq C \left( \| \eta \|_{ H^s(\mathbb{T}) } , \| \beta \|_{ H^s(\mathbb{T}) }  \right) \, \| \psi \|_{ H^s(\mathbb{T}) } . 
\end{align*}
Moreover, observe that by \eqref{eq:EstGUSob}
\begin{align*}
& \| \omega - \psi \|_{ H^s(\mathbb{T}) } + \| T_{ V(\eta,\beta)\psi } \eta_x \|_{ H^{s-1}(\mathbb{T}) } \\
&\leq \left[ \| B(\eta,\beta)\psi \|_{ L^{\infty}(\mathbb{T}) } + \|  V(\eta,\beta)\psi \|_{ L^{\infty}(\mathbb{T}) } \right] \, \| \eta \|_{ H^s(\mathbb{T}) } \\
&\leq C \left( \| \eta \|_{ H^s(\mathbb{T}) } , \| \beta \|_{ H^s(\mathbb{T}) }  \right) \, \| \psi \|_{ H^s(\mathbb{T}) } \, \| \eta \|_{ H^s(\mathbb{T}) } .
\end{align*}

We can deduce the thesis by arguing as in the proof of Proposition 2.4 of \cite{alazard2018control}.
\end{proof}

\section{Symmetrization of the water waves system} \label{sec:symm}

In this section we consider the system \eqref{eq:WWsys}, namely

\begin{equation*} 
\begin{cases}
\eta_t &= G(\eta,\beta)\psi \\
\psi_t &= - \frac{1}{2} \, \psi_x^2 + \frac{ 1 }{ 2(1+\eta_x^2)}   ( G(\eta,\beta)\psi + \eta_x \psi_x )^2 - g \, \eta + \kappa \, \left(  \frac{\eta_x}{ \sqrt{1+\eta_x^2}} \right)_x + \mathscr{P}_{ext}   ,
\end{cases}
\end{equation*}

and we use paradifferential calculus in order to rewrite the above system as a single wave-type equation in a suitable unknown $u$.

Consider the linearization of the system \ref{eq:WWsys}, namely
\begin{equation*}
\begin{cases}
\eta_t &= G_0(\beta)\psi , \\
\psi_t &= -g \, \eta + \kappa \eta_{xx} + \mathscr{P}_{ext}
\end{cases}
\end{equation*}
where 
\begin{align*}
G_0(\beta)\psi &= ( D \, \tanh(h D) + D \, L(\beta) ) \psi  ,
\end{align*}
see \eqref{eq:G0Irr}. Recall that by Proposition \ref{prop:DNOpara} we have 
\begin{align*}
G_0(\beta)\psi &= T_{\lambda}\psi + \mathcal{R}\psi , \quad \forall \beta \in H^{s+1/2}(\mathbb{T}), \;\; \forall \psi \in H^s(\mathbb{T}), 
\end{align*}
where
\begin{align*}
\| \mathcal{R}\psi \|_{ H^{s+1/2}(\mathbb{T}) } &\leq C \left( h,h_0,\|\beta\|_{H^{s+1/2}(\mathbb{T})} \right) \, \|\psi\|_{H^s(\mathbb{T})} .
\end{align*}

Using Proposition \ref{prop:DNOpara} and \eqref{eq:G0Irr}, we introduce the Fourier multiplier
\begin{align} \label{eq:AOp}
A_0 &\coloneqq \left( (g-\kappa \partial_x^2) G_{00} \right)^{1/2} , \quad G_{00} \coloneqq D \tanh(h D), 
\end{align}
and introducing the variable
\begin{align} \label{eq:uNewVar}
u &\coloneqq \psi - \mathrm{i} A_0 \, G_{00}^{-1} \eta ,
\end{align}
we obtain
\begin{align*}
u_t &+ \mathrm{i} A_0 u = \psi_t - \mathrm{i} A_0 G_{00}^{-1} G_0(\beta) \psi + \mathrm{i} A_0\psi + A_0^2 G_{00}^{-1}\eta,
\end{align*}
hence
\begin{align}  \label{eq:uSymmLin} 
u_t &+ \mathrm{i} A_{0,\beta} u  = \mathscr{P}_{ext} + r_0^{(2)} , \quad A_{0,\beta} \coloneqq A_0 \left[ 1 + G_{00}^{-1}D L(\beta) \right]  ,
\end{align}
where $r_0^{(2)}$ is the smooth remainder
\begin{align*}
r_0^{(2)} &\coloneqq  A_0 G_{00}^{-1} D L(\beta) A_0 G_{00}^{-1} \eta .
\end{align*}

Now we apply the same strategy to the full water waves system \eqref{eq:WWsys}. We denote by $\mathtt{A}_0$ and $\mathtt{G}_{00}$ the symbols associated to the operators $A_0$ and $G_{00}$ defined in \eqref{eq:AOp}.

\begin{proposition} \label{prop:ParalinFull}

Let $T>0$. There exists $s_0>0$ such that for any $s \geq s_0$ the following holds true: let $\beta \in H^{s+1/2}(\mathbb{T})$, and let $(\eta,\psi)$ be a solution of \eqref{eq:WWsys} such that
\begin{equation*}
(\eta,\psi) \in C([0,T]; H^{s+1/2}_0(\mathbb{T}) \times H^s(\mathbb{T}) ) ,
\end{equation*}
and such that \eqref{eq:StrConnected} holds true. Let
\begin{align} 
\omega &\coloneqq \psi - T_{B(\eta,\beta)\psi} \eta \in C( [0,T] ; H^s(\mathbb{T}) ), \label{eq:GoodUnB} \\
c &\coloneqq (1+\eta_x^2)^{-3/4} , \nonumber \\
\mathtt{p} &\coloneqq c^{-1/3} - \frac{5}{18} \mathrm{i} \chi \, \frac{ \left( \mathtt{A}_{0} \left[ 1 + \mathtt{G}_{00}^{-1}\xi \mathtt{L}(\beta) \right] \right)_{\xi} }{ \mathtt{A}_0 \left[ 1 + \mathtt{G}_{00}^{-1}\xi \mathtt{L}(\beta) \right] } c^{-4/3} c_x , \nonumber \\
\mathtt{q} &\coloneqq \chi(\xi) \left[ c^{2/3} \frac{ \mathtt{A}_{0}(\xi) }{ \mathtt{G}_{00}(\xi) } \left[ 1 + \mathtt{G}_{00}^{-1}\xi \mathtt{L}(\beta) \right] - \mathrm{i} \partial_x(c^{2/3}) \frac{ \mathtt{A}_{0}(\xi) \left[ 1 + \mathtt{G}_{00}^{-1}\xi \mathtt{L}(\beta) \right] }{ \xi \mathtt{G}_{00}(\xi) }\right] , \nonumber 
\end{align}
where $\mathtt{L}(\beta)$ is the symbol associated to the operator $L(\beta)$, and where $\chi \in C^{\infty}(\mathbb{R})$ satisfies $\chi \equiv 1$ for $|\xi| \geq \frac{2}{3}$ and $\chi \equiv 0$ for $|\xi| \leq \frac{1}{2}$.

Then the unknown 
\begin{align} \label{eq:TransfGU}
u &\coloneqq T_{\mathtt{p}} \omega - \mathrm{i} T_{\mathtt{q}} \eta
\end{align}
solves
\begin{align} \label{eq:uEq}
u_t + T_{V(\eta,\beta)\psi } u_x + \mathrm{i} A_{0}^{1/2} ( T_c A_{0}^{1/2} \left[ 1 + G_{00}^{-1}D L(\beta) \right] u ) + R(\eta,\psi,\beta,\kappa) &= T_{\mathtt{p}} \mathscr{P}_{ext} ,
\end{align}
where the remainder $R(\eta,\psi,\beta,\kappa)$ satisfies 
\begin{align}
R(\eta,\psi,\beta,\kappa) &= R_1(\eta,\beta)\psi + R_2(\eta,\beta,\kappa)\eta , \nonumber \\
\| R_1(\eta,\beta) \psi \|_{H^s(\mathbb{T})} &\leq C \left( \|\eta\|_{H^{s+1/2}(\mathbb{T})} , \|\beta\|_{H^{s+1/2}(\mathbb{T})} \right) \, \|\eta\|_{H^{s+1/2}(\mathbb{T})}^{\theta} \, \|\psi\|_{H^{s}(\mathbb{T})} , \nonumber 
\end{align}
\begin{align}
& \| R_2(\eta_1,\beta,\kappa) \eta_2 \|_{H^s(\mathbb{T})} \nonumber \\
&\leq C \left( \|\eta_1 \|_{H^{s+1/2}(\mathbb{T})} , \|\beta \|_{H^{s+1/2}(\mathbb{T})} , \kappa \right) \, \|\eta_1 \|_{H^{s+1/2}(\mathbb{T})}^{\theta} \, \| \eta_2 \|_{H^{s+1/2}(\mathbb{T})} , \label{eq:uEqEstRem}
\end{align}
for some $\theta \in (0,1]$.

\end{proposition}

\begin{proof}

First, we paralinearize the water waves system \eqref{eq:WWsys}: by applying Proposition 4.4 of \cite{pasquali2026two} (see also Proposition 3.13 and Proposition 3.24 of \cite{alazard2011water}) we obtain
\begin{equation} \label{eq:ParalinFull}
\begin{cases}
\eta_t + T_{V(\eta,\beta)\psi } \, \eta_x - G_0(\beta)\omega &= f_1 , \\
\omega_t + T_{V(\eta,\beta)\psi } \; \omega_x + \kappa \, T_{\mathtt{h}} \eta &= f_2 + \mathscr{P}_{ext} ,
\end{cases}
\end{equation}
where 
\begin{align}
\mathtt{h} &= \mathtt{h}^{(2)} + \mathtt{h}^{(1)}, \;\; \mathtt{h}^{(2)} = \frac{ \xi^2 }{ (1+\eta_x^2)^{3/2} }  , \;\; \mathtt{h}^{(1)} = - \frac{\mathrm{i}}{2} \, (\partial_x \, \partial_{\xi}) \mathtt{h}^{(2)} , \label{eq:hSymbol} 
\end{align}
and where
\begin{align*}
f_1 \in L^{\infty}([0,T]; H^{s+1/2}(\mathbb{T}) ) , &\;\; f_2 \in L^{\infty}([0,T]; H^{s}(\mathbb{T}) ) , \\
\| (f_1,f_2) \|_{ L^{\infty}([0,T]; H^{s+1/2}(\mathbb{T}) \times H^{s}(\mathbb{T}) ) } &\leq  C \left( \| (\eta,\psi) \|_{ L^{\infty}( [0,T]; H^{s+1/2}_0(\mathbb{T}) \times H^{s}(\mathbb{T}) )  }  \right) ,
\end{align*}
for some non-decreasing function $C$ depending on $h$, $h_0$, $\kappa$ and $\|\beta\|_{ H^{s+1/2}(\mathbb{T}) }$.

Then we introduce 
$\zeta \coloneqq T_{\mathtt{q}} \eta$ and $\vartheta \coloneqq T_{\mathtt{p}} \omega$, with $\mathtt{q}$ and $\mathtt{p}$ as in the statement of the proposition; we have
\begin{equation*}
\left\{
\begin{aligned}
&\zeta_t +T_V \zeta_x - T_{\mathtt{q}} G(\beta)\omega = \tilde{f}_1,\\
&\vartheta_t +T_V \vartheta_x + \kappa T_{\mathtt{p}} T_{\mathtt{h}}\eta =\tilde{f}_2+T_{\mathtt{p}} \mathscr{P}_{ext},
\end{aligned}
\right.
\end{equation*}
where
\begin{align*}
\tilde{f}_1 &\coloneqq T_{\mathtt{q}} f_1 +T_{\mathtt{q}_t}\eta +[T_V\partial_x ,T_{\mathtt{q}}]\eta,\\
\tilde{f}_2 &\coloneqq T_{\mathtt{p}} f_2 +T_{\mathtt{p}_t}\omega+[T_V\partial_x,T_{\mathtt{p}}]\omega.
\end{align*}
and where there exists $\theta \in (0,1]$ such that 
\begin{align*}
& \| \tilde{f}_j \|_{ H^s(\mathbb{T}) } \\
&\leq C \left( \| \eta \|_{ H^{s+1/2}(\mathbb{T}) } ,\| \beta \|_{ H^{s+1/2}(\mathbb{T}) },\kappa \right) \, \| \eta \|_{ H^{s+1/2}(\mathbb{T}) }^{\theta} \, \left[ \| \psi \|_{ H^{s}(\mathbb{T}) } + \| \eta \|_{ H^{s+1/2}(\mathbb{T}) }  \right] ,
\end{align*}
for $j=1,2$. In order to deduce the thesis, we need to prove that 
\begin{align}
& \| T_{\mathtt{q}} G_{0}(\beta)\omega - A_0^{1/2}  T_c A_0^{1/2} \left[ 1 + G_{00}^{-1}D L(\beta) \right] T_{\mathtt{p}} \omega \|_{ H^s(\mathbb{T}) } \nonumber \\
& \leq C \left( \|\eta\|_{H^{s+1/2}(\mathbb{T})} , \| \beta \|_{ H^{s+1/2}(\mathbb{T}) } \right) \, \|\eta\|_{H^{s+1/2}(\mathbb{T})}^{\theta} \, \| \omega \|_{H^{s}(\mathbb{T})} , \nonumber \\
& \| \kappa T_{\mathtt{p}} T_{\mathtt{h}} \eta - A_0^{1/2}  T_c A_0^{1/2} \left[ 1 + G_{00}^{-1}D L(\beta) \right] T_{\mathtt{q}} \eta \|_{ H^s(\mathbb{T}) } \nonumber \\
& \leq C \left( \|\eta\|_{H^{s+1/2}(\mathbb{T})} , \| \beta \|_{ H^{s+1/2}(\mathbb{T}) } , \kappa \right) \, \|\eta\|_{H^{s+1/2}(\mathbb{T})}^{\theta} \, \| \eta \|_{H^{s+1/2}(\mathbb{T})} , \label{eq:SymmTechEst1}
\end{align}

We introduce the following notation: given two operators $A$ and $B$, we write $A \sim B$ if for any $\mu \in \mathbb{R}$ there exists a positive constant $C \left( \| \eta \|_{ H^{s+1/2}(\mathbb{T}) } , \| \beta \|_{ H^{s+1/2}(\mathbb{T}) } \right)$ such that
\begin{align*}
\| (A-B) f \|_{ H^{\mu}(\mathbb{T}) } &\leq C \left( \| \eta \|_{ H^{s+1/2}(\mathbb{T}) } , \| \beta \|_{ H^{s+1/2}(\mathbb{T}) } \right) \| \eta \|_{ H^{s+1/2}(\mathbb{T}) } \| f \|_{ H^{\mu}(\mathbb{T}) } .
\end{align*}

In order to prove \eqref{eq:SymmTechEst1}, we need to show that
\begin{align} \label{eq:SymmTechEst2}
\begin{cases}
T_{\mathtt{q}} G_{0}(\beta) \sim A_0^{1/2}  T_c A_0^{1/2} \left[ 1 + G_{00}^{-1}D L(\beta) \right] \chi(D) T_{\mathtt{p}} , \\
\kappa T_{\mathtt{p}} T_{\mathtt{h}} \chi(D) \sim A_0^{1/2}  T_c A_0^{1/2} \left[ 1 + G_{00}^{-1}D L(\beta) \right] \chi(D) T_{\mathtt{q}}  ,
\end{cases}
\end{align}
where 
\begin{align*}
\mathtt{q} &= \mathtt{q}^{(1/2)} + \mathtt{q}^{(-1/2)}, \quad \mathtt{q}^{(1/2)} \in \Gamma^{1/2}_2, \quad \mathtt{q}^{(-1/2)} \in \Gamma^{-1/2}_1, \\
\mathtt{p} &= \mathtt{p}^{(0)} + \mathtt{p}^{(-1)}, \quad \mathtt{p}^{(0)} \in \Gamma^{0}_2, \quad \mathtt{p}^{(-1)} \in \Gamma^{-1}_1 .
\end{align*}

From symbolic calculus we have 
\begin{align*}
& A_0^{1/2}  T_c A_0^{1/2} \left[ 1 + G_{00}^{-1}D L(\beta) \right] \chi(D)  \sim T_{\mathtt{g}}, \\
& \mathtt{g} = \chi c \mathtt{A}_0 \left[ 1 + \mathtt{G}_{00}^{-1}\xi \mathtt{L}(\beta) \right] - \mathrm{i} \chi \, c_x \,  \partial_{\xi} ( \mathtt{A}_0^{1/2} ) \, \mathtt{A}_0^{1/2} \left[ 1 + \mathtt{G}_{00}^{-1}\xi \mathtt{L}(\beta) \right] ,
\end{align*}
while from \eqref{eq:SymmTechEst2} we obtain 
\begin{align*}
T_{\mathtt{g}} T_{\mathtt{p}} &\sim T_{\mathtt{g}_1} , \\
\mathtt{g}_1 &= \mathtt{g} \mathtt{p}^{(0)} + \chi c \mathtt{A}_0 \left[ 1 + \mathtt{G}_{00}^{-1}\xi \mathtt{L}(\beta) \right] \mathtt{p}^{(-1)} - \mathrm{i} \chi c \, \left( \mathtt{A}_{0} \left[ 1 + \mathtt{G}_{00}^{-1}\xi \mathtt{L}(\beta) \right] \right)_{\xi} \mathtt{p}^{(0)}_x ,
\end{align*}
hence we choose
\begin{equation} \label{eq:qformula}
\begin{cases}
\mathtt{q}^{(1/2)} \coloneqq \chi c \mathtt{p}^{(0)} \mathtt{G}_{00}^{-1}  \mathtt{A}_0 \left[ 1 + G_{00}^{-1}D L(\beta) \right] , \\
\mathtt{q}^{(-1/2)} \coloneqq -\mathrm{i} \chi \mathtt{G}_{00}^{-1}  \left( \mathtt{A}_{0} \left[ 1 + \mathtt{G}_{00}^{-1}\xi \mathtt{L}(\beta) \right] \right)_{\xi}  \left[ \frac{1}{2} c_x \mathtt{p}^{(0)} + c \mathtt{p}^{(0)}_x \right] \\
\qquad \qquad \quad + \chi c \mathtt{p}^{(-1)} \mathtt{G}_{00}^{-1} \mathtt{A}_0  \left[ 1 + \mathtt{G}_{00}^{-1}\xi \mathtt{L}(\beta) \right] .
\end{cases}
\end{equation}

Similarly, we have $T_{\mathtt{g}} T_{\mathtt{q}} \sim T_{\mathtt{g}_2}$, where
\begin{align*}
\mathtt{g}_2 &= \mathtt{g} \mathtt{q}^{(1/2)} - \mathrm{i} \chi c \left( \mathtt{A}_{0} \left[ 1 + \mathtt{G}_{00}^{-1}\xi \mathtt{L}(\beta) \right] \right)_{\xi} \mathtt{q}^{(1/2)}_x + \chi c \mathtt{A}_0 \left[ 1 + \mathtt{G}_{00}^{-1}\xi \mathtt{L}(\beta) \right] \mathtt{q}^{(-1/2)} \\
&= \chi \bigg[  c \, \mathtt{A}_0 \left[ 1 + \mathtt{G}_{00}^{-1}\xi \mathtt{L}(\beta) \right] \mathtt{q}^{(1/2)} -\mathrm{i} \chi \mathtt{G}_{00}^{-1}  \left( \mathtt{A}_0^2 \left[ 1 + \mathtt{G}_{00}^{-1}\xi \mathtt{L}(\beta) \right] \right)_{\xi}  \left( c \, c_x \mathtt{p}^{(0)} + c^2 \mathtt{p}^{(0)}_x \right) \\
&\qquad + \chi c^2 \mathtt{p}^{(-1)} \mathtt{G}_{00}^{-1} \mathtt{A}_0^2  \left[ 1 + \mathtt{G}_{00}^{-1}\xi \mathtt{L}(\beta) \right] \bigg] .
\end{align*}
Notice that
\begin{align*}
\mathtt{G}_{00}^{-1} \left( \mathtt{A}_0^2 \left[ 1 + \mathtt{G}_{00}^{-1}\xi \mathtt{L}(\beta) \right] \right)_{\xi}  = 3 \kappa \xi + r_1 , &\quad
\mathtt{G}_{00}^{-1}  \mathtt{A}_0^2 \left[ 1 + \mathtt{G}_{00}^{-1}\xi \mathtt{L}(\beta) \right] = \kappa \xi^2 + r_2,
\end{align*}
where $r_1$ and $r_2$ have order $0$. Observe that in the formula for $\mathtt{g}_2$ the terms $r_1 \left( c \, c_x \mathtt{p}^{(0)} + c^2 \mathtt{p}^{(0)}_x \right)$ and $c^2 \mathtt{p}^{(-1)} r_2$ can be regarded as remainders, so that
\begin{align*}
& A_0^{1/2}  T_c A_0^{1/2} \left[ 1 + G_{00}^{-1}\xi L(\beta) \right]T_{\mathtt{q}}  \sim T_{\mathtt{g}_3} , \\
& \mathtt{g}_3 = \chi \left[  c \, \mathtt{A}_0 \left[ 1 + \mathtt{G}_{00}^{-1}\xi \mathtt{L}(\beta) \right] \mathtt{q}^{(1/2)} -3 \kappa \mathrm{i} \chi \xi \left( c \, c_x \mathtt{p}^{(0)} + c^2 \mathtt{p}^{(0)}_x \right) + \chi c^2 \mathtt{p}^{(-1)} \kappa \xi^2 \right] .
\end{align*}

Similarly, we have
\begin{align*}
\kappa T_{\mathtt{p}} T_{\mathtt{h}} \chi(D) &\sim  \kappa T_{\chi \mathtt{p} ( \mathtt{h}^{(2)} + \mathtt{h}^{(1)} ) }  ,
\end{align*}
where 
\begin{align*}
\kappa \mathtt{h}^{(2)} = \kappa c^2 \xi^2 &= c^2 \left( \frac{ \mathtt{A}_0^2 }{ \mathtt{G}_{00} }  - r_1 \right) ,
\end{align*}
hence by \eqref{eq:qformula}
\begin{align*}
& \chi \mathtt{p} c^2 \left( \frac{ \mathtt{A}_0^2 \left[ 1 + \mathtt{G}_{00}^{-1}\xi \mathtt{L}(\beta) \right] }{ \mathtt{G}_{00} } - r_1 \right) \\
&= c \, \mathtt{A}_0 \mathtt{q}^{(1/2)} + \chi \left( \mathtt{p}^{(-1)} c^2 \frac{ \mathtt{A}_0^2 \left[ 1 + \mathtt{G}_{00}^{-1}\xi \mathtt{L}(\beta) \right] }{ \mathtt{G}_{00} }  - \mathtt{p} c^2 r_1 \right) ,
\end{align*}
and
\begin{align*}
\kappa T_{\mathtt{p} \mathtt{h}} \chi(D) &\sim T_{ c \, \mathtt{A}_0 \mathtt{q}^{(1/2)} + \chi \left( \mathtt{p}^{(-1)} c^2 \mathtt{G}_{00}^{-1}  \mathtt{A}_0^2 \left[ 1 + \mathtt{G}_{00}^{-1}\xi \mathtt{L}(\beta) \right]  - \mathtt{p} c^2 r_1 \right) } .
\end{align*}
Recalling the formulae for $\mathtt{g}_1$, $\mathtt{g}_2$ and $\mathtt{g}$, and observing that
\begin{align*}
\partial_{\xi} ( \mathtt{A}_0^{1/2} ) \, \mathtt{A}_0^{1/2} \, c_x   \mathtt{q}^{(1/2)} +  c \mathtt{A}_{0,\xi} \mathtt{q}^{(1/2)}_x 
%&= \partial_{\xi}( \mathtt{A}_0^{1/2} ) \mathtt{A}_0^{1/2} c_x \chi c \mathtt{p}^{(0)} \frac{ \mathtt{A}_0 }{ \mathtt{G}_{00} } ,
%&\quad + c \left[ \chi_x c \mathtt{p}^{(0)} \frac{ \mathtt{A}_0 }{ \mathtt{G}_{00} } + \chi c_x \mathtt{p}^{(0)} \frac{ \mathtt{A}_0 }{ \mathtt{G}_{00} } + \chi c \mathtt{p}^{(0)}_x \frac{ \mathtt{A}_0 }{ \mathtt{G}_{00} }+ \chi c \mathtt{p}^{(0)} \partial_x \left( \frac{ \mathtt{A}_0 }{ \mathtt{G}_{00} } \right) \right] \\
&= 3 \xi \left( c \, c_x \mathtt{p}^{(0)} + c^2 \mathtt{p}^{(0)}_x \right) + r_3 ,
\end{align*}
where $r_3$ has order $0$, we can deduce the second identity in \eqref{eq:SymmTechEst2}.

Finally, from \eqref{eq:qformula} we have
\begin{align*}
\mathtt{q} &= \chi \bigg[ c \mathtt{p}^{(0)} \mathtt{G}_{00}^{-1}  \mathtt{A}_0 \left[ 1 + \mathtt{G}_{00}^{-1}\xi \mathtt{L}(\beta) \right] -\mathrm{i} \mathtt{G}_{00}^{-1} \left( \mathtt{A}_{0} \left[ 1 + \mathtt{G}_{00}^{-1}\xi \mathtt{L}(\beta) \right] \right)_{\xi} \left( \frac{1}{2} c_x \mathtt{p}^{(0)} + c \mathtt{p}^{(0)}_x \right) \\
&\qquad +  c \mathtt{p}^{(-1)} \mathtt{G}_{00}^{-1}  \mathtt{A}_0 \left[ 1 + \mathtt{G}_{00}^{-1}\xi \mathtt{L}(\beta) \right] \bigg] ,
\end{align*}
where
\begin{align*}
\frac{1}{2} c_x \mathtt{p}^{(0)} + c \mathtt{p}^{(0)}_x &= \frac{1}{6} c^{-1/3} c_x ,
\end{align*}
and if we set
\begin{align*}
\mathtt{p}^{(-1)} &\coloneqq -\frac{5}{18} \mathrm{i} \chi \, \frac{ \left( \mathtt{A}_{0} \left[ 1 + \mathtt{G}_{00}^{-1}\xi \mathtt{L}(\beta) \right] \right)_{\xi} }{ \mathtt{A}_0 \left[ 1 + \mathtt{G}_{00}^{-1}\xi \mathtt{L}(\beta) \right] } c^{-4/3} c_x ,
\end{align*}
we can deduce the thesis.

\end{proof}

Now, using Proposition \ref{prop:ParalinFull} we have reduced the water waves system \eqref{eq:WWsys} to the single equation \eqref{eq:uEq}, namely
\begin{align*} 
u_t + T_{V(\eta,\beta)\psi } u_x + \mathrm{i} A_{0}^{1/2} ( T_c A_{0}^{1/2} \left[ 1 + G_{00}^{-1}D L(\beta) \right]  u ) + R(\eta,\psi,\beta,\kappa) &= T_{\mathtt{p}} \mathscr{P}_{ext} ,
\end{align*}
where $c$ depends on the unknown $\eta$, and where $V$ depends on the unknowns $(\eta,\psi)$ and on the fixed bottom topography $\beta$. We want to prove that $V$ and $c$ depend on $u$ and $\beta$ only; since by Lemma \ref{lem:InvGU} we already expressed $B$ and $V$ in terms of $\eta$, $\beta$ and $\omega$, we just need to write the unknowns $(\eta,\omega)$ in terms of $u$. For simplicity, we omit the dependence of all unknowns on time.

Let $s \geq 0$, and let us introduce the space
\begin{align} \label{eq:tildeHs}
\tilde{H}^s(\mathbb{T};\mathbb{C}) &\coloneqq \left\{ u \in H^s(\mathbb{T};\mathbb{C}) : \int_{\mathbb{T}} \mathrm{Im} \, u(x) \mathrm{d}x = 0 \right\} .
\end{align}

Let us introduce the map 
\begin{align*}
U : H^{s+1/2}_0(\mathbb{T}) \times H^{s}(\mathbb{T}) &\to \tilde{H}^{s}(\mathbb{T}) , \quad U(\eta,\psi) \coloneqq u = T_{\mathtt{p}} \omega - \mathrm{i} T_{\mathtt{q}}\eta .
\end{align*}

\begin{lemma} \label{lem:InvertUn}

Let $s \geq s_0 > \frac{5}{2}$, and let $\beta \in H^{s+1/2}(\mathbb{T})$. Then there exist positive constants $\epsilon_0$ and $K=K \left( \| \beta \|_{ H^{s+1/2}(\mathbb{T}) } \right)$ such that if $\| \eta \|_{ H^{s+1/2}(\mathbb{T}) } < \epsilon_0$, then there exists a map
\begin{align*}
Y_{\beta} : \tilde{H}^{s}(\mathbb{T};\mathbb{C}) &\to H^{s+1/2}_0(\mathbb{T}) \times H^{s}(\mathbb{T}) , \quad Y_{\beta}(u) = (\eta,\psi) ,
\end{align*}
where
\begin{align*}
\| \eta \|_{ H^{s+1/2}(\mathbb{T}) } \leq K \| u \|_{ H^s(\mathbb{T}) } , &\quad \| \psi \|_{ H^{s}(\mathbb{T}) } \leq K \| u \|_{ H^s(\mathbb{T}) } .
\end{align*}

\end{lemma}

\begin{proof}

Recalling that $u = T_{\mathtt{p}} \omega - \mathrm{i} T_{\mathtt{q}}\eta$, we have that 
\begin{align} \label{eq:InvEq}
\mathrm{Im} \, u = - T_{\mathtt{q}}\eta , &\quad \mathrm{Re} \, u = T_{\mathtt{p}} \omega ,
\end{align}
where $\mathtt{p}$, $ \mathtt{q}$ depend only on $\eta$. Hence, to obtain the result, it suffices to express $\eta$ in terms of $\mathrm{Im} \, u$ in the first equation of \eqref{eq:InvEq}, since we can invert the operator $T_{\mathtt{p}}$ in the second equation of \eqref{eq:InvEq} if $\| \eta \|_{ H^{s}(\mathbb{T}) }$ is sufficiently small, since
\begin{align*}
\| T_{\mathtt{p}} - \mathrm{Id} \|_{ L( H^s(\mathbb{T}) ,  H^s(\mathbb{T}) ) } &\leq C \left( \| \eta \|_{ H^{s+1/2}(\mathbb{T}) } \right) \, \| \eta \|_{ H^{s+1/2}(\mathbb{T}) } ,
\end{align*}
so that $T_{\mathtt{p}}$ is a small bounded perturbation of the identity. 

Let us introduce the symbol $Q_0 (\xi) \coloneqq \frac{ \chi(\xi) \mathtt{A}_0(\xi) }{ \mathtt{G}_{00}(\xi) }$, where $\chi$, $\mathtt{A}_{0}$ and $\mathtt{G}_{00}$ are the symbols defined in \eqref{eq:AOp} and in the statement of Proposition \ref{prop:ParalinFull}. We look for a solution of $F(\eta)=\eta$, where $F: H^{s+1/2}(\mathbb{T}) \to H^{s+1/2}(\mathbb{T})$ is given by
\begin{align*}
F(\eta) \coloneqq - (g- \kappa \partial_x^2)^{-1/2} G_{00}^{1/2} \left[ (T_{\mathtt{q}} - Q_0)\eta + \mathrm{Im} \, u \right] .
\end{align*}
Denoting by $\mathtt{q}_j$ the symbol $\mathtt{q}|_{\eta=\eta_j}$, $j=1,2$, we can check that $F$ is a contraction, since
\begin{align*}
& \| F(\eta_1) - F(\eta_2) \|_{ H^{s+1/2}(\mathbb{T}) } \\
&\leq C(g,\kappa) \left[ \| (T_{\mathtt{q}_1} - Q_0) (\eta_1-\eta_2) \|_{H^s(\mathbb{T})} + \| (T_{\mathtt{q}_1} - T_{\mathtt{q}_2}) \eta_2 \|_{H^s(\mathbb{T})} \right] \\
&\leq C(g,\kappa, M ) \, M \, \|  \eta_1 - \eta_2 \|_{ H^{s+1/2}(\mathbb{T}) } ,
\end{align*}
where $M \coloneqq \| \eta_1 \|_{ H^{s+1/2}(\mathbb{T}) } + \| \eta_2 \|_{ H^{s+1/2}(\mathbb{T}) }$, so that $F$ is a contraction if $M$ is sufficiently small.

\end{proof}

By Proposition \ref{prop:ParalinFull} and by Lemma \ref{lem:InvertUn} we can rewrite \eqref{eq:uEq} by expressing of $V=V(\eta,\beta)\psi$ in terms ofn $u$ (and similarly for $R$); by abuse of notation in the following we denote by $V(u,\beta)$ the function $V=V(\eta(u),\beta)\psi(u)$, and similarly for $R$.

\begin{corollary} \label{cor:ParalinFull}

Let $T>0$. There exists $s_0>0$ such that for any $s \geq s_0$ the following holds true: let $\beta \in H^{s+1/2}(\mathbb{T})$, and let $(\eta,\psi)$ be a solution of \eqref{eq:WWsys} such that
\begin{equation*}
(\eta,\psi) \in C([0,T]; H^{s+1/2}_0(\mathbb{T}) \times H^s(\mathbb{T}) ) ,
\end{equation*}
and such that \eqref{eq:StrConnected} holds true. Then there exists a positive constant $\epsilon_0$ such that if $\| \eta \|_{H^s(\mathbb{T})} < \epsilon_0$, then the unknown $u \in C([0,T] ; \tilde{H}^s(\mathbb{T}) ; \mathbb{C})$ defined in \eqref{eq:TransfGU} satisfies
\begin{align} \label{eq:uEqNew}
u_t + T_{V(u,\beta) } u_x + \mathrm{i} A_{0}^{1/2} ( T_{c(u,\beta)} A_{0}^{1/2} \left[ 1 + G_{00}^{-1}D L(\beta) \right]  u ) + R(u,\beta,\kappa) &= T_{\mathtt{p}(u,\beta)} \mathscr{P}_{ext} .
\end{align}

\end{corollary}

Observe that by \eqref{eq:EstGUSob} and by Lemma \ref{lem:InvertUn} we can deduce the following estimate: under the assumptions of Lemma \ref{lem:InvertUn}, there exists $\epsilon_0>0$ such that if $\|\eta\|_{H^{s+1/2}(\mathbb{T})} < \epsilon_0$, then
\begin{align} 
& \| V(u,\beta) \|_{H^{s-1}(\mathbb{T})} \nonumber \\
&\leq C \left( K(\|\beta\|_{H^{s+1/2}(\mathbb{T})}) \|u\|_{H^s(\mathbb{T})} , \|\beta\|_{H^{s+1/2}(\mathbb{T})} \right) \, K(\|\beta\|_{H^{s+1/2}(\mathbb{T})}) \|u\|_{H^s(\mathbb{T})} \nonumber \\
&\leq C_1 \left( \|\beta\|_{H^{s+1/2}(\mathbb{T})} , \|u\|_{H^s(\mathbb{T})} \right) \, \|u\|_{H^{s}(\mathbb{T})} . \label{eq:EstVubeta}
\end{align}

\begin{remark} \label{rem:StrucRem}

If we consider Eq. \eqref{eq:uEqNew}, we can rewrite the remainder $R(u,\beta,\kappa)$ in the form $R(u,\beta,\kappa)u$, where for any $\underline{u} \in \tilde{H}^s(\mathbb{T};\mathbb{C})$ the map 
\begin{align*}
\tilde{H}^s(\mathbb{T};\mathbb{C}) &\to \tilde{H}^s(\mathbb{T};\mathbb{C}) \\
u &\mapsto R(\underline{u},\beta,\kappa)u
\end{align*}
is linear. 

\end{remark}

By Proposition \ref{prop:ParalinFull}, by Corollary \ref{cor:ParalinFull} and by Remark \ref{rem:StrucRem}, we now fix $T>0$, and for any $s$ sufficiently large, for any $\beta \in H^{s+1/2}(\mathbb{T})$ and for any solution $(\eta,\psi) \in C( [0,T]; H^{s+1/2}_0(\mathbb{T}) \times H^s(\mathbb{T}) )$ of \eqref{eq:WWsys}, we have that the good unknown $u \in C( [0,T]; \tilde{H}^{s}(\mathbb{T} ; \mathbb{C}) )$ defined in \eqref{eq:TransfGU} satisfies 
\begin{align*} 
u_t + T_{V(u,\beta) } u_x + \mathrm{i} A_{0}^{1/2} ( T_{c(u,\beta)} A_{0}^{1/2} \left[ 1 + G_{00}^{-1}D L(\beta) \right]  u ) + R(u,\beta,\kappa)u &= T_{\mathtt{p}(u,\beta)} \mathscr{P}_{ext} .
\end{align*}

We now fix $\underline{u} \in C( [0,T]; \tilde{H}^{s}(\mathbb{T} ; \mathbb{C}) )$, and we set $V = V(\underline{u},\beta)$, $c = c(\underline{u},\beta)$, $R=R(\underline{u},\beta,\kappa)$ and $\mathtt{p}=\mathtt{p}(\underline{u},\beta)$, and we consider the linear operator 
\begin{align} \label{eq:OpP0beta}
P_{0,\beta} &\coloneqq \partial_t + T_{V } \partial_x + \mathrm{i} A_{0}^{1/2} ( T_{c} A_{0}^{1/2} \left[ 1 + G_{00}^{-1}D L(\beta) \right]  \cdot ) + R .
\end{align}

We now assume that the following conditions hold true.
\begin{itemize}
\item[(A1)] Let $s_0$ be sufficiently large, and let $V,c \in C^0( [0,T] ; H^{s_0}(\mathbb{T}) )$, where $c$ is bounded from below by $\frac{1}{2}$. Moreover, the symbol $\mathtt{p}$ is given by
\begin{equation*}
c^{-1/3} - \mathrm{i} \frac{5}{18} \chi(\xi) \frac{ \mathtt{A}_{0,\xi}(\xi) }{ \mathtt{A}_{0}(\xi) } c^{-4/3} c_x ,
\end{equation*}
and $\| c - 1 \|_{ W^{3/2,\infty}(\mathbb{T}) }$ is sufficiently small.
\item[(A2)]If $P_{0,\beta}u$ is a real-valued function, then $\frac{\mathrm{d}}{\mathrm{d}t} \int_{\mathbb{T}} \mathrm{Im} \, u(t,x) \mathrm{d}x = 0$.
\end{itemize}

Using Lemma 3.4 in \cite{alazard2018control}, we have that condition (A2) holds true.

We now fix a non-empty open domain $\Omega \subset \mathbb{T}$, and let us denote by $\chi_{\Omega}$ be a cut-off function such that $\chi_{\Omega} \equiv 1$ on $\Omega$. We consider the following problem: given $v_0 \in \tilde{H}^s(\mathbb{T};\mathbb{C})$, we want to determine whether there exists $f \in C^0( [0,T] ; H^s(\mathbb{T}) )$ such that the unique solution to the Cauchy problem
\begin{equation} \label{eq:CauchyProb}
\begin{cases}
P_{0,\beta}v &= T_{\mathtt{p}} \chi_{\Omega} \mathrm{Re} f , \\
v(0,\cdot) &= v_0 ,
\end{cases}
\end{equation}
satisfies $v(T,\cdot) =0$.

\begin{remark} \label{rem:SolCauchy}

The existence and uniqueness of the solution of the Cauchy problem \eqref{eq:CauchyProb} follows by the argument used in the proof of Proposition B.1 in \cite{alazard2018control}.

\end{remark}

Therefore, our aim in the next section will be to prove the following result.

\begin{proposition} \label{prop:ParaControl}

There exists $s_0 \gg 1$ such that for all $T \in (0,1]$, for all $s \geq s_0$ and for all $\beta \in H^{s+1/2}(\mathbb{T})$, if assumptions $\mathrm{(A1)}-\mathrm{(A2)}$ hold true, then there exist two positive constants $\delta=\delta(T,s, \| \beta \|_{ H^{s+1/2}(\mathbb{T}) }  )$ and $K=K(T,s, \| \beta \|_{ H^{s+1/2}(\mathbb{T}) } )$ such that if
\begin{equation} \label{eq:SmallParaControl}
\begin{cases}
\| V \|_{ C^0( [0,T] ; H^{s_0}(\mathbb{T}) ) } + \| c-1 \|_{ C^0( [0,T] ; H^{s_0}(\mathbb{T}) ) } &\leq \delta , \\
\| \partial_t^k V \|_{ C^0( [0,T] ; H^{1}(\mathbb{T}) ) } + \| \partial_t^k c \|_{ C^0( [0,T] ; H^{1}(\mathbb{T}) ) } &\leq \delta , \quad k=1,2,3, \\
\| R \|_{ C^0 \left( [0,T] ; L( H^{s}(\mathbb{T}) ) \right) } &\leq \delta ,
\end{cases}
\end{equation}
then for any $v_0 \in \tilde{H}^s(\mathbb{T};\mathbb{C})$ there exists $f \in C^0( [0,T] ; H^s(\mathbb{T}) )$ such that
\begin{itemize}
\item[i.] the unique solution to the Cauchy problem \eqref{eq:CauchyProb} satisfies $v(T,\cdot) =0$;
\item[ii.] $\| f \|_{ C^0( [0,T] ; H^s(\mathbb{T}) ) } \leq K \| v_0 \|_{ \tilde{H}^s(\mathbb{T};\mathbb{C}) }$.
\end{itemize}

\end{proposition}

We point out that the smallness assumption on $V$ and $c$ involves the $H^{s_0}$-norm, while the result of Proposition \ref{prop:ParaControl} holds for initial data in $H^s$, with $s \geq s_0$; we will use this property with $s_0=s-2$.

\section{Reductions} \label{sec:Red}

In this section we reduce Proposition \ref{prop:ParaControl} to a simpler result; more precisely, we prove that:
\begin{itemize}
\item it is enough to consider a classical equation instead of a paradifferential equation (this observation is used below in order to simplify the computation of a change of variable);
\item it is enough to prove an $L^2$-result instead of a result in higher-order Sobolev spaces. 
\end{itemize}

The main idea is to  commute the equation with an elliptic semi-classical operator $\Lambda_{\delta,s}$ of order $s$: in order to choose this elliptic operator, we prove that $\Lambda_{\delta,s}$ can be chosen so that it satisfies the following commutator estimates (this also explains why we introduce the small parameter $\delta$):
\begin{align*}
\| [\Lambda_{\delta,s},P_{0,\beta}]\Lambda_{\delta,s}^{-1} \|_{L( L^2(\mathbb{T}) )} &= \mathcal{O}(1),\\
\| [\Lambda_{\delta,s},\chi_\omega]\Lambda_{\delta,s}^{-1} \|_{L( L^2(\mathbb{T}) )} &= \mathcal{O}(\delta) .
\end{align*}

We make the following ansatz for the operator $\Lambda_{\delta,s}$:
\begin{align} \label{eq:ansatzLambda}
\Lambda_{\delta,s} &= \mathrm{id} + \delta^{s} T_{c^{\frac{2s}{3}}} A_0^{\frac{2s}{3}}.
\end{align}

As proved in Lemma 4.1 of \cite{alazard2018control}, the operator introduced in \eqref{eq:ansatzLambda} satisfies the following properties.

\begin{lemma} \label{lem:LambdaProp}

\begin{itemize}
\item[i.] Assume that the $L^\infty_{t,x}$-norm of $c-1$ is small enough. Then $\Lambda_{\delta,s}$ is invertible from $H^s(\mathbb{T})$ to $L^2(\mathbb{T})$, and we denote its inverse by $\Lambda_{\delta,s}^{-1}$. 

\item[ii.] Moreover, for any $s_1 \leq s$, the operator $\delta^{s_1} \Lambda_{\delta,s}^{-1}$ is uniformly bounded from $L^2(\mathbb{T})$ to $H^{s_1}(\mathbb{T})$: namely, there exists $K>0$ such that for any $\delta \in (0,1]$ and any $g$ in $L^2(\mathbb{T})$, 
\begin{align} \label{eq:EstLambdadeltas}
\| \delta^{s_1} \Lambda_{\delta,s}^{-1} g \|_{ H^{s_1}(\mathbb{T}) } &\leq K \| g \|_{ L^2(\mathbb{T}) } .
\end{align}
\end{itemize}

\end{lemma}

\subsection{Reduction to a classical equation} \label{subsec:RedClassical}

Now we conjugate the paradifferential equation in \eqref{eq:CauchyProb} to a classical equation.

First, we notice that we have suitable bounds for the commutators between $\Lambda_{\delta,s}$ and the operators appearing in \eqref{eq:CauchyProb}.

\begin{lemma} \label{lem:CommEst}

There exists $s_0 \gg 1$ such that for all $s \geq s_0$ and for all $\beta \in H^{s+1/2}(\mathbb{T})$, there exists $\varepsilon_0>0$ such that the following holds true: for any $0 < \varepsilon < \varepsilon_0$, if $\| c-1 \|_{ W^{\frac{3}{2},\infty}(\mathbb{T}) } \leq \varepsilon$, there exists $K>0$ (depending on $\varepsilon$ and on $\| \beta \|_{ H^{s+1/2}(\mathbb{T}) }$) such that for any $\delta \in (0,1]$ and for any $g$ in $L^2(\mathbb{T})$
\begin{align}
\| \left[ \Lambda_{\delta,s}, T_V \right] \Lambda_{\delta,s}^{-1} g \|_{ L^2(\mathbb{T}) } &\leq K 
\| V \|_{ W^{1,\infty}(\mathbb{T}) } \, \| g \|_{ L^2(\mathbb{T}) } , \label{eq:CommEst1} \\
\| \left[ \Lambda_{\delta,s}, \chi_{\Omega} \right] \Lambda_{\delta,s}^{-1} g \|_{ L^2(\mathbb{T}) } &\leq K \delta \, \| \chi_{\Omega} \|_{ H^{s+1}(\mathbb{T}) } \, \| g \|_{ L^2(\mathbb{T}) } , \label{eq:CommEst2} \\
\left\| \left[ \Lambda_{\delta,s}, A_0^{1/2} (T_c A_0^{1/2} \left[ 1 + G_{00}^{-1}D L(\beta) \right] \cdot ) \right] \Lambda_{\delta,s}^{-1} g \right\|_{ L^2(\mathbb{T}) } &\leq K \delta \, \| g \|_{ L^2(\mathbb{T}) } . \label{eq:CommEst3}
\end{align}

\end{lemma}

The proof of Lemma \ref{lem:CommEst} follows along the lines of the proof of Lemma 4.2 in \cite{alazard2018control}.

Next, we conjugate the operator $P_{0,\beta}$ introduced in \eqref{eq:OpP0beta} with $\Lambda_{\delta,s}$. Let us write $\tilde{P}_{0,\beta,\delta} \coloneqq \Lambda_{\delta,s} P_{0,\beta} \Lambda_{\delta,s}^{-1}$. Then
\begin{align*}
	\tilde{P}_{0,\beta,\delta} &= \partial_t + T_{V} \partial_x + \mathrm{i} A_0^{1/2} ( T_c A_0^{1/2} \left[ 1 + G_{00}^{-1}D L(\beta) \right] \cdot ) + R_{1,\delta} , \\
	R_{1,\delta} &\coloneqq \Lambda_{\delta,s} R \Lambda_{\delta,s}^{-1} + \left[ \Lambda_{\delta,s} , \partial_t \right] \Lambda_{\delta,s}^{-1} + \left[ \Lambda_{\delta,s} , T_{V} \partial_x \right] \Lambda_{\delta,s}^{-1} \\
	&\quad + \mathrm{i} \, \left[ \Lambda_{\delta,s} , A_0^{1/2} ( T_c A_0^{1/2} \left[ 1 + G_{00}^{-1}D L(\beta) \right] \cdot ) \right] \Lambda_{\delta,s}^{-1} .
\end{align*}

\begin{lemma} \label{lem:EstR1delta}
	
	There exists $s_0 \gg 1$ such that for all $s \geq s_0$ and for all $\beta \in H^{s+1/2}(\mathbb{T})$, there exists $\varepsilon_0>0$ such that the following holds true: for any $0 < \varepsilon < \varepsilon_0$, if $\| c-1 \|_{ W^{\frac{3}{2},\infty}(\mathbb{T}) } \leq \varepsilon$, there exists $K>0$ (depending on $\varepsilon$ and on $\| \beta \|_{ H^{s+1/2}(\mathbb{T}) }$) such that for any $\delta \in (0,1]$ and for any $g$ in $L^2(\mathbb{T})$
	\begin{align} \label{eq:EstR1delta}
	\| R_{1,\delta}g \|_{ L^2(\mathbb{T}) }	&\leq K \, \left[ \| V \|_{ W^{1,\infty}(\mathbb{T}) } + \| c_t \|_{ L^{\infty}(\mathbb{T}) } + \delta^{-s} \| R \|_{ L( H^s(\mathbb{T}) ) } \right] \, \| g \|_{ L^2(\mathbb{T}) } .
	\end{align}
	
\end{lemma}

\begin{proof}
	
	We have
	\begin{align*}
	\| \Lambda_{\delta,s}  R \Lambda_{\delta,s}^{-1} \|_{ L( L^2(\mathbb{T}) ) } &\leq	\| \Lambda_{\delta,s} \|_{ L(H^s(\mathbb{T}) , L^2(\mathbb{T}) ) } \, \| R \|_{ L(H^s(\mathbb{T}) , H^s(\mathbb{T}) ) } \,
	\| \Lambda_{\delta,s}^{-1} \|_{ L( L^2(\mathbb{T}) , H^s(\mathbb{T}) ) } \\
	&\leq K \delta^{-s}  \| R \|_{ L(H^s(\mathbb{T}) , H^s(\mathbb{T}) ) } ,
	\end{align*}
	since $\| \Lambda_{\delta,s} \|_{ L( H^s(\mathbb{T}) , L^2(\mathbb{T}) )} = \mathcal{O}(1)$ and $\| \Lambda_{\delta,s}^{-1} \|_{ L( L^2(\mathbb{T}) , H^s(\mathbb{T}) )} = \mathcal{O}(\delta^{-s})$ by \eqref{eq:EstLambdadeltas}. 
	
	Similarly, we can use \eqref{eq:CommEst1}  in order to bound the term 
	$\left[ \Lambda_{\delta,s} , T_{V} \partial_x \right] \Lambda_{\delta,s}^{-1}$, we can use  \eqref{eq:CommEst3} in order to bound the term $\mathrm{i} \, \left[ \Lambda_{\delta,s} , A_0^{1/2} ( T_c A_0^{1/2} \left[ 1 + G_{00}^{-1}D L(\beta) \right] \cdot ) \right] \Lambda_{\delta,s}^{-1}$, and we can deduce the thesis.

\end{proof}

We just point out that the choice of $\delta$, and hence $\delta^{-s}$, in \eqref{eq:EstR1delta} will not be an issue, as it will ultimately depend only on $T$.

We now want to transform the operator \eqref{eq:OpP0beta} by replacing the operators
\begin{equation*}
T_V \partial_x , \quad A_0^{1/2} ( T_c A_0^{1/2} \left[ 1 + G_{00}^{-1}D L(\beta) \right] \cdot ) ,
\end{equation*}
by 
\begin{equation*}
V \partial_x , \quad A_0^{1/2} ( c A_0^{1/2} \left[ 1 + G_{00}^{-1}D L(\beta) \right] \cdot ) ,
\end{equation*}
up to smoother remainder terms: hence, we introduce
\begin{align} 
\tilde{P}_{0,\beta,\delta} &\coloneqq \partial_t +V \partial_x + \mathrm{i} \, A_0^{1/2} \left( c A_0^{1/2} \left[ 1 + G_{00}^{-1}D L(\beta) \right] \cdot \right) + R_{2,\delta} , \label{eq:tildePdelta} \\
R_{2,\delta} g &\coloneqq R_{1,\delta} g + ( T_{V} \partial_x  - V \partial_x ) g \nonumber \\
&\quad + \mathrm{i} \left( A_0^{1/2} T_c A_0^{1/2}  - A_0^{1/2} \, c A_0^{1/2}   \right) \left[ 1 + G_{00}^{-1}D L(\beta) \right] g . \nonumber 
\end{align}

\begin{lemma} \label{lem:EstR2delta}
	
	There exists $s_0 \gg 1$ such that for all $s \geq s_0$ and for all $\beta \in H^{s+1/2}(\mathbb{T})$, there exists $\varepsilon_0>0$ such that the following holds true: for any $0 < \varepsilon < \varepsilon_0$, if $\| c-1 \|_{ W^{\frac{3}{2},\infty}(\mathbb{T}) } \leq \varepsilon$, there exists $K>0$ (depending on $\varepsilon$ and on $\| \beta \|_{ H^{s+1/2}(\mathbb{T}) }$ and independent of $\delta$) such that for any $\delta \in (0,1]$ and for any $g$ in $L^2(\mathbb{T})$
	\begin{align} 
	&\| R_{2,\delta}g \|_{ L^2(\mathbb{T}) } \nonumber \\
	&\leq K \, \left[ \| V \|_{ H^{s_0}(\mathbb{T}) } + \| c-1 \|_{ H^{s_0}(\mathbb{T}) } + \| c_t \|_{ H^1(\mathbb{T}) }  + \delta^{-s} \| R \|_{ L( H^s(\mathbb{T})  ) }  \right] \, \| g \|_{ L^2(\mathbb{T}) } . \label{eq:EstR2delta}
	\end{align}

\end{lemma}

\begin{proof}

 	Recall that we can bound $R_{1,\delta}$ by \eqref{eq:EstR1delta}, and that the right-hand side of \eqref{eq:EstR1delta} is bounded by that of \eqref{eq:EstR2delta} for $s_0 > \frac{3}{2}$. In order to estimate $(T_V \partial_x - V \partial_x ) g$, we have by Lemma \ref{lem:Paraprod4} that
	\begin{align*}
	\| ( T_V \partial_x  - V \partial_x ) g \|_{ L^2(\mathbb{T}) } &\leq K \, \| V \|_{ H^{s_0}(\mathbb{T}) } \, \| g \|_{ L^2(\mathbb{T}) }.
	\end{align*}
	The estimate for $ A_0^{1/2} \left( (T_c-c  \, \mathrm{Id}) A_0^{1/2} \cdot \right)$  follows by a similar argument, using that $s_0>2$. 

\end{proof}

Arguing as in pp. 683-684 of \cite{alazard2018control}, we have that Proposition \ref{prop:ParaControl} follows from the following result.

\begin{proposition} \label{prop:reduction}
	
	Let $T\in (0,1]$, and consider an open non-empty set $\Omega \subset \mathbb{T}$.
	
	There exists $s_0 \gg 1$ such that for all $s \geq s_0$ and for all $\beta \in H^{s+1/2}(\mathbb{T})$ the following holds true: consider an operator of the form
	\begin{align} \label{eq:tildeP}
	\tilde{P}_{0,\beta} &\coloneqq \partial_t + V \partial_x + \mathrm{i} A_0^{1/2} \left( c A_0^{1/2}  \left[ 1 + G_{00}^{-1}D L(\beta) \right] \cdot \right) + R_{2,\delta} ,
	\end{align}
	then there exist two positive constants $\varepsilon=\varepsilon(T, \|\beta\|_{ H^{s+1/2}(\mathbb{T}) } )$ and $K=K(T, \|\beta\|_{ H^{s+1/2}(\mathbb{T}) } )$ 
	such that, if
	\begin{align*} 
		\| V \|_{C^0([0,T],H^{s_0}(\mathbb{T}) )}+ \| c-1 \|_{C^0([0,T],H^{s_0}(\mathbb{T})  )} &\leq \varepsilon, \\
		\| \partial_t^k V \|_{C^0( [0,T],H^1(\mathbb{T}) )}
		+\| \partial_t^k c \|_{ C^0( [0,T],H^1(\mathbb{T}) )} &\leq \varepsilon , k=1,2,3,\\
		\| R_{2,\delta} \|_{ C^0( [0,T],\mathcal{L}(L^2(\mathbb{T})  )  )} &\leq \varepsilon ,
	\end{align*}
	then for any initial data $v_{0} \in L^2(\mathbb{T})$ there exists $f\in C^0([0,T];L^2(\mathbb{T}))$ such that: 
	\begin{itemize}
		\item 
		the unique solution $v$ to $\tilde{P}_{0,\beta}v=\chi_{\Omega} \, \mathrm{Re} f,\quad v(t=0,\cdot)=v_{0}$ is such that $v(T)$ is an imaginary constant, namely 
		$\exists b \in \mathbb{R}$ such that $v(T,x)=\mathrm{i} b$ $\forall x \in \mathbb{T}$.

		\item  $\| f \|_{ C^0( [0,T],L^2(\mathbb{T}) )} \leq K \| v_{0} \|_{ L^2(\mathbb{T}) }$. 
	\end{itemize}
	
\end{proposition}

\subsection{Second Reduction} \label{sec:2ndRed}

In Sec. \ref{subsec:RedClassical} we reduced the study of the control problem in Sobolev spaces for $P_{0,\beta} = \partial_t  +T_V\partial_x + \mathrm{i} A_0^{1/2} ( T_c A_0^{1/2} \left[ 1 + G_{00}^{-1}D L(\beta) \right] \cdot ) + R$
to a control problem in $L^2(\mathbb{T})$ for $\tilde{P}_{0,\beta}=\partial_t +V \partial_x + \mathrm{i}  A_0^{1/2} (c A_0^{1/2} \left[ 1 + G_{00}^{-1}D L(\beta) \right] \cdot)+ R_{2}$.

Now the aim is to reduce the analysis to an equation in which the operator $ A_0^{1/2} (c A_0^{1/2} \cdot)$ is replaced with a constant-coefficient operator. In order to do so, we exploit changes of variables which preserve the $L^2$-scalar product; this allows us to conjugate $\tilde{P}_{0,\beta}$ to an operator of the form
\begin{align*}
\partial_t + W \partial_x + \mathrm{i} A_{0,\beta} + R_3
\end{align*}
where $A_{0,\beta}$ is the operator defined in \eqref{eq:uSymmLin}, $R_3$ is a pseudodifferential operator of order $0$, and where $W = W(t,x)$ has zero average, namely $\int_{\mathbb{T}} W(t,x)\, \mathrm{d}x=0$. 

\begin{proposition} \label{prop:conj1}
	
	There exists $s_0 \gg 1$ such that for all $s \geq s_0$ and for all $\beta \in H^{s+1/2}(\mathbb{T})$ there exist positive constants $\delta_0 \in (0,1)$, $r \geq 2$, $C > 0$ such that the following properties hold.  
	Assume that $c,V,R_2$ satisfy
	\begin{align} 
	&\| c - 1 \|_{C^0( [0,T] , L^\infty(\mathbb{T}) ) } < \delta_0, \quad \mathcal{N}_0 \leq 1, \label{eq:smallness} \\
	& \mathcal{N}_0 \coloneqq  \| V \|_{ C^0([0,T], H^1(\mathbb{T}) ) } + \| c-1 \|_{ C^0([0,T], H^r(\mathbb{T}) ) }	\nonumber \\
	&\qquad + \| c_t \|_{ C^0([0,T], H^1(\mathbb{T}) ) } + \| R_2 \|_{ C^0([0,T], L(L^2(\mathbb{T})) ) } . \nonumber
	\end{align}
	Then there exist a constant $T_1 > 0$ and a bounded invertible linear map 
	\begin{align*}
	\Phi &: C^0([0,T];L^2(\mathbb{T})) \to C^0([0,T_1];L^2(\mathbb{T}))
	\end{align*}
	with bounded inverse $\Phi^{-1}$ such that 
	\begin{align*}
	\tilde{P}_{0,\beta} g = m \Phi^{-1} \left( \tilde{P}_3 ( \Phi g ) \right), 
	\end{align*}
	where $m = m(t)$ is a function of time only, defined for $t \in [0,T]$, and 
	\begin{align*}
	\tilde{P}_3 = \partial_t + W \partial_x + \mathrm{i} A_{0,\beta} + R_3 .
	\end{align*}
	Moreover, we have that:
	\begin{itemize}
	\item[i.] $W = W(t,x)$ is defined for $t \in [0,T_1]$,  it has zero average, and it satisfies
	\begin{align} \label{eq:EstW}
	\| W \|_{ C^0([0,T_1];H^2(\mathbb{T})) } &\leq C \left( \| c-1 \|_{ C^0([0,T];H^2(\mathbb{T})) } + \| V \|_{ C^0([0,T];H^2(\mathbb{T})) }
	+ \| c_t \|_{ C^0([0,T];H^1(\mathbb{T})) } 	\right) ;
	\end{align}
	\item[ii.] $R_3$ maps $C^0([0,T_1];L^2(\mathbb{T}))$ into itself, with 
	\begin{align} \label{eq:EstR3}
	\| R_3 \|_{ C^0([0,T]; L(L^2(\mathbb{T}))) } &\leq C \mathcal{N}_0 ;
	\end{align}
	\item[iii.] $T_1$ and the function $m$ satisfy
	\begin{align*}
	\left| \frac{T_1}{T}\, - 1 \right| + \| m - 1 \|_{C^0([0,T])} &\leq C \| c-1 \|_{ C^0([0,T];L^{\infty}(\mathbb{T})) } ;
	\end{align*}
	\item[iv.] the map $\Phi$ is the composition of three local transformations 
	$\Phi = \varphi_{\ast}^{-1} \psi_{\ast}^{-1} \Psi_1$, where 
	\begin{align} \label{eq:EstLocTransf}
		(\Psi_1 h)(t,x) & := (1 + \partial_x \tilde\beta_1(t,x))^{1/2} h(t,x + \tilde\beta_1(t,x)),
		\\
		(\psi_{\ast}^{-1} h)(t,x) & := h(\psi^{-1}(t),x), 
		\qquad (\varphi_{\ast}^{-1} h)(t,x) := h(t,x-\mathtt{p}(t)). \nonumber
	\end{align}
	\end{itemize}
	
\end{proposition}

The proof of the above proposition follows the argument in Appendix C of \cite{alazard2018control}.

In order to study the control problem for the equation
\begin{align*}
\left[ \partial_t + W \partial_x +\mathrm{i} A_{0,\beta} + R_3 \right] v &= 0
\end{align*}
we use the HUM method. The first step consists in proving an observability 
inequality for solutions of the dual equation, namely
\begin{align*}
	\left[ -\partial_t -\partial_x (W \cdot) -\mathrm{i} A_{0,\beta} + R_3^{\ast} \right] w &= 0 ,
\end{align*}
which can be rewritten in the form
\begin{align*}
\mathcal{P}w &\coloneqq \left[ \partial_t + W \partial_x + \mathrm{i} A_{0,\beta} + R_4 \right] w = 0 , \quad R_4 w \coloneqq -R_3^{\ast}w + W_x w.
\end{align*}
In order to do so, we prove that $\mathcal{P}$ is conjugated to a simpler operator, where $\partial_t +W \partial_x $ is replaced by $\partial_t$. 

In the rest of the section we use the following notation. We recall that given a function $f$ with zero mean, $\partial_x^{-1}f$ is the zero-mean primitive of $f$, and is given by
\begin{align*}
	\partial_x^{-1}f &= \sum_{j\neq 0} \frac{f_j}{ \mathrm{i}j } e^{\mathrm{i}jx}, \quad f(x)=\sum_{j\neq 0} f_j e^{\mathrm{i}jx}.
\end{align*}

We now look for an operator $\mathcal{A}$ such that
\begin{align*}
\left[ \partial_t + W \partial_x + \mathrm{i} A_{0,\beta} + R_4 \right] \mathcal{A} &= \mathcal{A} \left[  \partial_t +\mathrm{i} A_{0,\beta} + R_5 \right] ,
\end{align*}
where $R_5$ is a bounded operator: more explicitly,
\begin{align} \label{eq:R5}
R_5 &\coloneqq \mathcal{A}^{-1}\Big(\left[ \partial_t,\mathcal{A}\right] + R_4 \mathcal{A} +W \partial_x \mathcal{A}+\mathrm{i} \left[ A_{0,\beta} ,\mathcal{A} \, \right] \Big).
\end{align}
We introduce the pseudodifferential operator $\mathcal{A}$ as follows: let 
\begin{align*}
\phi(t,x,k) &\coloneqq k x + b(t,x) |k|^{1/2} ,
\end{align*}
for some function $b$ to be determined. Consider also an amplitude $q(t,x,k)$ that has also to be determined: then set
\begin{align} \label{eq:OpCalA}
\mathcal{A} f(t,x) &\coloneqq \sum_{k \in \mathbb{Z}} \hat{f}_k \, q(t,x,k) e^{\mathrm{i}\phi(t,x,k)}, \quad f(t,x) = \sum_{k \in \mathbb{Z}} \hat{f}_k(t) e^{\mathrm{i}kx} ,
\end{align}
for any periodic functions $f$.

Below $t$ is considered as a parameter, and we omit it in most expressions. 

We set
\begin{align} 
\mathcal{N} &\coloneqq  \| V \|_{ C([0,T];H^{s_0}(\mathbb{T})) } + \| c-1 \|_{ C([0,T];H^{s_0}(\mathbb{T})) } \nonumber \\
&\quad + \| c_t \|_{ C([0,T];H^1(\mathbb{T})) }	+ \| R_2 \|_{ C([0,T];\mathcal{L}(L^2(\mathbb{T}))) }, \label{eq:calN}
\end{align}
where $s_0$ is some fixed large enough integer. Hereafter, $s_0$ always refers to an index large enough whose value may vary from one statement to another. In the rest of the section we always assume that the quantity $\mathcal{N}$ defined in \eqref{eq:calN} is sufficiently small.

Now we state two technical results, together with a commutator estimate for the operator $\mathcal{A}$  (see Lemma 5.4, Proposition 5.5 and Corollary 5.6 in \cite{alazard2018control}).

\begin{lemma} \label{lem:AB}
	There exists a universal constant $\delta > 0$ with the following properties.
	
	\begin{itemize} 
	\item[i.] Consider the case when the amplitude $q$ is of order zero in $k$ and is a perturbation of 1,  
	\[
	q(x,k) = 1 + \mathtt{b}(x,k)\,.
	\]
	Denote $|\mathtt{b}|_s :=\sup_{k \in \mathbb{Z}} \| \mathtt{b}(\cdot\,,k) \|_{H^s(\mathbb{T})}$. 
	If 
	\[
	\| b \|_{H^3(\mathbb{T})} + |\mathtt{b}|_{3} \leq \delta\,,
	\]
	then $\mathcal{A}$ and $\mathcal{A}^{\ast}$ are invertible from $L^2(\mathbb{T})$ onto itself, with 
	\[
	\| \mathcal{A} u \|_{L^2(\mathbb{T})} + \| \mathcal{A}^{-1} u \|_{L^2(\mathbb{T})} + \| \mathcal{A}^{\ast}\, u \|_{L^2(\mathbb{T})} + 
	\| (\mathcal{A}^{\ast})^{-1}u \|_{L^2(\mathbb{T})} \leq C \, \| u \|_{L^2(\mathbb{T})} \,,
	\]
	where $C>0$ is a universal constant. 
	
	\item[ii.] In the case when the amplitude $q$ is small, namely, if
	\[
	\| b \|_{H^3(\mathbb{T})} + |q|_{3} \leq \delta\,,
	\]
	then
	\[
	\| \mathcal{A} u \|_{L^2(\mathbb{T})} \le  C\delta  \, \| u \|_{L^2(\mathbb{T})} \,,
	\]
	where $C>0$ is a universal constant. 
	\end{itemize}
\end{lemma}

\begin{proposition} \label{prop:AB}
	Assume that $\| b \|_{W^{1,\infty}(\mathbb{T})} \leq 1/4$ and $\| b \|_{H^2(\mathbb{T})} \leq 1/2$. 
	Let 
	\[
	r,m,s_0 \in \mathbb{R}, \quad 
	m \geq 0, \quad
	s_0 > 1/2, \quad 
	M \in \mathbb{N}, \quad 
	M \geq 2(m + r + 1) + s_0. 
	\] 
	Then 
	\[
	|D_x|^r \mathcal{A} u = \sum_{\alpha=0}^{M-1} \mathrm{Op} \left( \frac{1}{i^\alpha \alpha!}
	\left(\partial_\xi^\alpha |\xi|^r\right)\partial_x^\alpha \left(q(x,k) e^{\mathrm{i} |k|^{1/2} b(x)}\right)\right)
	u
	+ R_M u,
	\]
	where, for every $s \geq s_0$, the remainder satisfies
	\begin{equation}\label{eq:RemAB}
		\| R_M |D_x|^m u \|_{H^s(\mathbb{T})} 
		\leq 
		C(s) \Big\{ \mathcal{K}_{2(m+r+s_0+1)} \, \| u \|_{H^s(\mathbb{T})} 
		+\mathcal{K}_{s + M + m + 2} \, \| u \|_{H^{s_0}(\mathbb{T})} \Big\},
	\end{equation}
	where $\mathcal{K}_\mu := | q-1 |_{\mu} + | q |_1  \| b \|_{H^{\mu + 1}(\mathbb{T})}$ and $| q |_\mu := \sup_t \sup_{k \in \mathbb{Z}} \| q(t,\cdot,k) \|_{H^\mu(\mathbb{T})}$. 
\end{proposition}

\begin{corollary} \label{cor:OpCalA}
	There exists a universal constant $\delta > 0$ with the following property. 
	Assume that 
	\[
	| q-1 |_{14} + \| b \|_{H^{14}(\mathbb{T})} \leq \delta,
	\]
	and let $\mathcal{A}$ be the operator $\mathcal{A} \coloneqq \mathrm{Op}\left( q(x,\xi) e^{\mathrm{i} |k|^{1/2} b(x)} \right)$. 
	For any $u$ in $L^2$, there holds
	\begin{align}
		& \mathrm{i} [ |D_x|^{3/2}, \mathcal{A} ] u \nonumber \\ 
		&= \frac{3}{2} (\partial_x b) \partial_x (\mathcal{A}u) 
		+ \mathrm{Op} \Big( \Big( \frac{3}{2} \frac{\xi}{|\xi|}\partial_x q 
		- \mathrm{i} \frac{9}{8}(\partial_x b)^2  q \Big) |\xi|^{1/2} e^{\mathrm{i}|\xi|^{1/2} b} \Big) u + R_{\mathcal{A}} u , \label{eq:CommCalA}
	\end{align}
	where $R_{\mathcal{A}}$ satisfies 
	\begin{align*}
	\| R_{\mathcal{A}} u \|_{ L^2(\mathbb{T}) } &\leq C \delta \| u \|_{ L^2(\mathbb{T}) }.
	\end{align*}
\end{corollary}

Now set
\begin{equation*}	
	\mathcal{N}_1 \coloneqq \| W \|_{C( [0,T];H^{s_0-d}(\mathbb{T}) )} 
	+ \| R_3 \|_{C( [0,T];L(L^2(\mathbb{T})) )},
\end{equation*}
where $s_0$ is the large enough integer which appears in the definition of $\mathcal{N}$ (see \eqref{eq:calN}) and $d$ is a number independent of $s_0$.

We now choose $b=b_0(t)+b_1(t,x)$ for some function coefficient $b_0(t)$ to be determined, and with $b_1=\frac{2}{3} \partial_x^{-1}W$. Then $b$ satisfies
\begin{equation*}
\frac{3}{2} \partial_x b = \frac{3}{2} \partial_x b_1=W.
\end{equation*}
Recall from \eqref{eq:R5} that 
\begin{align*} 
R_5 &\coloneqq \mathcal{A}^{-1}\Big(\left[ \partial_t,\mathcal{A}\right] + R_4 \mathcal{A} 
+W\partial_x  \mathcal{A}+ \mathrm{i} \left[ A_{0,\beta} ,\mathcal{A} \right] \Big).
\end{align*}
Now we split the last term as 
\begin{equation*}
	\mathrm{i} \left[ A_{0,\beta} ,\mathcal{A} \right] = \mathrm{i} \left[ |D_x|^{3/2}, \mathcal{A} \right]+ \mathrm{i} \left[ A_{0,\beta} -|D_x|^{3/2} , \mathcal{A} \right] .
\end{equation*}
Then it follows from the previous corollary that the remainder $R_5$ 
(as defined by \eqref{eq:R5}) satisfies

\begin{align} 
R_5 &\coloneqq \mathcal{A}^{-1} \Big( \left[ \partial_t,\mathcal{A}\right] 
	-\mathrm{Op}\Big(\Big(\frac{3}{2} \frac{\xi}{|\xi|} q_x -
	\mathrm{i} \frac{9}{8} b_x^2  q\Big)|\xi|^{3/2} 
	e^{\mathrm{i} |\xi|^{3/2} b}\Big) + R_4 \mathcal{A} \nonumber \\
	&\qquad + \mathrm{i} \left[ A_{0,\beta} -|D_x|^{3/2} ,\mathcal{A} \right] - R_{\mathcal{A}} \Big) . \label{eq:R5New}
\end{align}
where $R_{\mathcal{A}}$ is as given by Corollary \ref{cor:OpCalA}. 

Recall that $R_4$ is an operator of order $0$. Moreover,
\begin{equation*}
	\left[ \partial_t,\mathcal{A}\right]=\mathrm{Op}\Big(\big( q_t +\mathrm{i}|\xi|^{1/2} b_t q \big) e^{\mathrm{i}|\xi|^{1/2} b}\Big).	
\end{equation*}

So one can write $R_5$ as $R_5=R_5^{(1/2)}+R_5^{(0)}$, where $R_5^{(1/2)}$ (resp. $R_5^{(0)}$) is of order $1/2$ (resp. $0$).

Using that (see (5.13) in \cite{alazard2018control})
\begin{align*}
\| R_5^{(0)} \|_{C([0,T];\mathcal{L}(L^2(\mathbb{T}) ) )} &\leq C \mathcal{N}_1 ,
\end{align*}
the estimate for $R_{\mathcal{A}}$, the estimate $\| \partial_x W \|_{L^\infty(\mathbb{T})} \leq \| W \|_{H^2(\mathbb{T})}$ (it follows from Sobolev embedding), and \eqref{eq:EstW}-\eqref{eq:EstR3}, we obtain $\mathcal{N}_1 \leq C \mathcal{N}$ and hence 
\begin{align*}
	\| R_5^{(0)} \|_{C([0,T];\mathcal{L}(L^2))} &\leq C \mathcal{N} .
\end{align*}

We now have to show that $b$ and $q$ can be chosen such that $R_5^{(1/2)}=0$.

To do so, we fix $b_0(t)$ such that 
\begin{equation} \label{eq:b0choice}
2\pi \partial_t b_0 = -\int_{\mathbb{T}} \Big(\partial_t b_1+\frac{9}{8}(\partial_x b_1)^2\Big)(t,x) \, \mathrm{d}x,
\end{equation}
where we recall that $b_1=-\frac{2}{3}\partial_x^{-1}W$, 
so that
\begin{equation*}
	\int_{\mathbb{T}} \Big(\partial_t b+\frac{9}{8}(\partial_x b)^2 \Big)(t,x) \, \mathrm{d}x=0.
\end{equation*}

Now define $q$ as $q=e^{\mathtt{g}}$, where $\mathtt{g}$ is given by
\begin{align} \label{eq:ttg}
\mathtt{g} &= \frac{2}{3} \mathrm{i} \frac{\xi}{|\xi|}\partial_x^{-1}\Big(\partial_t b+\frac{9}{8}(\partial_x b)^2\Big).
\end{align}
(Notice that $\mathtt{g}$ is periodic in $x$.)  With this choice one has $R_5^{(1/2)}=0$.\\

By combining the previous results, we obtain the following result.

\begin{corollary} \label{cor:Reduction}
	Assume that $s_0$ is sufficiently large, and that the quantity $\mathcal{N}$ defined in \eqref{eq:calN} is sufficiently small. Let $\beta \in H^{s_0+1/2}(\mathbb{T})$. Let us consider the operator
	\begin{align*}
		\mathcal{A} &\coloneqq \mathrm{Op}\big(q(t,x,\xi)e^{\mathrm{i}b(t,x)|\xi|^{1/2}}\big)
	\end{align*}
	with
	\begin{align*}
		b=b_0(t)+\frac{2}{3}\partial_x^{-1}W ,
	\end{align*}	
	where $b_0$ is determined by \eqref{eq:b0choice}, and $q=e^{\mathtt{g}}$ where $\mathtt{g}$ is given by \eqref{eq:ttg}. 
	Then 
	\begin{align*}
		\left[ \partial_t +W\partial_x+\mathrm{i} A_{0,\beta} +R_4 \right] \mathcal{A} &= \mathcal{A} \left[ \partial_t + \mathrm{i} A_{0,\beta} + R_5 \right] ,
	\end{align*}
	where	
	\begin{align*}
		\| R_5 \|_{C([0,T];\mathcal{L}(L^2(\mathbb{T})) )}\leq C \mathcal{N} .
	\end{align*}

\end{corollary}

\section{Ingham type inequalities} \label{sec:Ingham}

The controllability of the linearized water waves system at the origin, namely 
\begin{equation*}
	\begin{cases}
		\eta_t &= G_0(\beta)\psi , \\
		\psi_t &= -g\eta + \kappa \eta_{xx} ,
	\end{cases}
\end{equation*}
is based on a modification of the following Ingham's inequality: for every $T>0$ there exist two positive constants $C_1=C_1(T)$ and $C_2=C_2(T)$ such that, 
for all $(w_n)_{n \in \mathbb{Z}} \in \ell^2(\mathbb{Z};\mathbb{C})$,
$$
C_1\sum_{n\in\mathbb{Z}}|w_n|^2
\le \int_{0}^T \bigg| \sum_{n\in\mathbb{Z}} w_n e^{i n |n|^{\frac12} t}\bigg|^2\, \mathrm{d}t
\le C_2\sum_{n\in\mathbb{Z}}|w_n|^2.
$$
Hereafter, $(w_n)_{n\in\mathbb{Z}}$ always refers to an arbitrary complex-valued
sequence in $\ell^2(\mathbb{Z})$.

We recall that for the case of finite depth and a flat bottom, one considers the sequence 
\begin{equation} \label{eq:mu0}
	\mu_0(n) = (g + \kappa n^2)^{1/2} |n|^{1/2} \tanh^{1/2}(h |n|), \quad n \in \mathbb{Z},
\end{equation}
of imaginary parts of the eigenvalues of the operator $\mathrm{i} A_0$, with $A_0$ given by \eqref{eq:AOp}, obtained by linearizing the water waves system with flat bottom around the origin.

For our purposes, we need to consider more general phases that do not depend linearly on $t$. 
For some given real-valued function $b \in C^3(\mathbb{R})$, set
\begin{equation} \label{eq:nu}
\nu_n(t) \coloneqq \mathrm{sgn}(n) \left[ \mu(n) t + b(t) |n|^{1/2} \right], 
\end{equation}
where $\nu_0=0$, $\mathrm{sgn}(n)=n/|n|$ for $n\neq 0$ and where $(\mu(n))_{n \in \mathbb{Z}}$ is the sequence of imaginary parts of the eigenvalues of the operator $\mathrm{i} A_{0,\beta}$, see \eqref{eq:AOp}-\eqref{eq:uSymmLin}.

In this section we fix $R>0$, and we consider $\beta \in B_{H^{s+1/2}}(R) \subset H^{s+1/2}(\mathbb{T})$ for which there exists $h_0>0$ for which \eqref{eq:StrConnAss} holds true.

We begin by giving an estimate for the elements of the sequence $(\mu(n))_{ n \in \mathbb{Z}}$.

\begin{lemma} \label{lem:BottomSpec}
	
	Let $s > 5/2$, and let $\beta \in H^{s+1/2}(\mathbb{T})$ be such that there exists $h_0>0$ such that \eqref{eq:StrConnAss} holds true. The operator $\mathrm{i} A_{0,\beta}$ defined in \eqref{eq:uSymmLin} has a discrete spectrum given by $( \mathrm{i} \mu(n))_{n \in \mathbb{Z}}$, and there exists $C>0$ such that
	\begin{align} \label{eq:BottomSpec}
		| \mu(n) - \mu_0(n) | &\leq C , \quad \text{as} \, |n| \to +\infty.
	\end{align}
	
\end{lemma}

\begin{proof}
	
	We first recall that for any $\psi \in H^s(\mathbb{T})$
	\begin{align*}
		G_0(\beta)\psi &= ( G_{00} + D L(\beta)) \psi, \quad G_{00} = D \tanh(h D) , 
	\end{align*}
	and that by Proposition \ref{prop:DNOpara} (see also Proposition 3.14 in \cite{alazard2011water}) we have
	\begin{align*}
		G_0(\beta)\psi = T_{\lambda}\psi + \mathcal{R}\psi, \quad \lambda(x,\xi) = |\xi| ,
	\end{align*}
	where
	\begin{align*}
		\| \mathcal{R}\psi \|_{ H^{s+1/2}(\mathbb{T}) } &\leq C \left( h,h_0,\|\beta\|_{H^{s+1/2}(\mathbb{T})} \right) \, \|\psi\|_{H^{s}(\mathbb{T})} .
	\end{align*}
	Therefore we obtain that
	\begin{align*}
		D L(\beta) &= \underbrace{ (T_{\lambda} - G_{00}) }_{ \in OPS^{-\infty} } + \underbrace{ \mathcal{R} }_{ \in OPS^{-1/2} } ,
	\end{align*}
	namely, $D L(\beta)$ is a smoothing operator (of order $-\frac{1}{2}$), and hence compact.
	
	Since $A_0$ is self-adjoint and elliptic of order $\frac{3}{2}$, and since $G_{00}^{-1} DL(\beta)$ is smoothing (and hence compact), it follows by Proposition 2.1 in \cite{komornik2000observability} that the operator $\mathrm{i} A_{0,\beta}$ has discrete point spectrum. 
	
	Moreover, the eigenvalues of $\mathrm{i} A_{0,\beta}$ are purely imaginary: indeed, from \eqref{eq:AOp} and \eqref{eq:uSymmLin} we have
	\begin{equation*}
		A_{0,\beta} = A_0 \left[ 1+ G_{00}^{-1} D L(\beta) \right] = A_0 G_{00}^{-1} G_0(\beta) = (g+\kappa D^2)^{1/2} G_{00}^{-1/2} G_0(\beta) .
	\end{equation*}
	Now, $A_{0,\beta}$ may not even be self-adjoint because in general $G_0(\beta)$ does not commute with the Fourier multipliers $G_{00}$ and $g+\kappa D^2$. However, $A_{0,\beta}$ is similar to the operator
	\begin{equation*}
	(g+\kappa D^2)^{1/4} G_{00}^{-1/4} G_0(\beta) (g+\kappa D^2)^{1/4} G_{00}^{-1/4},
	\end{equation*}
	which is self-adjoint and positive, since $(g+\kappa D^2)^{1/2} G_{00}^{-1/2}$ and $G_0(\beta)$ are both self-adjoint and positive.
	This in turn implies that the spectrum of $\mathrm{i} A_{0,\beta}$ consists of purely imaginary eigenvalues.	
	
	We denote the eigenvalues of $\mathrm{i} A_{0,\beta}$ by $\mathrm{i} \mu(n)$, $n \in \mathbb{Z}$: by (2.13) in \cite{komornik2000observability} we have
	\begin{align*}
		\lim_{|n| \to +\infty} | \mu(n) | = +\infty .
	\end{align*}
	Moreover, using classical arguments (see \cite{taylor2013partial}, Ch.7, $\S\S$ 1–4 and $\S$ 10), we have that the eigenvalues of the operator $\mathrm{i} A_0$ satisfy
	\begin{align*}	
		\mu_0(n) &\sim c \, |n|^{3/2} ,
	\end{align*}
	and the same leading-order asymptotics hold also for $\mathrm{i} A_{0,\beta}$.
		
	Moreover, we have the following bound which is uniform in $n$,
	\begin{align*}
		| \mu(n) -  \mu_0(n))| &\leq \| A_0 G_{00}^{-1}DL(\beta) \|_{ L^2(\mathbb{T}) \to L^2(\mathbb{T}) } ,
	\end{align*}
	from which we can deduce the estimate \eqref{eq:BottomSpec}.
	
\end{proof}

Now we begin by proving a lower bound which holds for any $T>0$, under the assumption that the functions contain only large enough frequencies.

\begin{proposition} \label{prop:HighFreq}
	Let $T > 0$. Let $s> \frac{5}{2}$, and let $\beta \in H^{s+1/2}(\mathbb{T})$ be such that there exists $h_0>0$ such that \eqref{eq:StrConnAss} holds true. Let $| b_t | \leq \frac{1}{2} \tanh^{1/2}(h)$ and $| b_{tt} | \leq 1$ 
	for all $t \in [0,T]$. 
	
	Then there exists $N_0 \geq 0$ such that, for all $N \geq N_0$,
	\begin{align} \label{eq:ObsHigh}
	\frac{T}{2}\sum_{\substack{n\in\mathbb{Z} \\ |n| \geq N}} |w_n|^2 &\leq \int_0^T 
	| \sum_{\substack{n\in\mathbb{Z} \\ |n| \geq N}} w_n e^{\mathrm{i}\nu_n(t)} |^2 \, \mathrm{d}t , \quad \forall (w_n)_{n \in \mathbb{Z}} \in \ell^2(\mathbb{Z};\mathbb{C}) .
	\end{align}
\end{proposition}

\begin{remark} \label{rem:InghamHigh}
	\begin{itemize}
		\item[i.] For $T$ small, one can take $N_0 = C T^{-2-\varepsilon}$ for some $\varepsilon>0$. 
		\item[ii.] For $\| b_tt \|_{L^\infty}$ sufficiently small and for $T$ sufficiently large, the above Proposition \ref{prop:HighFreq} holds with $N_0 = 0$.
		\item[iii.] We have that $N_0 = N_0(s,R,T;h,h_0)$, where $h_0>0$ is the constant appearing in \eqref{eq:StrConnAss}.
	\end{itemize}
\end{remark} 

\begin{proof}
	
	Splitting the sum into $n=m$ and $n \neq m$, we write
	\begin{equation*}
	\int_0^T 
	\bigg| \sum_{\substack{n\in\mathbb{Z} \\ |n| \geq N}} w_n e^{\mathrm{i}\nu_n(t)} \bigg|^2\, \mathrm{d}t
	\geq 
	T \sum_{\substack{n\in\mathbb{Z} \\ |n| \geq N}} |w_n|^2 
	+ \sum_{\substack{n\neq m\\ |m| , |n| \geq N}} w_n \overline{w_m} \, \int_0^T e^{\mathrm{i}(\nu_n(t)-\nu_m(t))}\, \mathrm{d}t .
	\end{equation*}
	We have to estimate
	\begin{equation*}
		K(n,m) \coloneqq \int_0^T e^{\mathrm{i}(\nu_n(t)-\nu_m(t))}\, \mathrm{d}t.
	\end{equation*}
	Integrating by parts, we have
	\begin{equation*}
	K(n,m) = \bigg[ \frac{ e^{\mathrm{i}(\nu_n(t)-\nu_m(t))} }{\mathrm{i} (\nu_n'(t) - \nu_m'(t)) } \bigg]_{t=0}^{t=T}
	+ \int_0^T e^{\mathrm{i}(\nu_n(t)-\nu_m(t))} \frac{\nu_n'' - \nu_m''}{\mathrm{i}(\nu_n' - \nu_m')^2} \, \mathrm{d}t,
	\end{equation*}
	and therefore 
	\begin{equation*}
	|K(n,m)| \leq \tilde{\kappa}(n,m) := 
	\bigg\| \frac{2}{\nu_n' - \nu_m'} \bigg\|_{L^\infty([0,T])} 
	+ \int_0^T \frac{| \nu_n'' - \nu_m'' |}
	{|\nu_n' - \nu_m' |^2}\, \mathrm{d}t.
	\end{equation*}
	Now let us define the auxiliary functions
	\begin{equation*} 
		\nu_{0,n}(t) \coloneqq \mathrm{sgn}(n) \left[ \mu_0(n) t + b(t) |n|^{1/2} \right] ,
	\end{equation*}
	and observe that by Lemma \ref{lem:BottomSpec}, and in particular by \eqref{eq:BottomSpec}, we have
	\begin{align*}
	|\nu'_n(t) - \nu'_{0,n}(t)| &\leq C + |b'(t)| |n|^{1/2} , \quad \forall |n| \geq N , \, \forall t \in [0,T], \\
	\nu''_n(t) &= \nu''_{0,n}(t) , \quad \forall n \in \mathbb{Z}, \, \forall t \in [0,T].
	\end{align*}	
	%Moreover, there exists $N_{\ast}=N_{\ast}(T)$ such that if $N \geq N_{\ast}$, then
	%\begin{align*}
	%	|\nu'_n(t) - \nu'_{0,n}(t)| &\geq \frac{1}{2}   |b'(t)|  |n|^{1/2} , \quad \forall |n| \geq N , \, \forall t \in [0,T].
	%\end{align*}	
	Hence
	%\begin{align*}
	%	K(n,m) &\leq \tilde{\kappa}(n,m) \\
	%	&\coloneqq \bigg\| \frac{4}{|b'(t)| ( |n|^{1/2} - |m|^{1/2} ) } \bigg\|_{L^\infty([0,T])} 
	%	+ \int_0^T \frac{4| \nu_{0,n}'' - \nu_{0,m}'' |}
	%	{|b'(t)|^2 ( |n|^{1/2} - |m|^{1/2} )^2}\, \mathrm{d}t,
	%\end{align*}
	%and
	\begin{equation*}
		\int_0^T
		\bigg| \sum_{\substack{n\in\mathbb{Z} \\ |n| \geq N}} w_n e^{\mathrm{i}\nu_n(t)}\bigg|^2\, \mathrm{d}t
		\geq \sum_{\substack{n\in\mathbb{Z} \\ |n| \geq N}} \Big( T - \sum_{\substack{m\in \mathbb{Z}\setminus \{n\} \\ |m| \geq N}} \tilde{\kappa}(n,m) \Big) |w_n|^2.
	\end{equation*}
	We have to prove that we can choose $N$ such that
	\begin{equation} \label{eq:EstTildekappa}
		T - \sum_{\substack{m\in \mathbb{Z}\setminus \{n\} \\ |m| \geq N}} \tilde{\kappa}(n,m) \geq \frac{T}{2}.
	\end{equation}
	Arguing as in Lemma 6.3 of \cite{alazard2018control}, one can prove the following 
	
	\begin{itemize}
	\item Claim: assume that $| b_t | \leq \frac{1}{2} \tanh^{1/2}(h)$ for all $t$. Let $0 < \delta < \frac{1}{2}$. Then 
	
	\begin{itemize}
		\item[i.] There exists a positive constant $K_\delta$ such that, 
	for all integers $N \geq 0$ and all $n \in \mathbb{Z}$ with $|n| \geq N$, 
	\begin{equation*}
		\sum_{\substack{m\in \mathbb{Z}\setminus \{n\}\\ |m| \geq N}}
		\bigg\| \frac{1}{\nu'_n - \nu'_m} \bigg\|_{L^\infty([0,T])} 
		\le \frac{K_\delta}{(1+N)^{\frac{1}{2}-\delta}}\,.
	\end{equation*}
	
		\item[ii.] For all integers $n,m$ in $\mathbb{Z}$ with $n\neq m$, and all $t$,
	\begin{equation*}
		\frac{| \nu''_n-\nu''_m |}{|\nu'_n-\nu'_m|}\le 2\tanh^{-1/2}(h)\, | b_{tt} | .
	\end{equation*}
	\end{itemize}
	
	\end{itemize}
	
	The previous lemma and the definition of $\kappa(n,m)$ imply that
	\begin{equation*}
	\sum_{\substack{m \in \mathbb{Z}\setminus \{n\} \\ |m| \geq N}} \tilde{\kappa}(n,m) 
	\leq 
	\frac{2 K_\delta}{(1+N)^{\frac{1}{2} - \delta}}\,(1 + T \|  b_{tt} \|_{L^\infty}).
	\end{equation*}
	Hence \eqref{eq:EstTildekappa} is satisfied, provided that
	\begin{equation} \label{eq:EstFourierTrunc}
		\frac{4 K_\delta}{T} \big( 1 + T \| b_{tt} \|_{L^\infty} \big) 
		\leq (1+N)^{\frac12 - \delta},
	\end{equation}
	and we can deduce \eqref{eq:ObsHigh}.

\end{proof}

For solutions supported only at low frequencies, following Sec. 1 of \cite{micu2004introduction}, observability (which we will prove in Sec. \ref{sec:Obs}) follows from the injectivity of the map
\begin{equation} \label{eq:Psi1Map}
	\Psi_1: (w_n)_{|n| \leq N} \mapsto \sum_{|n| \leq N} w_n e^{\mathrm{i} \nu_n(t)} .
\end{equation}
To prove this injectivity, we first prove that for ``many" choices of the bottom topography the map 
\begin{equation} \label{eq:Psi2Map}
	\Psi_2: (w_n)_{|n| \leq N} \mapsto \sum_{|n| \leq N} w_n e^{\mathrm{i} \mu(n) \, t} ,
\end{equation}
where $(\mu(n))_{n}$ denotes the eigenvalues of $A_{0,\beta}$, is also injective.

\begin{lemma} \label{lem:LowFreq}
	
Let $R>0$, let $s > 3$, and let $N$ be such that \eqref{eq:ObsHigh} holds true. There exists an open subset $\mathcal{G} \subseteq B_{H^{s+1/2}}(R)$ which is dense in $B_{H^{s+1/2}}(R)$ such that the following holds true: for any $\beta \in \mathcal{G}$ satisfying \eqref{eq:StrConnAss} there exists $\delta=\delta(\beta)>0$ such that the eigenvalues of the operator $A_{0,\beta}$ defined in \eqref{eq:uSymmLin} satisfy
\begin{align} \label{eq:distinct}
	| \mu(n)-\mu(m) | &> \delta , \quad \forall |n|, |m| \leq N\ \text{with} \; n\neq m.
\end{align}
		
\end{lemma}

\begin{proof}

First, since \eqref{eq:distinct} involves only a finite number of eigenvalues of $A_{0,\beta}$, it suffices to show that $(\mu(n))_{|n| \leq N}$ are distinct. We also notice that $A_{0,\beta}$ has always zero as a simple eigenvalue (since the fluid domain is connected), with the constant $1$ as associated normalized eigenfunction.

Hence, in the following we consider $A_{0,\beta}$ as an operator on $H^s_0(\mathbb{T})$: recall that 
\begin{align*}
	A_{0,\beta} &= (g+\kappa D^2)^{1/2} G_{00}^{-1/2} G_0(\beta) .
\end{align*}
	
Now, using that 
\begin{align*}
	P &\coloneqq (g+\kappa D^2)^{1/2} G_{00}^{-1/2}
\end{align*}
is a positive self-adjoint operator, we have that $A_{0,\beta} = P G_0(\beta)$ satisfies
\begin{align*}
	P^{-1/2} A_{0,\beta} P^{1/2} &= P^{1/2} G_0(\beta) P^{1/2} \eqqcolon B_{0,\beta} ,
\end{align*}
namely, $A_{0,\beta}$ is similar to the self-adjoint operator $B_{0,\beta}$. Therefore, $A_{0,\beta}$ and $B_{0,\beta}$ have the same eigenvalues (with the same algebraic multiplicities), and it suffices to prove the thesis for the self-adjoint operator $B_{0,\beta}$.

Since $\beta \in H^{s+1/2}(\mathbb{T})$ with $s > 3$, in the following we use the formula for the shape derivative of the Dirichlet--Neumann operator $G_0(\beta)$ with respect to the bottom variations (see Proposition B.3 ii. of \cite{pasquali2026two}; see also Theorem 3.5 of \cite{iguchi2011mathematical} for more details): we have
\begin{align} \label{eq:BottomDerDNO}
	\partial_{\beta} G_0(\beta)(\check{\beta})\psi  &= - G^{NN}(0,\beta)\left[ \partial_x \left( \check{\beta} \, w(0,\beta)(\psi,0)  \right) \right] ,
\end{align}
where
\begin{align*}
	w(0,\beta)(\psi,0) &\coloneqq \partial_x \left( G^{DD}(0,\beta)\psi \right) - W_b(0,\beta)(\psi,0) \, \beta_x, \\
	W_b(0,\beta)(\psi,0) &\coloneqq \frac{1}{1+\beta_x^2} \left[  \beta_x \, \partial_x \left( G^{DD}(0,\beta)\psi  \right) \right] , \\
	G^{DD}(0,\beta)\psi &\coloneqq \varphi|_{y=-h+\beta} ,
\end{align*}
and where $G^{NN}(0,\beta)$ is a bounded operator on $H^{s}(\mathbb{T})$ whose adjoint on $L^2(\mathbb{T})$ is given by $-G^{DD}(0,\beta)$, see Proposition B.2 and Proposition B.6 in \cite{pasquali2026two}.

In order to make the argument easier to understand, in the rest of the proof we denote by $( \mu(n;\check{\beta}) )_{n \in \mathbb{Z}}$ the eigenvalues of $B_{0,\check{\beta}}$, for any $\check{\beta} \in H^{s+1/2}(\mathbb{T})$.

Given $0 < \varepsilon \ll 1$ and $\beta_0 \in B_{H^{s+1/2}}(R)$, we want to find $\beta=\beta_{\varepsilon} \in H^{s+1/2}(R)$ with 
\begin{equation*}
	\| \beta_0 - \beta \|_{ H^{s+1/2}(\mathbb{T}) } \leq \varepsilon ,
\end{equation*}
such that the eigenvalues $( \mu(n;\beta) )_{|n| \leq N}$ of the operator $B_{0,\beta}$ are distinct.

If the eigenvalues $( \mu(n;\beta_0) )_{|n| \leq N}$ of $B_{0,\beta_0}$ are distinct, then we can deduce \eqref{eq:distinct} by choosing $\beta=\beta_0$. 

Otherwise, among the eigenvalues $( \mu(n;\beta_0) )_{|n| \leq N}$ of $B_{0,\beta_0}$ there exist some repeated eigenvalues $\mu_1^{\ast},\ldots,\mu_\mathtt{n}^{\ast}$ ($\mathtt{n} \leq N$) with multiplicity (restricted to $|n| \leq N$) given by $m_1,\ldots,m_{\mathtt{n}}$ respectively. The proof consists in suitably perturbing the bottom such that each repeated eigenvalue splits into distinct eigenvalues; the procedure is based on a finite number of steps, and we choose the bottom perturbation to be increasingly smaller in $H^{s+1/2}$-norm at every step, in order to ensure that eigenvalues which split in the previous steps do not collide again in subsequent steps. 

In order to avoid that the perturbed eigenvalues collide with eigenvalues which are initially distinct, we require that
\begin{equation*}
	\varepsilon < \frac{1}{2} \delta_0 , \quad \delta_0 \coloneqq  \min_{ \substack{  \lambda_i,\lambda_j \in ( \mu(n;\beta_0) )_{|n| \leq N} \\ \lambda_i , \lambda_j \text{distinct} } } | \lambda_i - \lambda_j |   .
\end{equation*}

\begin{itemize}
	\item \textbf{Step 1}: we study the splitting of the eigenvalue $\mu_1^{\ast}$.
\end{itemize}
For simplicity, let us denote the multiplicity (restricted to $|n| \leq N$) of $\mu_1^{\ast}$ by $m \in \{ 2, \ldots,2N \}$, and with orthonormal associated eigenfunctions $u_{1,1},\ldots,u_{1,m} \neq 0$, namely
\begin{align*}
	B_{0,\beta_0} u_{1,i} &= \mu^{\ast} u_{1,i}, \quad i=1,\ldots,m.
\end{align*}

If we define
\begin{align*}
	f_{1,i} &\coloneqq P^{1/2} u_{1,i} , \quad i=1,\ldots,m,
\end{align*}
then 
\begin{align*}
	G_0(\beta) f_{1,i} &= \mu_1^{\ast} \, P^{-1} f_{1,i}, \quad i=1,\ldots,m.	
\end{align*}

Now, for any $f_{1,i} \in H^s_0(\mathbb{T})$ we denote by $\varphi_{1,i}$ the solution of the elliptic boundary-value problem \eqref{eq:EllBVP} with $\psi=f_{1,i}$.

%Let us first assume that $m \geq 3$; we will discuss the case $m=2$ later.

\begin{itemize}
	\item[-] \textbf{Claim 1}: there exist $\varepsilon_1 \in ( 0,\varepsilon )$ and $\check{\beta}_1 \in H^{s+1/2}(\mathbb{T})$ with 
	\begin{equation*}
		\| \beta_0 - \beta_1 \|_{ H^{s+1/2}(\mathbb{T}) } \leq \varepsilon , \quad \beta_1 \coloneqq \beta_0 + \varepsilon_1 \check{\beta}_1 ,
	\end{equation*}
	such that the eigenvalues $( \mu(n;\beta_1) )_{|n| \leq N}$ of the operator $B_{0,\beta_1}$ which split from $\mu_1^{\ast}$ have multiplicities (restricted to $|n| \leq N$) at most $m-1$.
\end{itemize}

We now prove Claim 1. We argue by contradiction: assume that for any $0 < \varepsilon_1 < \varepsilon$ and for any $\check{\beta}_1 \in H^{s+1/2}(\mathbb{T})$ the eigenvalue $\mu_1^{\ast}$ does not split into eigenvalues for the operator $B_{0,\beta_1}$ with multiplicity smaller than $m$.

Then the correction of order $\mathcal{O}(\varepsilon_1)$ to $\mu_1^{\ast}$ under the bottom perturbation $\check{\beta}_1$ is described by the spectrum of the following $m\times m$ matrix (see \cite{kato1966perturbation}, Ch.2, $\S$ 5.4, Theorem 5.4)
\begin{align} 
M_1(\check{\beta}_1) &= ( M_{1,i,j}(\check{\beta}_1) )_{1 \leq i,j \leq m} , \label{eq:splitting1} \\
M_{1,i,j}(\check{\beta}_1) &\coloneqq \langle \partial_{\beta}B_{0,\beta_0}(\check{\beta}_1) u_{1,i},u_{1,j} \rangle_{ L^2(\mathbb{T}) } = \langle  \partial_{\beta}G_0(\beta_0)(\check{\beta}_1)  f_{1,i},f_{1,j} \rangle_{ L^2(\mathbb{T}) } . \nonumber
\end{align}

Using formulae (3.9)-(3.10) of \cite{lannes2013water}, we have that 
\begin{align*}
M_{1,i,j}
&=
\int_{ \{ y =-h+\beta_0 \} }
\check{\beta}_1 \, F_{\beta_0} \,
\partial_\tau \varphi_{1,i} \, \partial_\tau \varphi_{1,j} \, ,
\end{align*}
where $F_{\beta_0} >0 $, and where $\partial_\tau$ denotes the tangential derivative. 

\begin{align} \label{eq:splitting2}
M_{1,i,j}(\check{\beta})
&=
\int_{ \{ y =-h+\beta_0 \} }
\check{\beta}_1 \, F_{\beta_0} \,
\partial_\tau \varphi_{1,i} \, \partial_\tau \varphi_{1,j} \, ,
\end{align}

Since the eigenvalue $\mu_1^{\ast}$ does not split into eigenvalues with smaller algebraic multiplicity for any perturbation $\check{\beta}_1$, and since the operator $B_{0,\beta_1}$ is self-adjoint for any small $\varepsilon_1$ and for any $\check{\beta}_1$, then $M_1(\check{\beta}_1)$ is a symmetric matrix with one eigenvalue of multiplicity $m$, namely
\begin{equation*}
	M_1(\check{\beta}_1) = c(\check{\beta}_1) \, \mathbb{I}_m .
\end{equation*}
Since $\check{\beta}_1$ is arbitrary, this implies that
\begin{equation*}
\partial_\tau \varphi_{1,i} \, \partial_\tau \varphi_{1,j} =0 , \quad i \neq j,
\end{equation*}
pointwise on the bottom of the fluid domain; in turn, this implies 
\begin{equation*}
\partial_\tau \varphi_{1,i} =0 , \quad \text{at} \, y=-h+\beta_0 , \quad \forall i=1,\ldots,m.
\end{equation*}
Recalling that the boundary condition at the bottom of \eqref{eq:EllBVP} already gives
\begin{equation*}
\partial_{n}\varphi_{1,i} =0 , \quad \text{at} \, y=-h+\beta_0 , \quad \forall i=1,\ldots,m ,
\end{equation*}
we have that both the normal and the tangential derivatives of $\varphi_{1,i}$ vanish at the
bottom. By unique continuation for harmonic functions, we have that $\varphi_{1,i}$ must be constant on the fluid domain $D_{0,\beta_0}$. Hence, each trace (at the free surface) $f_{1,i}$ is constant; actually, we have $f_{1,i} =0$, since $f_{1,i}$ has zero average. Using that $P^{1/2}$ is invertible, this implies that $u_{1,i} = 0$, which leads to contradiction.

Therefore, by choosing $0 < \varepsilon_1 \ll 1$ and a suitable bottom topography perturbation $\check{\beta}_1$, the eigenvalue $\mu_1^{\ast}$ of multiplicity (restricted to $|n| \leq N$) $m$ for $B_{0,\beta_0}$ (and hence of $A_{0,\beta_0}$) splits into eigenvalues of $B_{0,\beta_1}$ with multiplicity (restricted to $|n| \leq N$) at most $m-1$, and Claim 1 is proved. \\

Next, assume that there exists an eigenvalue $\mu_{1,1}^{\ast}$ for $B_{0,\beta_1}$ which is $\mathcal{O}(\varepsilon_1)$-close to $\mu_1^{\ast}$ and has multiplicity (restricted to $|n| \leq N$) exactly $m-1$ (if $B_{0,\beta_1}$ has only eigenvalues with multiplicity at most $m-2 \geq 2$, we pass to Claim 3). We denote the (non-zero) eigenfunctions associated to $\mu_{1,1}^{\ast}$ by $( u^{[1]}_{1,i} )_{i=1,\ldots,m-1}$.

Then
\begin{itemize}
	\item[-] \textbf{Claim 2}: there exist $\varepsilon_2 \in ( 0,\varepsilon_1 )$ and $\check{\beta}_2 \in H^{s+1/2}(\mathbb{T})$ with 
	\begin{equation*}
		\| \beta_1 - \beta_2 \|_{ H^{s+1/2}(\mathbb{T}) } \leq \varepsilon_1 , \quad \beta_2 \coloneqq \beta_0 + \varepsilon_1 \check{\beta}_1 + \varepsilon_2 \check{\beta}_2 ,
	\end{equation*}
	such that the eigenvalues $( \mu(n;\beta_2) )_{|n| \leq N}$ of the operator $B_{0,\beta_2}$ which split from $\mu_{1,1}^{\ast}$ have multiplicity (restricted to $|n| \leq N$) at most $m-2$.
\end{itemize}

Now we prove Claim 2. We argue by contradiction: assume that for any $0 < \varepsilon_2 < \varepsilon_1$ and for any $\check{\beta}_2 \in H^{s+1/2}(\mathbb{T})$ the eigenvalue $\mu_{1,1}^{\ast}$ does not split into eigenvalues for the operator $B_{0,\beta_2}$ with multiplicity smaller than $m-1$.

Then the correction of order $\mathcal{O}(\varepsilon_2)$ to $\mu_{1,1}^{\ast}$ under the bottom perturbation $\check{\beta}_2$ is described by the spectrum of the $(m-1)\times (m-1)$ matrix
\begin{align*} 
	M_2(\check{\beta}_2) &= ( M_{2,i,j}(\check{\beta}_2) )_{1 \leq i,j \leq m-1} , \\
	M_{2,i,j}(\check{\beta}_2) &\coloneqq \langle \partial_{\beta} B_{0,\beta_1}(\check{\beta}_2) u^{[1]}_{1,i} , u^{[1]}_{1,j} \rangle_{ L^2(\mathbb{T}) } .
\end{align*}

Since the eigenvalue $\mu_{1,1}^{\ast}$ does not split into eigenvalues with smaller algebraic multiplicity for any perturbation $\check{\beta}_2$, and since the operator $B_{0,\beta_2}$ is self-adjoint for any small $\varepsilon_2$ and for any $\check{\beta}_2$, then $M_2(\check{\beta}_2)$ is a symmetric matrix with one eigenvalue of multiplicity $m-1$, namely
\begin{equation*}
	M_2(\check{\beta}_2) = c(\check{\beta}_2) \, \mathbb{I}_{m-1} .
\end{equation*}
Since $\check{\beta}$ is arbitrary, we can conclude as in the proof of Claim 1 that $u^{[1]}_{1,i} = 0$ for any $i=1,\ldots,m-1$, which is a contradiction.

Therefore, by choosing $0 < \varepsilon_2 < \varepsilon_1$ and a suitable bottom topography perturbation $\check{\beta}_2$, the eigenvalue $\mu_{1,1}^{\ast}$ of multiplicity (restricted to $|n| \leq N$) $m-1$ for $B_{0,\beta_1}$ (and hence of $A_{0,\beta_1}$) splits into eigenvalues of $B_{0,\beta_2}$ with multiplicity (restricted to $|n| \leq N$) at most $m-2$, and Claim 2 is proved. \\

Similarly, one can prove the following inductive step.
\begin{itemize}
	\item[-] \textbf{Claim 3}: let $k \in \{ 2, \ldots, m-2 \}$, and let us assume that there exist 
	\begin{equation*}
		0 < \varepsilon_k < \ldots < \varepsilon_1 < \varepsilon ,
	\end{equation*}
	and $\check{\beta}_1, \ldots, \check{\beta}_k \in H^{s+1/2}(\mathbb{T})$ with 
	\begin{equation*}
		\| \beta_{h-1} - \beta_h \|_{ H^{s+1/2}(\mathbb{T}) } \leq \varepsilon_{h-1} , \quad \beta_h \coloneqq \beta_0 + \sum_{j=1}^h \varepsilon_j \check{\beta}_j , \quad h=1,\ldots,k,
	\end{equation*}
	such that the eigenvalues $( \mu(n;\beta_k) )_{|n| \leq N}$ of the operator $B_{0,\beta_k}$ which split from $\mu_{1,k-1}^{\ast} \in \mathrm{spec}(B_{0,\beta_{k-1}})$  have multiplicity (restricted to $|n| \leq N$) at most $m-k$.
		
	Then there exist $\varepsilon_{k+1} \in ( 0,\varepsilon_k )$ and $\check{\beta}_{k+1} \in H^{s+1/2}(\mathbb{T})$ with 
	\begin{equation*}
		\| \beta_k - \beta_{k+1} \|_{ H^{s+1/2}(\mathbb{T}) } \leq \varepsilon_{k} , \quad \beta_{k+1} \coloneqq \beta_0 + \sum_{j=1}^{k+1} \varepsilon_j \check{\beta}_j ,
	\end{equation*}	
	such that the eigenvalues $( \mu(n;\beta_{k+1}) )_{|n| \leq N}$ of the operator $B_{0,\beta_{k+1}}$ which split from have multiplicity (restricted to $|n| \leq N$) at most $m-k-1$.
	
	In particular, if $k=m-2$, then the eigenvalues $( \mu(n;\beta_{m-1}) )_{|n| \leq N}$ of the operator $B_{0,\beta_{m-1}}$ which split from $\mu_1^{\ast} \in \mathrm{spec}( B_{0,\beta_0} )$ are distinct.
\end{itemize}

\begin{itemize}
	\item \textbf{Step 2}: we study the splitting of the (possibly) repeated eigenvalue
	\begin{equation*}
		\mu_2^{\star} = \mu_2^{\ast} + \sum_{h=1}^{m-1}\mathcal{O}(\varepsilon_h).
	\end{equation*}
	 
\end{itemize}

We notice that (provided that the $\varepsilon_h$ are sufficiently small), then the multiplicity (restricted to $|n| \leq N$) of the eigenvalue $\mu_2^{\star}$ for the operator $B_{0,\beta_{m-1}}$ is given by $m_2^{\star} \leq m_2$. If $m_2^{\star} = 1$, then we study the splitting of 
\begin{equation*}
	\mu_3^{\star} = \mu_3^{\ast} + \sum_{h=1}^{m-1}\mathcal{O}(\varepsilon_h).
\end{equation*}

Otherwise, we proceed as in Step 1 in order to study the splitting of $\mu_2^{\star}$, denoting its orthonormal associated eigenfunctions by $u_{2,1},\ldots,u_{2,m_2} \neq 0$, namely
\begin{align*}
	B_{0,\beta_{m-1}} u_{2,i} &= \mu_2^{\star} u_{2,i}, \quad i=1,\ldots,m_2.
\end{align*}

Repeating the procedure of Claim 1, Claim 2 and Claim 3 we construct a finite sequence
\begin{equation*}
	0 < \varepsilon_{m-1+m_2-2} < \ldots < \varepsilon_m < \varepsilon_{m-1} ,
\end{equation*}
and $\check{\beta}_m, \ldots, \check{\beta}_{m_1+m_2-2} \in H^{s+1/2}(\mathbb{T})$ with 
\begin{equation*}
	\| \beta_{h-1} - \beta_h \|_{ H^{s+1/2}(\mathbb{T}) } \leq \varepsilon_{h-1} , \quad \beta_h \coloneqq \beta_0 + \sum_{j=1}^{m_2-1} \varepsilon_{m-1+j} \check{\beta}_{m-1+j} , \quad h=1,\ldots,m_2-1,
\end{equation*}
such that the eigenvalues $( \mu(n;\beta_h) )_{|n| \leq N}$ of the operator $B_{0,\beta_h}$ have algebraic multiplicity at most $m-h$.

Then there exist $\varepsilon_{k+1} \in ( 0,\varepsilon_k )$ and $\check{\beta}_{k+1} \in H^{s+1/2}(\mathbb{T})$ with 
\begin{equation*}
	\| \beta_k - \beta_{k+1} \|_{ H^{s+1/2}(\mathbb{T}) } \leq \varepsilon_{k} , \quad \beta_{k+1} \coloneqq \beta_0 + \sum_{j=1}^{m+m_2-2} \varepsilon_j \check{\beta}_j ,
\end{equation*}	
such that the eigenvalues $( \mu(n;\beta_{k+1}) )_{|n| \leq N}$ of the operator $B_{0,\beta_{k+1}}$ which split from $\mu_{2,k}^{\star} \in \mathrm{spec}( B_{0,\beta_{k+1}} )$ have multiplicity (restricted to $|n| \leq N$) at most $m_2-k-1$.

In particular, the eigenvalues $( \mu(n;\beta_{m+m_2-2}) )_{|n| \leq N}$ of the operator $B_{0,\beta_{m-1}}$ which split from $\mu_2^{\star} \in \mathrm{spec}( B_{0,\beta_{m-1}} )$ are distinct.

\begin{itemize}
	\item \textbf{Step 3}: repeat iteratively Step 2 for the splitting of the (possibly) repeated eigenvalues
	\begin{equation*}
		\mu_h^{\star} = \mu_h^{\ast} + \sum_{j=1}^{m_1+\ldots+m_{h-1}-h} \mathcal{O}(\varepsilon_j) , \quad h=1,\ldots,\mathtt{n} .
	\end{equation*}
	
\end{itemize}

If we define
\begin{align*}
	\delta_k &\coloneqq  \min_{ \substack{  \lambda_i,\lambda_j \in ( \mu(n;\beta_k) )_{|n| \leq N} \\ \lambda_i , \lambda_j \text{distinct} } } | \lambda_i - \lambda_j |  , \quad  k=1,\ldots, \left( \sum_{h=1}^{\mathtt{n}} m_h \right) -\mathtt{n} -1 ,
\end{align*}
the thesis follows by choosing
\begin{align*}
	\varepsilon_k &< \min \left( \frac{\varepsilon}{2^k \, \| \check{\beta}_k \|_{ H^{s+1/2}(\mathbb{T}) }  }  , \frac{1}{2} \delta_{k-1} \right) , \quad k=1,\ldots, \left( \sum_{h=1}^{\mathtt{n}} m_h \right) -\mathtt{n} < 2N, \\
	\beta &\coloneqq \beta_{ m_1+\cdots+m_{\mathtt{n}}-\mathtt{n} } ,
\end{align*}
so that the eigenvalues $(\mu(n;\beta))_{|n| \leq N}$ of the operator $B_{0,\beta}$ are distinct by construction; moreover, $\beta \in B_{H^{s+1/2}}(R)$ and $\| \beta_0 - \beta \|_{ H^{s+1/2}(\mathbb{T}) } < \varepsilon$. 

The above argument shows the existence of a subset $\mathcal{G} \subseteq B_{H^{s+1/2}}(R)$ which is dense in $B_{H^{s+1/2}}(R)$ and such that any $\beta \in \mathcal{G}$ satisfies \eqref{eq:distinct}. 

The fact that this set $\mathcal{G}$ is open comes from the continuous dependence on $\beta$ for isolated eigenvalues of $B_{0,\beta}$: indeed, if $\beta^{\ast} \in \mathcal{G}$ satisfies \eqref{eq:StrConnAss}, then there exists $\delta^{\ast} >0$ such that the eigenvalues $( \mu(n;\beta^{\ast}) )_{|n| \leq N} \in \mathrm{spec}( B_{0,\beta^{\ast}} )$ satisfy
\begin{align*}
	| \mu(n;\beta^{\ast})-\mu(m;\beta^{\ast}) | &> \delta^{\ast} , \quad \forall |n|, |m| \leq N.
\end{align*}
Hence, by continuity there exists a neighbourhood $\mathcal{U}_{\beta^{\ast}} \subset B_{H^{s+1/2}}(R)$ of $\beta^{\ast}$ such that for any $\beta \in  \mathcal{U}_{\beta^{\ast}}$ satisfying \eqref{eq:StrConnAss} the eigenvalues $( \mu(n;\beta) )_{|n| \leq N} \in \mathrm{spec}( B_{0,\beta} )$ satisfy
\begin{align*}
	| \mu(n;\beta)-\mu(m;\beta) | &> \frac{1}{2} \delta^{\ast} , \quad \forall |n|, |m| \leq N,
\end{align*}
hence $\mathcal{U}_{\beta^{\ast}} \subseteq \mathcal{G}$.

\end{proof}

\begin{remark} \label{rem:GdeltaSet}
	
	Of course, the set $\mathcal{G}$ constructed in Lemma \ref{lem:LowFreq} depends on some parameters, more explicitly $\mathcal{G} = \mathcal{G}(s,N,R)$. By iterating the construction of Lemma \ref{lem:LowFreq} one can construct a $G_{\delta}$ set (namely, a countable intersection of open sets)
	\begin{align*}
		\mathscr{D} = \mathscr{D}(s,R) &\coloneqq \bigcap_{N \in \mathbb{N}_0} \mathcal{G}(s,N,R)
	\end{align*}
	which is dense in $B_{H^{s+1/2}}(R)$, and such that for any $\beta \in \mathscr{D}$ satisfying \eqref{eq:StrConnAss} we have that \emph{all} eigenvalues of the operator $A_{0,\beta}$ are simple.
	
	This fact may be of independent interest, especially in relation to the studies of the spectrum of the Dirichlet--Neumann operator in \cite{craig2018bloch,lacave2025bloch}.
	
\end{remark}

As we mentioned before Lemma \ref{lem:LowFreq}, we have to deduce the injectivity property of the map $\Psi_1$ in \eqref{eq:Psi1Map} from the one of the map $\Psi_2$ in \eqref{eq:Psi2Map}. One can prove this property by a perturbative argument, obtaining the following result.

\begin{corollary} \label{cor:lownu}
	
	Let $R>0$, $s>3$, let $N>0$ be such that \eqref{eq:ObsHigh} holds true, and let $\mathcal{G} \subseteq B_{H^{s+1/2}}(R)$ be the set appearing in the statement of Lemma \ref{lem:LowFreq}. There exists $\varepsilon>0$ such that if 
	\begin{align*}
		|b(t)| N^{1/2} \leq \varepsilon  , \quad \forall t \in [0,T] ,
	\end{align*}
	then for any $\beta \in \mathcal{G}$ satisfying \eqref{eq:StrConnAss} there exists $\delta=\delta(\beta)>0$ such that the quantities $(\nu_n(t) )_{n \in \mathbb{Z}}$ defined in \eqref{eq:nu} satisfy
	\begin{align} \label{eq:nudistinct}
		| \nu_n(t)-\nu_m(t) | &> \frac{1}{2}\delta t , \quad \forall t \in [0,T], \quad \forall |n|, |m| \leq N\ \text{with}\ n\neq m.
	\end{align}
	
\end{corollary}

\begin{remark} \label{rem:lownu}
	
	The smallness assumption on $b(t)$ in the statement of Corollary \ref{cor:lownu} does not impose an additional restriction.
	
	Indeed, $b(t)$ is related to $\| (\eta(t),\psi(t)) \|_{ H^{s+1/2}_0(\mathbb{T}) \times H^s(\mathbb{T}) }$, which will satisfy a smallness assumption, see the statement of Theorem \ref{thm:main}.
\end{remark}

Now we prove an estimate corresponding to the upper bound in Ingham's inequality.

\begin{proposition}[Ingham-type inequality, upper bound] \label{prop:UpperIngham}

There exists $C>0$ with the following property. Let $T>0$, $s> 3$, and let us assume that $\beta \in H^{s+1/2}(\mathbb{T})$ satisfies \eqref{eq:StrConnAss}. Let
$|\partial_t b| \leq \frac{1}{2}\tanh^{\frac12}(h)$, and $|\partial_t^k b|\leq 1$, $k=2,3$ on $[0,T]$.
Then there exists $C = C\left( \|\beta\|_{ H^{s+1/2}(\mathbb{T}) } \right) >0$ such that for all $(w_n)\in \ell^2(\mathbb{Z};\mathbb{C})$,
\begin{equation} \label{eq:UpperIngham}
	\int_0^T \left|\sum_{n\in\mathbb{Z}} w_n \zeta_n(t)e^{\mathrm{i}\nu_n(t)}\right|^2 \mathrm{d}t
	\leq C M(\zeta)^2(1+T)\sum_{n\in\mathbb{Z}} |w_n|^2 , \quad \forall (w_n)_{n \in \mathbb{Z}} \in \ell^2(\mathbb{Z};\mathbb{C}) ,
\end{equation}
where the quantities $(\nu_n(t) )_{n \in \mathbb{Z}}$ are defined in \eqref{eq:nu}, and where
\begin{equation*}
	M(\zeta) \coloneqq \sup_{n\in\mathbb{Z}}\|\zeta_n\|_{L^\infty(\mathbb{T})}
	+ \sup_{n\in\mathbb{Z}} \frac{\|\partial_t\zeta_n\|_{L^\infty(\mathbb{T})}}{\sqrt{1+|n|}}
	+ \sup_{n\in\mathbb{Z}} \frac{\|\partial_t^2\zeta_n\|_{ L^\infty(\mathbb{T})} }{1+|n|}.
\end{equation*}
\end{proposition}

\begin{proof}

By splitting the sum into $n=m$ and $n\neq m$, we have
\[
\int_0^T \left|\sum_{n\in\mathbb{Z}} w_n\zeta_n(t)e^{\mathrm{i}\nu_n(t)}\right|^2 \mathrm{d}t
= \sum_{n\in\mathbb{Z}} \left(\int_0^T |\zeta_n(t)|^2\mathrm{d}t\right)|w_n|^2
+ \sum_{n\neq m} w_n \overline{w_m} E(n,m)
\]
with
\begin{equation*}
E(n,m) \coloneqq \int_0^T \zeta_n(t)\overline{\zeta_m(t)} e^{\mathrm{i}(\nu_n(t)-\nu_m(t))}\,\mathrm{d}t.
\end{equation*}
The bound for the first sum on the right-hand side is straightforward. It remains to bound the
sum for $n\neq m$; integrating by parts twice, one has
\begin{align*}
	E(n,m) &= \int_0^T f e^{\mathrm{i}h}\,\mathrm{d}t
	= \left[e^{\mathrm{i}h}\left(-ifp+f'p^2-fh''p^3\right)\right]_0^T \\
	&\quad + \int_0^T e^{\mathrm{i}h}\left(f''p^2-3f'h''p^3+3fh''^2p^4-fh'''p^3\right)\,\mathrm{d}t
\end{align*}
with
\begin{equation*}
f \coloneqq \zeta_n\overline{\zeta_m},\qquad h:=\nu_n-\nu_m,\qquad p:=\frac{1}{\nu'_n-\nu'_m}.
\end{equation*}
Thus $|E(n,m)|\leq e(n,m)$, where
\begin{align}
	e(n,m) &:=2\|fp\|_{L^\infty(\mathbb{T})}+2\|f'p^2\|_{L^\infty(\mathbb{T})}+2\|fh''p^3\|_{L^\infty(\mathbb{T})} \nonumber\\
	&\quad +T\left(\|f''p^2\|_{L^\infty(\mathbb{T})}+3\|f'h''p^3\|_{L^\infty(\mathbb{T})}+3\|fh''^2p^4\|_{L^\infty(\mathbb{T})}+\|fh'''p^3\|_{L^\infty(\mathbb{T})}\right).
	\label{eq:enm}
\end{align}
We have to estimate the sum $\sum_{m\in\mathbb{Z}\setminus\{n\}}e(n,m)$, uniformly in $n$. First, we note that
\begin{equation*}
\|\partial_t^k(\zeta_n\overline{\zeta_m})\|_{L^\infty}=\|\partial_t^k f\|_{L^\infty}
\leq \left\{(1+|n|)^{\frac12}+(1+|m|)^{\frac12}\right\}^k M(\zeta)^2,
\quad k=0,1,2.
\end{equation*}
We have already seen in the proof of \eqref{eq:ObsHigh} that $|h''p|\leq 2|\partial_t^2 b|$. Similarly, $|h'''p|\leq 2|\partial_t^3 b|$. Also, applying the claim in the proof of Proposition \ref{prop:HighFreq} with $N=0$, $\varepsilon=\frac{1}{4}$, we deduce that $\sum_{m\in\mathbb{Z}\setminus\{n\}}\|p\|_{L^\infty}\leq C$ for some absolute constant $C$. Therefore all terms in \eqref{eq:enm} proportional to $f$ (i.e. the first, the third and the last two) are bounded by $CM(\zeta)^2(1+T)$. The remaining three terms of \eqref{eq:enm} are also bounded by $CM(\zeta)^2(1+T)$, provided that
\begin{equation} 	\label{eq:TechBound}
	\sum_{m\in\mathbb{Z}\setminus\{n\}} \left\|\frac{|n|+|m|}{(\nu'_n-\nu'_m)^2}\right\|_{L^\infty}\leq C
\end{equation}
for all $n\in\mathbb{Z}$, for some $C$ independent of $n$. The bound \eqref{eq:TechBound} is proved using the same argument as in the proof of \eqref{eq:ObsHigh}.
\end{proof}

Now we prove the following technical result (see Claim 6.6 of \cite{alazard2018control}; see also \cite{micu2004introduction}).

\begin{lemma} \label{lem:IndLemma}

Let $R>0$, $s>3$, let $N>0$ be such that \eqref{eq:ObsHigh} holds true, and let $\mathcal{G} \subseteq B_{H^{s+1/2}}(R)$ be the set appearing in the statement of Lemma \ref{lem:LowFreq}.

Then for any $\beta \in \mathcal{G}$ satisfying \eqref{eq:StrConnAss} the following holds true: consider two subsets $\mathscr{A},\mathscr{A}'$ of $\mathbb{Z}$ with $\mathscr{A}'=\mathscr{A}\cup\{N\}$ for some $N\in\mathbb{Z}$, and with $|n|\geq |N|$ for all $n$ in $\mathscr{A}$. Assume that for every $T>0$ there exist two positive constants $\varepsilon=\varepsilon(T, \|\beta\|_{H^{s+1/2}(\mathbb{T})} )$ and $K=K(T, \|\beta\|_{H^{s+1/2}(\mathbb{T})} )$ such that
\begin{equation} \label{eq:IndAss}
	\|b\|_X  \coloneqq \sup_{t\in[0,T]} (\partial_t b,\partial_t^2 b,\partial_t^3 b)
	\leq \varepsilon \Rightarrow
	K \sum_{n\in \mathscr{A}}|w_n|^2 \leq \int_0^T \left|\sum_{n\in \mathscr{A}} w_n e^{\mathrm{i}\nu_n(t)}\right|^2 \mathrm{d}t.
\end{equation}
Then for every $T>0$ there exist two positive constants $\varepsilon'$ and $K'$ (also depending on $T$ and on $\|\beta\|_{H^{s+1/2}(\mathbb{T})}$) such that
\begin{equation} \label{eq:IndThesis}
	\|b\|_X \leq \varepsilon' \Rightarrow
		K' \sum_{n\in \mathscr{A}'}|w_n|^2 \leq \int_0^T \left|\sum_{n\in \mathscr{A}'} w_n e^{\mathrm{i}\nu_n(t)}\right|^2 \mathrm{d}t.
\end{equation}
\end{lemma}

\begin{proof}
	
Introduce
\begin{equation*}
	f(t) \coloneqq \sum_{n\in \mathscr{A}} w_ne^{\mathrm{i}\nu_n(t)},\quad
	f'(t) \coloneqq \sum_{n\in \mathscr{A}'} w_ne^{\mathrm{i}\nu_n(t)},\quad
	f_1(t) \coloneqq \sum_{n\in \mathscr{A}'} w_ne^{\mathrm{i}\nu_n(t)-\mathrm{i}\nu_N(t)},
\end{equation*}
so that $f'=f+w_Ne^{\mathrm{i}\nu_N}$, $f_1=e^{-\mathrm{i}\nu_N}f'=fe^{-\mathrm{i}\nu_N}+w_N$, and
\begin{equation*}
	\int_0^T |f_1(t)|^2\mathrm{d}t=\int_0^T |f'(t)|^2\mathrm{d}t
	=\int_0^T\left|\sum_{n\in \mathscr{A}'}w_ne^{\mathrm{i}\nu_n(t)}\right|^2\mathrm{d}t.
\end{equation*}
We prove that there exist two constants $C_1,C_2$ (both depending on $T$ and on  $\|\beta\|_{H^{s+1/2}(\mathbb{T})}$) such that
\begin{equation}
	C_1\sum_{n\in \mathscr{A}}|w_n|^2 \leq \int_0^T |f'(t)|^2\mathrm{d}t,
		\qquad
	C_2|w_N|^2 \leq \int_0^T |f'(t)|^2\mathrm{d}t. \label{eq:claim}
\end{equation}
Then \eqref{eq:claim} implies the second inequality of \eqref{eq:IndThesis} with $K' \coloneqq \frac{1}{2}\min(C_1,C_2)$. Let us begin with the first inequality of \eqref{eq:claim}. Let $\tau \coloneqq \frac{1}{2} \min(1,T)$; notice that
\begin{align} 
	\int_0^\tau (f_1(t+\eta)-f_1(t))\,\mathrm{d}\eta
		&= e^{-\mathrm{i}\nu_N(t)}\sum_{n\in \mathscr{A}} w_ne^{\mathrm{i}\nu_n(t)}\theta_n(t), \label{eq:IntPart} \\
	\theta_n(t) &\coloneqq \int_0^\tau\left(e^{\mathrm{i}(\nu_n(t+\eta)-\nu_n(t)-\nu_N(t+\eta)+\nu_N(t))}-1\right)\mathrm{d}\eta , \nonumber		
\end{align}
and that the sum is over $\mathscr{A}$.

Assume that $n,N$ are positive. We split $\theta_n=c_n+\zeta_n$, where $c_n$ is a constant, independent of time (such that $c_n=\theta_n$ for $b=0$), and $\zeta_n$ is defined by difference, namely
\begin{align*}
	c_n &\coloneqq \int_0^\tau (e^{\mathrm{i}[\mu(n)-\mu(N)]\eta}-1)\mathrm{d}\eta
	=\frac{e^{\mathrm{i}[\mu(n)-\mu(N)]\tau}-1}{\mathrm{i}[\mu(n)-\mu(N)]}-\tau, \\
	\zeta_n &\coloneqq \int_0^\tau e^{\mathrm{i}[\mu(n)-\mu(N)]\eta}\left(e^{\mathrm{i}[b(t+\eta)-b(t)](\sqrt{n}-\sqrt{N})}-1\right)\mathrm{d}\eta.
\end{align*}
Now we use the following inequality: there exists a constant $c_0>0$ such that
\begin{equation*}
	|e^{\mathrm{i}\vartheta}-1-\mathrm{i}\vartheta|^2\geq c_0\min(\vartheta^2,\vartheta^4) , \quad  \forall \vartheta\in\mathbb{R}.
\end{equation*}
We apply this inequality with $\vartheta=[\mu(n)-\mu(N)]\tau$, and, using that
\begin{equation} \label{eq:EstHighFreq}
	| \mu(n)-\mu(m) | \geq C (\max(n,m))^{1/2} |n-m| , \quad \forall |n|,|m| \geq N, \quad C \coloneqq \frac{1}{2} \tanh^{1/2}(h) , 
\end{equation}
we get
\begin{equation*}	
	|c_n|^2\geq c\tau^4
\end{equation*}
for some $c>0$ (note that $\min\{\tau^2,\tau^4\}=\tau^4$ because, by assumption, $\tau<1$).
	
It remains to estimate $\zeta_n$ and its derivatives: in particular, we need a bound on $\zeta_n$ which shows that $\zeta_n$ is small when $b$ is
small.
%	\[
%	|\zeta_n|\leq 2\tau,\qquad |\partial_t\zeta_n|\leq 2\|\partial_t b\|_{L^\infty}\tau\sqrt{n},\qquad
%	|\partial_t^2\zeta_n|\leq 4(\|\partial_t b\|_{L^\infty}^2+\|\partial_t^2b\|_{L^\infty})\tau n.
%	\]

Integrating by parts we obtain
\begin{align*}
	\zeta_n(t) &= \frac{e^{\mathrm{i}[\mu(n)-\mu(N)]\tau}}{\mathrm{i}[\mu(n)-\mu(N)]}
		\left(e^{\mathrm{i}[b(t+\tau)-b(t)](\sqrt n-\sqrt N)}-1\right)\\
	&\quad -\int_0^\tau \frac{e^{\mathrm{i}[\mu(n)-\mu(N)]\eta}}{\mathrm{i}[\mu(n)-\mu(N)]}
		\partial_\eta\left(e^{\mathrm{i}[b(t+\eta)-b(t)](\sqrt n-\sqrt N)}-1\right)d\eta,
\end{align*}
and one can check, using \eqref{eq:EstHighFreq} and the bound $|b(t+\tau)-b(t)|\leq \tau\|\partial_t b\|_{L^\infty}$, that $|\zeta_n|\leq C\tau\|\partial_t b\|_{L^\infty}$. By combining the previous estimates, we have $M(\zeta)\leq C\tau\|b\|_X$ where $M(\zeta)$ is as in the statement of Proposition \ref{prop:UpperIngham}, and $C$ is independent of $T,\tau$.
	
Now set $F(t) \coloneqq \sum_{n\in \mathscr{A}} w_n e^{\mathrm{i}\nu_n(t)}\theta_n(t)$ and split $F=F_1+F_2$ with
\begin{equation*}
	F_1(t) \coloneqq \sum_{n\in \mathscr{A}}w_n e^{\mathrm{i}\nu_n(t)}c_n,
	\qquad
	F_2(t) \coloneqq \sum_{n\in \mathscr{A}}w_n e^{\mathrm{i}\nu_n(t)}\zeta_n(t).
\end{equation*}
Since $|c_n|^2\geq c\tau^4$, the assumption \eqref{eq:IndAss} implies that, if $\|b\|_X\leq \varepsilon (T-\tau)$, then
\begin{equation*}
	c\tau^4 K(T-\tau)\sum_{n\in \mathscr{A}}|w_n|^2
	\leq K(T-\tau)\sum_{n\in \mathscr{A}}|w_nc_n|^2
	\leq \int_0^{T-\tau}|F_1(t)|^2\mathrm{d}t.
\end{equation*}
On the other hand, Proposition \ref{prop:UpperIngham} applied with $M(\zeta)\leq C\tau\|b\|_X$ implies that, if $\|b\|_X \leq \frac{1}{2} \tanh^{1/2}(h)$, then
\begin{equation*}
	\int_0^{T-\tau} |F_2(t)|^2\mathrm{d}t
	\leq C_0\tau^2\|b\|_X^2(1+T-\tau)\sum_{n\in \mathscr{A}}|w_n|^2
\end{equation*}
where $C_0$ is independent of $T,\tau$. Therefore, if
\begin{equation} \label{eq:bSmallCond1}
		4C_0\tau^2\|b\|_X^2(1+T-\tau)\leq c\tau^4 K(T-\tau),
\end{equation}
then $\int_0^{T-\tau}|F_2|^2\mathrm{d}t\leq \frac14\int_0^{T-\tau}|F_1|^2\mathrm{d}t$, which in turn implies 
\begin{equation*}
		\int_0^{T-\tau}|F|^2\mathrm{d}t\geq \frac14\int_0^{T-\tau}|F_1|^2\mathrm{d}t .
\end{equation*}
By \eqref{eq:IntPart}, this implies that
\begin{align*}
	\frac{1}{4} c\tau^4 K(T-\tau)\sum_{n\in \mathscr{A}}|w_n|^2
		&\leq \int_0^{T-\tau}\left|\int_0^\tau (f_1(t+\eta)-f_1(t))\,d\eta\right|^2\mathrm{d}t.
\end{align*}
The condition \eqref{eq:bSmallCond1} holds if
\begin{equation} \label{eq:bSmallCond2}
	\|b\|_X \leq \frac{\tau\sqrt{cK(T-\tau)}}{2\sqrt{C_0(1+T)}},
\end{equation}
and we set $\varepsilon'$ as the minimum among $\frac{1}{2} \tanh^{1/2}(h)$, $\varepsilon$, and the constant on the right-hand side of \eqref{eq:bSmallCond2}. Moreover,
\begin{align*}
	\int_0^{T-\tau}\left|\int_0^\tau (f_1(t+\eta)-f_1(t))\,d\eta\right|^2\mathrm{d}t
	&\quad \leq 2T\tau\int_0^T |f_1(t)|^2\mathrm{d}t
		=2T\tau\int_0^T |f'(t)|^2\mathrm{d}t,
\end{align*}
and we can deduce the first inequality in \eqref{eq:claim}, with $C_1=\frac{1}{8} c\tau^3K(T-\tau)T^{-1}$.
	
Now we prove the second inequality in \eqref{eq:claim}. We have
\begin{equation*}
	|w_N|^2=\frac1T\int_0^T |f'(t)-f(t)|^2\mathrm{d}t
	\leq \frac2T\left(\int_0^T |f'(t)|^2\mathrm{d}t+\int_0^T |f(t)|^2\mathrm{d}t\right).
\end{equation*}
By Proposition \ref{prop:UpperIngham} (applied with $\zeta_n=1$) we have
\begin{equation*}
	\int_0^T |f(t)|^2\mathrm{d}t\leq (1+T)C\sum_{n\in \mathscr{A}}|w_n|^2.
\end{equation*}
Using the first inequality in \eqref{eq:claim}, we obtain that
\begin{equation*}
	\int_0^T |f(t)|^2\mathrm{d}t\leq \frac{(1+T)C}{C_1}\int_0^T |f'(t)|^2\mathrm{d}t,
\end{equation*}
where $C$ is the constant of Proposition \ref{prop:UpperIngham}, and where $C_1$ has been determined above. Therefore, the second inequality in \eqref{eq:claim} holds with $C_2=\frac{1}{2} TC_1[C_1+(1+T)C]^{-1}$. We set $K'=\frac{1}{2} \min(C_1,C_2)$ and obtain \eqref{eq:IndThesis}. This completes the proof of the statement in the case of $n,N$ positive. The other cases are analogous. 
	
\end{proof}

\begin{proposition}[Ingham-type inequality, lower bound] \label{prop:InghamLowFinal}
	
Let $T>0$. 
Let $R>0$, $s>3$, let $N>0$ be such that \eqref{eq:ObsHigh} holds true, and let $\mathcal{G} \subseteq B_{H^{s+1/2}}(R)$ be the set appearing in the statement of Lemma \ref{lem:LowFreq}.

Then for any $\beta \in \mathcal{G}$ satisfying \eqref{eq:StrConnAss} there exist constants 
\begin{equation*}
	C=C( T,\|\beta\|_{H^{s+1/2}(\mathbb{T})} ) >0 , \quad \varepsilon=\varepsilon( T,\|\beta\|_{H^{s+1/2}(\mathbb{T})} ) >0 ,
\end{equation*}
such that, if
\begin{equation} \label{eq:normX}
	\|b\|_X = \sup_{t\in[0,T]} (\partial_t b,\partial_t^2 b,\partial_t^3 b) \leq \varepsilon,
\end{equation}
then, for all $(w_n)\in\ell^2(\mathbb{Z};\mathbb{C})$,
\begin{equation*}
C \sum_{n\in\mathbb{Z}} |w_n|^2
\leq \int_0^T \left|\sum_{n\in\mathbb{Z}}w_n e^{\mathrm{i}\nu_n(t)}\right|^2\mathrm{d}t.
\end{equation*}
\end{proposition}

\begin{proof}
	
This proposition can be deduced by combining Proposition \ref{prop:HighFreq}, Corollary \ref{cor:lownu}, Lemma \ref{lem:IndLemma} and a finite induction argument.
\end{proof}

\section{Observability} \label{sec:Obs}

We now use the Ingham type inequalities proved in Sec. \ref{sec:Ingham} in order to prove an observability property.

\begin{proposition}[Observability for the reduced equation] \label{prop:RedObs}

Let $T>0$, $R>0$, $s>3$ and let $\mathcal{G} \subseteq B_{H^{s+1/2}}(R)$ be the set appearing in the statement of Lemma \ref{lem:LowFreq}. Consider an open subset $\omega\subset\mathbb{T}$ and a constant $0<c\leq1$. Then for any $\beta \in \mathcal{G}$ for which there exists $h_0>0$ such that \eqref{eq:StrConnAss} holds true, there exist $K,\varepsilon_1 >0$ such that the following property holds. Consider a pseudo-differential $\mathcal{A}_0$ with symbol $\exp(\mathrm{i} b(t,x)|\xi|^{1/2})$ for some function $b$ satisfying
\begin{equation*}
\sup_{t\in[0,T]}\sup_{x\in[0,2\pi]} (\partial_t b(t,x),\partial_t^2 b(t,x),\partial_t^3 b(t,x))\leq \varepsilon( T,\|\beta\|_{ H^{s+1/2}(\mathbb{T}) } ),
\end{equation*}
where $\varepsilon( T,\|\beta\|_{ H^{s+1/2}(\mathbb{T}) } )$ is the constant appearing in Proposition \ref{prop:InghamLowFinal}. Then for every initial data $v_0\in L^2(\mathbb{T})$ with mean value $\langle v_0\rangle=\frac{1}{2\pi}\int_{\mathbb{T}}v_0(x)\,dx$ satisfying
\begin{equation} \label{eq:RealAss}
|\operatorname{Re}\langle v_0\rangle|\geq c|\langle v_0\rangle|-\varepsilon_1\|v_0\|_{L^2} ,
\end{equation}
the solution $v$ of
\begin{equation} \label{eq:ReducedEq}
v_t + \mathrm{i} A_{0,\beta}v=0,\qquad v(0)=v_0,
\end{equation}
satisfies
\begin{equation} \label{eq:ReducedObs}
\int_0^T\int_\omega |\operatorname{Re}(\mathcal{A}_0 v)(t,x)|^2\,\mathrm{d}x \, \mathrm{d}t
\geq K\int_0^{2\pi}|v_0(x)|^2\,\mathrm{d}x.
\end{equation}

\end{proposition}

\begin{remark} \label{rem:RedObs}
	
The condition \eqref{eq:RealAss} cannot be eliminated. Indeed, consider the case $b=0$, so $\mathcal{A}_0=\mathrm{Id}$, and consider a constant solution $v(t,x)=C$ of \eqref{eq:ReducedEq}. Then \eqref{eq:ReducedObs} holds for some $K$ if and only if the real part of $C$ is non zero. This naturally leads to the stronger assumption
\begin{equation} \label{eq:RealAssStrong}
|\operatorname{Re}\langle v_0\rangle|\geq c|\langle v_0\rangle| ,
\end{equation}
but the advantage of assuming \eqref{eq:RealAss} instead of \eqref{eq:RealAssStrong} will be more apparent later in the Section.
\end{remark}

\begin{proof}
	
Write
\begin{equation*}
v(t,x)=\sum_{n\in\mathbb{Z}} a_n e^{\mathrm{i} ( nx + \mu(n)t)},
\qquad
a_n=\frac{1}{2\pi}\int_0^{2\pi} e^{-\mathrm{i}nx}v_0(x)\,\mathrm{d}x,
\end{equation*}
where $(\mu(n))_{n \in \mathbb{Z}}$ are the eigenvalues of the operator $A_{0,\beta}$ (see \eqref{eq:uSymmLin}). Then set $w=A_0v$, given by
\begin{equation*}
w(t,x)=\sum_{n\in\mathbb{Z}}a_n e^{\mathrm{i} \left( nx + \mu(n)t+b(t,x)|n|^{1/2} \right)}.
\end{equation*}
For $n\in\mathbb{Z}$, set
\begin{equation*}
\lambda_n \coloneqq \mu(n)t+b(t,x)|n|^{1/2},\qquad \nu_n \coloneqq \mathrm{sign}(n)\lambda_n,
\qquad c_n(x) \coloneqq a_n e^{\mathrm{i}nx} ,
\end{equation*}
and write
\begin{align*}
2\mathrm{Re} \, w
&=2\mathrm{Re} \, a_0+\sum_{n>0}c_n e^{\mathrm{i}\lambda_n}+\sum_{n>0}\overline{c_n}e^{-\mathrm{i}\lambda_n}
+\sum_{n<0}c_n e^{\mathrm{i}\lambda_n}+\sum_{n<0}\overline{c_n}e^{-\mathrm{i}\lambda_n} \\
&= 2\mathrm{Re} \, a_0 + \sum_{n>0}c_n e^{\mathrm{i}\nu_n} + \sum_{n<0} \overline{c_{-n}} e^{-\mathrm{i}\nu_{-n}} + \sum_{n>0} c_{-n} e^{-\mathrm{i} \nu_{-n} } + \sum_{n<0} \overline{c_n} e^{\mathrm{i} \nu_n } ,
\end{align*}
(notice that in the case $\beta=0$ we have the further simplification $\nu_{-n}=-\nu_n$), so that
\begin{equation*}
2\mathrm{Re} \, w=\sum_{n\in\mathbb{Z}}\gamma_ne^{\mathrm{i}\nu_n}
\quad\text{with}\quad
\gamma_n=\begin{cases}
c_n+c_{-n} e^{-\mathrm{i} ( \nu_{-n} + \nu_n ) }, & \text{for } n>0,\\
2\mathrm{Re} \, c_0, & \text{for } n=0,\\
\overline{c_n}+\overline{c_{-n}} e^{-\mathrm{i} ( \nu_{-n} + \nu_n ) } , & \text{for } n<0.
\end{cases}
\end{equation*}
Consider an interval $\omega_0=[a,b]\subset\omega$. By Proposition \ref{prop:InghamLowFinal},
\begin{align}
\int_0^T\int_\omega |\operatorname{Re}(w(t,x))|^2\,\mathrm{d}x \, \mathrm{d}t
&\geq \int_{\omega_0}\int_0^T |\operatorname{Re}(w(t,x))|^2\,\mathrm{d}t \, \mathrm{d}x \nonumber\\
&\geq \frac{1}{4} \, C( T,\|\beta\|_{H^{s+1/2}(\mathbb{T})} )\, \int_{\omega_0}\sum_{n\in\mathbb{Z}}|\gamma_n(x)|^2 \mathrm{d}x,
\label{eq:ObsEst}
\end{align}
where $C( T,\|\beta\|_{H^{s+1/2}(\mathbb{T})} )$ is the constant given in Proposition \ref{prop:InghamLowFinal}. We have
\begin{align*}
|\gamma_n(x)|^2 &= |a_n|^2+|a_{-n}|^2+a_n\overline{a_{-n}}e^{2\mathrm{i}nx} e^{\mathrm{i} ( \nu_{-n} + \nu_n ) } +\overline{a_n}a_{-n}e^{-2\mathrm{i}nx} e^{-\mathrm{i} ( \nu_{-n} + \nu_n ) } , \; n>0, \\
|\gamma_n(x)|^2 &= |a_n|^2+|a_{-n}|^2+a_n\overline{a_{-n}}e^{2\mathrm{i}nx} e^{-\mathrm{i} ( \nu_{-n} + \nu_n ) } +\overline{a_n}a_{-n}e^{-2\mathrm{i}nx} e^{\mathrm{i} ( \nu_{-n} + \nu_n ) }, \; n<0, 
\end{align*}
so that
\begin{align}
\int_{\omega_0}|\gamma_n(x)|^2 \mathrm{d}x &\geq |\omega_0| \left[ |a_n|^2+|a_{-n}|^2 \right] \nonumber \\
&\quad -|a_n||a_{-n}| \left(\left|\int_{\omega_0}e^{2\mathrm{i}nx} e^{\mathrm{i} ( \nu_{-n} + \nu_n ) } \mathrm{d}x\right| +\left|\int_{\omega_0}e^{-2\mathrm{i}nx} e^{-\mathrm{i} ( \nu_{-n} + \nu_n ) } \mathrm{d}x\right| \right). \label{eq:ObsEst2}
\end{align}
Now, if $|n| \geq N$ we can use Lemma \ref{lem:BottomSpec} in order to argue that the leading order terms in \eqref{eq:ObsEst2} are proportional to 
\begin{equation*}
\left|\int_{\omega_0}e^{2\mathrm{i}nx}\mathrm{d}x\right|=\left|\int_{\omega_0}e^{-2\mathrm{i}nx}\mathrm{d}x\right|=\frac{|\sin(n(b-a))|}{|n|} ,
\end{equation*}
so that arguing as in Proposition 7.1 of \cite{alazard2018control} we get that there exists $c_1' >0$ such that 
\begin{equation*}
\int_{\omega_0}|\gamma_n(x)|^2 \mathrm{d}x \geq c_1'(|a_n|^2+|a_{-n}|^2) .
\end{equation*}
On the other hand, if $|n| \leq N$ we can use Corollary \ref{cor:lownu} in order to obtain again that there exists $c_2' >0$ such that 
\begin{equation*}
	\int_{\omega_0}|\gamma_n(x)|^2 \mathrm{d}x \geq c_2'(|a_n|^2+|a_{-n}|^2) .
\end{equation*}

Then, recalling that $\gamma_0=2\mathrm{Re} \, a_0$, it follows from \eqref{eq:ObsEst} that
\begin{align*}
& \int_0^T\int_\omega |\operatorname{Re}(w(t,x))|^2\, \mathrm{d}x \, \mathrm{d}t \\
& \geq C( T,\|\beta\|_{ H^{s+1/2}(\mathbb{T}) } ) \left[(b-a)|\mathrm{Re} \, a_0|^2 + \frac{\min(c_1',c_2')}{2}\sum_{n\in\mathbb{Z}\setminus\{0\}}|a_n|^2\right].
\end{align*}
Now, using $(x+y)^2\geq \frac{1}{2} x^2-y^2$ and \eqref{eq:RealAss}, we have
\begin{equation*}
|\mathrm{Re}a_0|^2\geq \frac{c^2}{2}|a_0|^2-2\pi\varepsilon_1^2\sum_{n\in\mathbb{Z}}|a_n|^2,
\end{equation*}
and therefore
\begin{equation*}
\int_0^T\int_\omega |\mathrm{Re}(w(t,x))|^2\, \mathrm{d}x \, \mathrm{d}t
\geq K\sum_{n\in\mathbb{Z}}|a_n|^2,
\end{equation*}
with 
\begin{equation*}
K=C( T,\|\beta\|_{ H^{s+1/2}(\mathbb{T}) } ) \min\{(b-a)(\frac{1}{2} c^2-2\pi\varepsilon_1^2) , \frac{1}{2} \min(c_1',c_2')-(b-a)2\pi\varepsilon_1^2\} .
\end{equation*}
If $\varepsilon_1$ is small enough, then $K>0$, which leads to the thesis.

\end{proof}

\begin{corollary} \label{cor:RedObs}
	
Let $T>0$, $R>0$, $s>3$ and let $\mathcal{G} \subseteq B_{H^{s+1/2}}(R)$ be the set appearing in the statement of Lemma \ref{lem:LowFreq}. Consider an open subset $\omega\subset\mathbb{T}$ and a constant $0<c\leq1$. Then for any $\beta \in \mathcal{G}$ for which there exists $h_0>0$ such that \eqref{eq:StrConnAss} holds true, there exist positive constants $\varepsilon_0,\varepsilon_1,r,K$ such that the following property holds. Assume that
\begin{equation*}
\langle W(t)\rangle=0
\end{equation*}
for all $t\in[0,T]$ and
\begin{equation*}
\sup_{t\in[0,T]}\sum_{1\leq k\leq3}\|\partial_t^k W(t)\|_{H^1}+\sup_{t\in[0,T]}\|W(t)\|_{H^r}\leq \varepsilon_0,
\end{equation*}
and consider the pseudo-differential operator $\mathcal{A}$, given by Corollary \ref{cor:Reduction}, with symbol $q(t,x,\xi)\exp(\mathrm{i} b(t,x)|\xi|^{1/2})$. Then for every initial data $v_0\in L^2(\mathbb{T})$ with mean value $\langle v_0\rangle=\frac{1}{2\pi}\int_{\mathbb{T}}v_0(x)\, \mathrm{d}x$ satisfying \eqref{eq:RealAss}, the solution $v$ of \eqref{eq:ReducedEq} satisfies \eqref{eq:ReducedObs}.

(The constants $\varepsilon_0,\varepsilon_1,K$ depend on $T,c$ and $\|\beta\|_{ H^{s+1/2}(\mathbb{T}) }$, while $r$ is a universal constant.)

\end{corollary}

\begin{proof}

We decompose $\mathcal{A}$ as $\mathcal{A}_0+\mathcal{A}_1$, where
\begin{equation*}
\mathcal{A}_0 \coloneqq \mathrm{Op}(\exp(\mathrm{i} b(t,x)|\xi|^{1/2})),\qquad
\mathcal{A}_1 \coloneqq \mathrm{Op}((q(t,x,\xi)-1)\exp(\mathrm{i} b(t,x)|\xi|^{1/2})).
\end{equation*}
The contribution due to $\mathcal{A}_0$ is estimated by Proposition \ref{prop:RedObs}. Notice that, for $\varepsilon_0$ small enough, the smallness assumption on $b$ of Proposition \ref{prop:RedObs} holds true, since
\begin{equation*}
\sup_{t\in[0,T]}\sup_{x\in[0,2\pi]}(\partial_t b(t,x),\partial_t^2b(t,x),\partial_t^3b(t,x))
\lesssim \sup_{t\in[0,T]}\sum_{1\leq k\leq3}\|\partial_t^k W(t)\|_{H^1}\lesssim \varepsilon_0.
\end{equation*}
On the other hand, it follows from the definition of $q$ and $b$ and the estimate in part ii) of Lemma 5.4 of \cite{alazard2018control} that $\mathcal{A}_1$ is bounded from $L^2$ onto itself, with an operator norm of size $O(\|W\|_{H^r})=O(\varepsilon_0)$. Then
\begin{equation*}
\int_0^T\int_{\omega} |\mathrm{Re}(\mathcal{A}_1 v)(t,x)|^2\,\mathrm{d}x \, \mathrm{d}t
\leq \int_0^T \|\mathcal{A}_1 v(t)\|_{ L^2(\mathbb{T})  }^2 \, \mathrm{d}t
\lesssim \int_0^T \varepsilon_0^2\|v(t)\|_{ L^2(\mathbb{T}) }^2 \, \mathrm{d}t.
\end{equation*}
Since $\|v(t)\|_{L^2}=\|v(0)\|_{ L^2(\mathbb{T}) }$, by taking $\varepsilon_0$ small enough we can deduce the thesis.

\end{proof}

We now want to deduce an observability result for equations of the form
\begin{equation*}
	w_t+W w_x+\mathrm{i}A_{0,\beta}w+\mathscr{R}w=0,
\end{equation*}
where $\mathscr{R}$ is an operator of order $0$.

\begin{corollary} \label{cor:ObsFull}
	
Let $T>0$, $R>0$, $s>3$ and let $\mathcal{G} \subseteq B_{H^{s+1/2}}(R)$ be the set appearing in the statement of Lemma \ref{lem:LowFreq}. Consider an open subset $\omega\subset\mathbb{T}$ and a constant $0<c\leq1$. Then for any $\beta \in \mathcal{G}$ for which there exists $h_0>0$ such that \eqref{eq:StrConnAss} holds true, there exist constants $\varepsilon_2,\varepsilon_3,r,K >0$ such that the following property holds. Assume that
\begin{equation*}
\langle W(t)\rangle=0
\end{equation*}
for all $t\in[0,T]$ and
\begin{equation} \label{eq:SmallAssObs}
\sup_{t\in[0,T]}\sum_{1\leq k\leq3}\|\partial_t^k W(t)\|_{H^1(\mathbb{T})}
+\sup_{t\in[0,T]}\|W(t)\|_{H^r(\mathbb{T})}
+\sup_{t\in[0,T]}\|R(t)\|_{L( L^2(\mathbb{T}) )}\leq \varepsilon_2.
%\tag{7.9}
\end{equation}
Then for every initial data $w_0\in L^2(\mathbb{T})$ with mean value $\langle w_0\rangle=\frac{1}{2\pi}\int_{\mathbb{T}}w_0(x)\,dx$ satisfying
\begin{equation} \label{eq:RealAssFull}
|\mathrm{Re}\langle w_0\rangle|\geq c|\langle w_0\rangle| - \varepsilon_3 \|w_0\|_{L^2(\mathbb{T})},
\end{equation}
the solution $w$ of
\begin{equation} \label{eq:FullEq}
w_t + W w_x + \mathrm{i} A_{0,\beta}w+\mathscr{R}w=0,\qquad w(0)=w_0
\end{equation}
satisfies
\begin{equation} \label{eq:FullObs}
\int_0^T\int_{\omega} |\mathrm{Re} \, w|^2\,\mathrm{d}x \, \mathrm{d}t
\geq K\int_0^{2\pi}|w_0(x)|^2\,\mathrm{d}x.
\end{equation}

\end{corollary}

\begin{remark} \label{rem:ObsFull} 

Corollary \ref{cor:ObsFull} also holds for data at time $T$, namely: if $w_0\in L^2(\mathbb{T})$ satisfies \eqref{eq:RealAssFull}, then the solution $w$ of
\begin{equation*}
w_t +W w_x + \mathrm{i} A_{0,\beta}w+\mathscr{R}w=0,\qquad w(T)=w_0
\end{equation*}
also satisfies \eqref{eq:FullObs} (see also Remark 7.5 in \cite{alazard2018control}).
\end{remark}

\begin{proof}
	
By Corollary \ref{cor:Reduction} there exists a change of variable $w=\mathcal{A}v$ such that $v$ satisfies an equation of the form
\begin{equation*}
v_t + \mathrm{i} A_{0,\beta}v+\mathcal{R}v=0,
\end{equation*}
where $\mathcal{R}$ is an operator of order $0$, satisfying $\|\mathcal{R}(t)v\|_{L^2(\mathbb{T})} \leq C\varepsilon_2 \|v\|_{L^2(\mathbb{T})}$ for all $t\in[0,T]$. By a perturbative argument, we deduce observability for this equation from observability for \eqref{eq:ReducedEq}. In order to do so, we decompose $v=v_1+v_2$, where $v_1$ and $v_2$ solve
\begin{equation*}
\begin{cases}
\partial_t v_1 + \mathrm{i} A_{0,\beta}v_1=0,\\
v_1(0)=v_0 \coloneqq (\mathcal{A}^{-1}w)(0),
\end{cases}
\qquad
\begin{cases}
\partial_t v_2 + \mathrm{i} A_{0,\beta}v_2+\mathcal{R}v_2=-\mathcal{R}v_1,\\
v_2(0)=0.
\end{cases}
\end{equation*}
By \eqref{eq:RealAssFull} and by Lemma 7.6 of \cite{alazard2018control} we have
\begin{equation} \label{eq:RealEstFull}
|\mathrm{Re} \, \langle v_0\rangle|\geq c |\langle v_0 \rangle| - \varepsilon_1 \|v_0\|_{L^2(\mathbb{T})} ,
\end{equation}
where $\varepsilon_1$ is the quantity appearing in Corollary \ref{cor:RedObs}. As a consequence, from Corollary \ref{cor:RedObs} we deduce that
\begin{equation*}
\int_0^T\int_{\omega} |\mathrm{Re} \, (\mathcal{A} v_1)|^2 \, \mathrm{d}x \, \mathrm{d}t
\geq K\int_0^{2\pi}|v_0(x)|^2\,\mathrm{d}x .
\end{equation*}
On the other hand, since $\mathcal{A}$ is bounded on $L^2(\mathbb{T})$, we deduce that
\begin{align*}
\int_0^T\int_{\omega} |\mathrm{Re}(\mathcal{A}v_2)|^2\, \mathrm{d}x \, \mathrm{d}t
&\leq \int_0^T\| \mathcal{A}v_2(t)\|_{L^2(\mathbb{T})}^2 \mathrm{d}t
\leq T\sup_{[0,T]}\|\mathcal{A}v_2(t)\|_{L^2(\mathbb{T})}^2 \\
&\leq C T\|v_2\|_{C^0([0,T];L^2(\mathbb{T}))}^2
\leq C T^3\varepsilon_2^2\|v_0\|_{L^2(\mathbb{T})}^2.
\end{align*}
Using the inequality $(x+y)^2\geq \frac{1}{2} x^2-y^2$, for $\varepsilon_2$ sufficiently small we get
\begin{equation*}
\int_0^T\int_{\omega} |\operatorname{Re}(\mathcal{A}v)|^2\, \mathrm{d}x \, \mathrm{d}t
\geq \frac{K}{4}\int_0^{2\pi}|v_0(x)|^2 \, \mathrm{d}x.
\end{equation*}
Since $\mathcal{A}v=w$ and $\|w_0\|_{L^2(\mathbb{T})}=\|\mathcal{A}v_0\|_{L^2(\mathbb{T})}\leq C\|v_0\|_{L^2(\mathbb{T})}$, we obtain
\begin{equation*}
\int_0^T\int_\omega |\mathrm{Re} \, w|^2\, \mathrm{d}x \, \mathrm{d}t
\geq \frac{K}{4}\int_0^{2\pi}|v_0(x)|^2 \mathrm{d}x
\geq K'\int_0^{2\pi}|w_0(x)|^2 \mathrm{d}x,
\end{equation*}
which leads to the thesis.

\end{proof}

\section{Controllability} \label{sec:Control}

In this section we fix $R>0$, a sufficiently large $s>0$, a bottom topography $\beta \in \mathcal{G} \subseteq B_{H^{s+1/2}}(R)$ ($\mathcal{G}$ is the set introduced in Lemma \ref{lem:LowFreq}), and we consider an operator of the form
\begin{equation*}
Q \coloneqq \partial_t + W\partial_x + \mathrm{i} A_{0,\beta} + \mathscr{R},
\end{equation*}
where $W$ is a real-valued function and $\mathscr{R}$ is an operator of order $0$. We study the following control problem: given a time $T>0$, a subset $\omega\subset\mathbb{T}$ and an
initial data $w_0\in L^2(\mathbb{T})$, we want to find a complex-valued function $f\in C^0( [0,T];L^2(\mathbb{T}) )$ such that the unique solution $w\in C^0( [0,T];L^2(\mathbb{T}) )$ of
\begin{equation} \label{eq:ControlEq}
Qw=\chi_\omega \operatorname{Re} f,\qquad w(0)=w_{0}
\end{equation}
satisfies $w(T)=0$. We study this control problem by using a HUM-type method; we adapt the standard argument, since we want to prove the existence of a real-valued control, while the unknown is complex-valued. For this reason one cannot obtain $w(T)=0$. Instead, we prove that
for any real-valued function $M$ such that the $L^{\infty}(\mathbb{T})$-norm of $M-1$ is small enough, one can find a control such that $w(T,x)=\mathrm{i} \, d M(x)$ for some constant $d \in\mathbb{R}$. 

Unlike \cite{alazard2018control}, where the authors prove regularity and stability estimates for the control problem (this result is non-trivial, as the water waves system \eqref{eq:WWsys} is quasi-linear, see parts (iv) and (v) of Proposition 8.1 in \cite{alazard2018control}), we do not prove stability estimates, as we later use the Nash--Moser--H\"ormander Theorem \ref{thm:NashMoser} in order to deduce Theorem \ref{thm:main} (see also the comment at the end of Sec. 2 in \cite{alazard2016controllability}).

Before stating the main technical result, we recall the definition of the adjoint
operator $Q^{\ast}$, namely
\begin{equation} \label{eq:Adjoint}
Q^{\ast}=-\mathcal{Q},\qquad \mathcal{Q} \coloneqq \partial_t + W\partial_x + \mathrm{i}A_{0,\beta} + \mathscr{R}_2, \qquad 
\mathscr{R}_2 \coloneqq -\mathscr{R}^{\ast}+(\partial_x W).
\end{equation}

\begin{proposition} \label{prop:ControlLin}

Fix $R>0$, a sufficiently large $s>0$, a bottom topography $\beta \in \mathcal{G} \subseteq B_{H^{s+1/2}}(R)$ (where $\mathcal{G}$ is the set appearing in the statement of Lemma \ref{lem:LowFreq}) for which there exists $h_0>0$ such that \eqref{eq:StrConnAss} holds true, and an open domain $\omega\subset\mathbb{T}$. 

Then there exist $r$ and four increasing functions $\mathcal{F}_j:\mathbb{R}_+\to\mathbb{R}_+$ $(0\leq j\leq 3)$, satisfying $\lim_{T \to 0}\mathcal{F}_j(T)=0$, such that, for
any $T>0$, for any real-valued function $M\in H^{3/2}(\mathbb{T})$ with $\|M-1\|_{L^\infty(\mathbb{T})}\leq \mathcal{F}_0(T)$,
the following results hold.

\noindent\emph{i) Existence.} Consider $R\in C^0([0,T];\mathcal{L}( L^2(\mathbb{T}) ))$ and a function $W$ satisfying
\begin{equation*}
\int_{\mathbb{T}} W(t,x)\,\mathrm{d}x=0
\end{equation*}
for any $t\in[0,T]$. Assume that the norm
\begin{align*}
& \|(W,R)\|_{r,T} \\
&\coloneqq \sum_{1\leq k\leq 3}\|\partial_t^k W\|_{C^0([0,T];H^1(\mathbb{T}) )} + \|W\|_{C^0([0,T];H^r(\mathbb{T}))} + \|R\|_{C^0([0,T];\mathcal{L}( L^2(\mathbb{T}) ))}
\end{align*}
satisfies
\begin{equation} \label{eq:AssF1}
\|(W,R)\|_{r,T}\leq \mathcal{F}_1(T).
\end{equation}
Then there exists an operator $\Theta_{M,T}:L^2\to C^0([0,T];L^2(\mathbb{T}) )$ such that for any $w_{0}\in L^2(\mathbb{T})$,
setting $f:=\Theta_{M,T}(w_{0})$, the unique solution $w\in C^0([0,T];L^2(\mathbb{T}) )$ of
\begin{equation} \label{eq:QOp}
Qw=\chi_\omega \mathrm{Re} \, f,\qquad w(0)=w_{0}
\end{equation}
satisfies
\begin{equation} \label{eq:timeT}
w(T,x) = \mathrm{i} \, d M(x)
\end{equation}
for some constant $d\in\mathbb{R}$, and
\begin{equation} \label{eq:ControlEst}
\|f\|_{C^0([0,T];L^2(\mathbb{T}) )}\leq \frac{\|w_{0}\|_{L^2(\mathbb{T})}}{\mathcal{F}_2(T)} .
\end{equation}

\noindent\emph{ii) Uniqueness.} For any $w_{0}\in L^2(\mathbb{T})$ and any $T>0$, $\Theta_{M,T}(w_{0})$ is determined
as the unique function $f\in C^0([0,T];L^2(\mathbb{T}))$ satisfying the two following conditions:
\begin{enumerate}
\item There holds $Q^{\ast}f=0$ and $\mathrm{Im} \, \int_\mathbb{T}M(x)f(T,x)\,dx=0$.
\item The solution $w$ of \eqref{eq:QOp} satisfies \eqref{eq:timeT} for some constant $d\in\mathbb{R}$.
\end{enumerate}

\noindent\emph{iii) Regularity.} Let $\mu\in[0,3/2]$ and consider $w_{0}\in H^\mu(\mathbb{T})$. If
\begin{equation} \label{eq:AssF1Reg}
\|(W,R)\|_{r,T}+\|R\|_{C^0([0,T];\mathcal{L}(H^\mu(\mathbb{T}) ))} \leq \mathcal{F}_1(T),
\end{equation}
then $\Theta_{M,T}(w_{0}) \in C^0([0,T];H^\mu(\mathbb{T}))$, and
\begin{equation} \label{eq:ControlEstReg}
\|\Theta_{M,T}(w_{0})\|_{C^0([0,T];H^\mu(\mathbb{T}) )}\leq \frac{\|w_{0}\|_{H^\mu(\mathbb{T})}}{\mathcal{F}_3(T)} .
\end{equation}

\end{proposition}

In this section we often use the notation $A\lesssim B$ to say that $A\leq CB$ for some
constant $C$ depending only on $T$ and on $\|\beta\|_{H^{s+1/2}(\mathbb{T})}$. The key result is the following.

\begin{lemma}	\label{lem:ControlTec}

Fix $R>0$, a sufficiently large $s>0$, a bottom topography $\beta \in \mathcal{G} \subseteq B_{H^{s+1/2}}(R)$ (where $\mathcal{G}$ is the set appearing in the statement of Lemma \ref{lem:LowFreq}) for which there exists $h_0>0$ such that \eqref{eq:StrConnAss} holds true, and an open domain $\omega\subset\mathbb{T}$. 

Let us define the space
\begin{equation*}
L^2_M \coloneqq \left\{\varphi\in L^2(\mathbb{T};\mathbb{C}):\ \operatorname{Im}\int_\mathbb{T}M(x)\varphi(x)\,dx=0\right\}.
\end{equation*}
For any $w_{0}\in L^2(\mathbb{T})$, there exists a unique $f_1\in L^2_M$ such that,
\begin{equation*}
\forall \phi_1\in L^2_M,\qquad
\operatorname{Re}\int_0^T (\chi_\omega \operatorname{Re} f(t),\phi(t))\,\mathrm{d}t=-\operatorname{Re}(w_{0},\phi(0)),
\end{equation*}
where $f$ and $\phi$ are the unique functions in $C^0([0,T];L^2(\mathbb{T}))$ satisfying
\begin{equation} \label{eq:QSys}
\begin{cases}
\mathcal{Q}f=0,\\
f(T)=f_1,
\end{cases}
\qquad
\begin{cases}
\mathcal{Q}\phi=0,\\
\phi(T)=\phi_1,
\end{cases}
\end{equation}
where $\mathcal{Q}$ is given by \eqref{eq:Adjoint}. We set $\Theta_{M,T}(w_{0}) \coloneqq f$. Moreover, \eqref{eq:ControlEst} holds.

\end{lemma}

\begin{proof}
	
The space $L^2_M$ is an $\mathbb{R}$-vector space. Introduce the $\mathbb{R}$-bilinear symmetric map
$a(\cdot,\cdot)$ defined by
\begin{align} \label{eq:aFunc}
a(f_1,\phi_1) &\coloneqq \mathrm{Re} \, \int_0^T\int_{\mathbb{T}}\chi_\omega(x)\mathrm{Re} \, (f(t,x))\overline{\phi(t,x)}\,\mathrm{d}x \, \mathrm{d}t .
\end{align}
This map is well defined and continuous. Moreover,
\begin{align} 
|a(f_1,\phi_1)| &\leq \int_0^T\int_{\mathbb{T}} |f|\,|\phi|\, \mathrm{d}x \, \mathrm{d}t \nonumber \\
&\leq T\|f\|_{C^0([0,T];L^2(\mathbb{T}))}\|\phi\|_{C^0([0,T];L^2(\mathbb{T}))}\leq C(T)\|f_1\|_{L^2(\mathbb{T})}\|\phi_1\|_{L^2(\mathbb{T})}.
\label{eq:aFuncEst}
\end{align}
Since $\chi_\omega(x)=1$ for $x$ in an open subset $\omega_1\subset\omega$, one has
\begin{equation*}
a(f_1,f_1)\geq \int_0^T\int_{\omega_1}(\mathrm{Re} \, f)^2\,\mathrm{d}x \, \mathrm{d}t.
\end{equation*}
If $f_1\in L^2_M$ then $\mathrm{Im} \, \int_{\mathbb{T}} M f_1\,\mathrm{d}x=0$ and we have
\begin{equation*}
\left|\mathrm{Im} \, \int_{\mathbb{T}} f_1(x)\,\mathrm{d}x\right| = \left|\mathrm{Im} \, \int_{\mathbb{T}} (1-M(x))f_1(x)\,\mathrm{d}x\right| \leq \|M-1\|_{L^\infty(\mathbb{T})} \sqrt{2\pi} \|f_1\|_{L^2(\mathbb{T})}
\end{equation*}
from which (using $|\mathrm{Re} \, z |\geq |z|-|\mathrm{Im} \, z|$) we deduce that
\begin{equation*}
|\mathrm{Re} \, \langle f_1\rangle|\geq |\langle f_1\rangle|-\|M-1\|_{L^\infty(\mathbb{T})} \sqrt{2\pi} \|f_1\|_{L^2(\mathbb{T})}.
\end{equation*}
If $\|M-1\|_{L^\infty(\mathbb{T})}$ is sufficiently small, one can apply the observability inequality proved in Corollary \ref{cor:ObsFull} (see also Remark \ref{rem:ObsFull}) in order to deduce that
\begin{equation} \label{eq:f1Est}
C_1(T)\|f_1\|_{L^2(\mathbb{T})}^2\leq a(f_1,f_1).
\end{equation}
On the other hand, \eqref{eq:aFuncEst} implies that $a(f_1,f_1)\leq C(T)\|f_1\|_{L^2(\mathbb{T})}^2$. Hence $a(\cdot,\cdot)$ is
a real scalar product on $L^2_M$ which induces the norm $N(f_1)=\sqrt{a(f_1,f_1)}$, which
is equivalent to the norm $\|\cdot\|_{L^2(\mathbb{T},\mathbb{C})}$ on $L^2_M$. Now, arguing as in Lemma 8.2 of \cite{alazard2018control}, the Riesz theorem implies that, for any $\mathbb{R}$-linear form $\Lambda$ on $L^2_M$, there exists a unique $f_1\in L^2_M$
such that $a(f_1,\phi_1)=\Lambda(\phi_1)$ for all $\phi_1\in L^2_M$, and 
\begin{equation} \label{eq:f1Est2}
\|f_1\|_{L^2(\mathbb{T})}\leq \frac{\|\Lambda\|}{C_1(T)}.
\end{equation}
Moreover, by \eqref{eq:f1Est2} and the bound $\|f\|_{C^0([0,T];L^2(\mathbb{T}))}\lesssim \|f_1\|_{L^2(\mathbb{T})}$  we can deduce \eqref{eq:ControlEst}. 

\end{proof}

\begin{proof}[Proof of Proposition 8.1. Proof of statement i).] 
	
%We begin by proving that if $M\in H^{3/2}(\mathbb{T})$ then $H^{3/2}(\mathbb{T})\cap L^2_M$ is dense in $(L^2_M,\|\cdot\|_{L^2(\mathbb{T})})$. To see this, let $\Pi_N$ be the Fourier truncation operator defined by $\Pi_N h(x)=\sum_{|j|\leq N}h_j e^{\mathrm{i}jx}$, where $h(x)=\sum_{j\in\mathbb{Z}}h_j e^{\mathrm{i}jx}$.
%Given $u\in L^2_M$, define $u_N \coloneqq \frac{1}{M}\Pi_N(Mu)$. Since the operator $\Pi_N$ preserves the average, one has that $u_N\in L^2_M$. Moreover, since $u\in L^2$, one has $Mu\in L^2(\mathbb{T})$, $\Pi_N(Mu)\in C^\infty(\mathbb{T})$, and hence $M^{-1}\Pi_N(Mu)\in H^{3/2}(\mathbb{T})$ since $M^{-1}\in H^{3/2}(\mathbb{T})$.
%Since $(u_N)$ converges to $u$, this proves that $H^{3/2}(\mathbb{T})\cap L^2_M$ is dense in $L^2_M$.
%Now let $f$ be as given by the previous lemma. It is proved in the appendix that there is a unique solution $w$ in $C^0([0,T];L^2(\mathbb{T}))$ of (8.4). 

Arguing as in the proof of Proposition 8.1 in \cite{alazard2018control} one can show that $H^{3/2}(\mathbb{T})\cap L^2_M$ is dense in $L^2_M$, and that there exists a unique $w \in C^0([0,T];L^2(\mathbb{T}))$ which solves \eqref{eq:QOp}. Next, we want to show that $w(T)$ satisfies \eqref{eq:timeT}. In order to do so, we first check that \eqref{eq:timeT} will be proved if $\mathrm{Re}(w(T),\phi_1)=0$ for all $\phi_1$ in $L^2_M$. Indeed, one has
\begin{equation*}
\mathrm{Re}(w(T),\phi_1)=\int (\mathrm{Re}w(T,x))\mathrm{Re}\phi_1(x)\,\mathrm{d}x+
\int (\mathrm{Im}w(T,x))\mathrm{Im}\phi_1(x)\,\mathrm{d}x=0
\end{equation*}
for all $\phi_1\in L^2_M$. Therefore we obtain $\int (\mathrm{Re}w(T,x))f(x)\,\mathrm{d}x=0$ for any real-valued function $f$ and $\int (\mathrm{Im}M(x)^{-1}w(T,x))g(x)\,\mathrm{d}x=0$ for any real-valued function $g$ with $\int g(x)\,\mathrm{d}x=0$. This implies that \eqref{eq:timeT} holds.

We now have to prove that $\mathrm{Re}(w(T),\phi_1)=0$ for any $\phi_1$ in $L^2_M$. By the density of $L^2_M\cap H^{3/2}(\mathbb{T})$ in $L^2_M$, it is enough to assume that $\phi_1\in L^2_M\cap H^{3/2}(\mathbb{T})$. Given $\phi_1\in L^2_M\cap H^{3/2}(\mathbb{T})$, let $\phi\in C^0([0,T];H^{3/2}(\mathbb{T}))$ be such that
\begin{equation} \label{eq:QtimeTEq}
\mathcal{Q}\phi=0, \qquad \phi(T)=\phi_1.
\end{equation}
Since $\mathcal{Q}=-Q^{\ast}$, by multiplying the equation \eqref{eq:QOp} by $\overline{\phi}$ and by integrating by parts, we find that
\begin{equation} \label{eq:wTIntParts}
(w(T),\phi_1)=(w(0),\phi(0))+\int_0^T(\chi_\omega\operatorname{Re}f,\phi)\,\mathrm{d}t+
\int_0^T(w,\mathcal{Q}\phi)\,\mathrm{d}t,
\end{equation}
where we used that $\phi\in C^1([0,T];L^2(\mathbb{T}))$. By definition of $\phi$, the last term in the right-hand side vanishes and, by definition of $f$, the real part of the sum of the first two terms vanishes. Hence, we obtain that $\mathrm{Re}(w(T),\phi_1)=0$, which concludes the proof of statement \emph{i)}.

\end{proof}

\begin{proof}[Proof of statement ii).] 
	
Recall that $\mathcal{Q}=-Q^{\ast}$ is given by \eqref{eq:Adjoint}. Consider $\phi_1\in L^2_M$
and denote by $\phi$ the unique function in $C^0([0,T];L^2(\mathbb{T}))$ satisfying \eqref{eq:QtimeTEq}. As in \eqref{eq:wTIntParts}, multiplying both sides of the equation $Qw=\chi_\omega \mathrm{Re}f$ by $\overline{\phi}$ and by integrating by parts, we obtain \eqref{eq:wTIntParts}. Since $\phi_1\in L^2_M$ and $w(T,x)=\mathrm{i}d M(x)$ for some constant
$d\in\mathbb{R}$, one has $\mathrm{Re}(w(T),\phi_1)=0$. Therefore, since $\mathcal{Q}\phi=0$,
\begin{equation*}
\mathrm{Re}\int_0^T(\chi_\omega\mathrm{Re} \, f,\phi)\,\mathrm{d}t = -\mathrm{Re}(w_{0},\phi(0)).
\end{equation*}
Since $\mathcal{Q}f=0$ and $f(T)\in L^2_M$ by assumption, and since the function $f_1$ given by Lemma \ref{lem:ControlTec} is unique, we can deduce that $f(T)=f_1$. Hence $f=\Theta_{M,T}(w_{0})$, by uniqueness of the solution to the Cauchy problem \eqref{eq:QSys}.	

\end{proof}

\begin{proof}[Proof of statement iii).]
	
We want to prove \eqref{eq:ControlEstReg}. Arguing as in Proposition 8.1 of \cite{alazard2018control}, it suffices to prove that $\|f_1\|_{H^{\mu}(\mathbb{T})}$ is controlled by $\|w_{0}\|_{H^{\mu}(\mathbb{T})}$. We only prove an a priori estimate, assuming that $f_1$ belongs to $H^\mu(\mathbb{T})$. 

Let us consider the map
\begin{equation*}
S:L^2_M\to L^2(\mathbb{T}),\qquad S:f_1 \mapsto  w_0 ,
\end{equation*}
with $w_0=w(0)$, and where $w$ is determined by the following Cauchy problems
\begin{equation*}
\begin{cases}
\mathcal{Q}f=0\\
f(T)=f_1,
\end{cases}
\qquad
\begin{cases}
\mathcal{Q}w=\chi_\omega\operatorname{Re}f\\
w(T)=0.
\end{cases}
\end{equation*}
It follows from statements \emph{i)} and \emph{ii)} that $S$ is an isomorphism of $L^2_M$ onto $L^2(\mathbb{T})$. Hence, writing $\Lambda^\mu=(\mathrm{Id}-\partial_x^2)^{\mu/2}$, we have
\begin{equation} \label{eq:Estf1}
\|f_1\|_{H^{\mu}(\mathbb{T})} =\|\Lambda^\mu f_1\|_{L^2(\mathbb{T})} \leq C \, \|S\Lambda^\mu f_1\|_{L^2(\mathbb{T})}.
\end{equation}
Now we have to estimate the action of the operator $S \Lambda^\mu$ in \eqref{eq:Estf1}: in order to do so, it suffices to compare $(\Lambda^\mu f,\Lambda^\mu w)$ with
$(f',w')$ given by
\begin{equation*}
\begin{cases}
\mathcal{Q}f'=0\\
f'(T)=\Lambda^\mu f_1,
\end{cases}
\qquad
\begin{cases}
\mathcal{Q}w'=\chi_\omega\mathrm{Re} \, f'\\
w'(T)=0.
\end{cases}
\end{equation*}
We first estimate $f'-\Lambda^\mu f$, and then deduce an estimate for $w'-\Lambda^\mu w$: we write
\begin{equation*}
\mathcal{Q}(f'-\Lambda^\mu f)=[\Lambda^\mu,\mathscr{R}_2]f+[\Lambda^\mu,W]\partial_x f,
\qquad
(f'-\Lambda^\mu f)\big|_{t=T}=0,
\end{equation*}
so that
\begin{equation} \label{eq:EstControlTech1}
\|f'-\Lambda^\mu f\|_{C^0([0,T];L^2(\mathbb{T}))}\leq C \,  \|[\Lambda^\mu,\mathscr{R}_2]f+[\Lambda^\mu,W]\partial_x f\|_{L^1([0,T];L^2(\mathbb{T}))}.
\end{equation}
Similarly,
\begin{align} 
& \|w'-\Lambda^\mu w\|_{C^0([0,T];L^2(\mathbb{T}))} \leq C \,  \|\mathcal{F}\|_{L^1([0,T];L^2(\mathbb{T}))} ,  \label{eq:EstControlTech2} \\
& \mathcal{F} \coloneqq \chi_\omega\mathrm{Re}(f'-\Lambda^\mu f) + [\Lambda^\mu,\mathscr{R}_2]w +[\Lambda^\mu,W]\partial_x w -[\Lambda^\mu,\chi_\omega]\mathrm{Re} \, f. \nonumber 
\end{align}
By \eqref{eq:EstControlTech1}, we can deduce that
\begin{align*}
\|\mathcal{F}\|_{L^1([0,T];L^2(\mathbb{T}))} &\leq C \, \bigg[ \|[\Lambda^\mu,\chi_\omega]\mathrm{Re} \, f\|_{C^0([0,T];L^2(\mathbb{T}))}\\
&\quad +\|[\Lambda^\mu,\mathscr{R}_2]f\|_{C^0([0,T];L^2(\mathbb{T}))} + \|[\Lambda^\mu,\mathscr{R}_2]w\|_{C^0([0,T];L^2(\mathbb{T}))}\\
&\quad +\|[\Lambda^\mu,W] f_x\|_{C^0([0,T];L^2(\mathbb{T}))} +\|[\Lambda^\mu,W] w_x \|_{C^0([0,T];L^2(\mathbb{T}))} \bigg].
\end{align*}
In order to estimate the commutators $[\Lambda^\mu,\chi_\omega]$ and $[\Lambda^\mu,W]$, we exploit that
\begin{equation*}
s>\frac{3}{2},\quad 0\leq \mu\leq s\quad\Rightarrow\quad \|[\Lambda^\mu,W]u\|_{L^2(\mathbb{T})} \leq K\|W\|_{H^s(\mathbb{T})} \|u\|_{H^{\mu-1}(\mathbb{T})}.
\end{equation*}
On the other hand, in order to estimate the commutator $[\Lambda^\mu,\mathcal{R}]$ (or $[\Lambda^\mu,R]$) we estimate separately $\Lambda^\mu\mathcal{R}$ and $\mathcal{R}\Lambda^\mu$. Recalling the definition of $\mathscr{R}_2$, we can deduce that
\begin{align*}
\|\mathcal{F}\|_{L^1([0,T];L^2(\mathbb{T}))} &\leq C \, \bigg[ \|f\|_{C^0([0,T];H^{\mu-1}(\mathbb{T}))}+a\|(f,w)\|_{C^0([0,T];H^\mu(\mathbb{T}))} \bigg] , \\
a &\coloneqq \|W\|_{C^0([0,T];H^3(\mathbb{T}))}+\|\mathscr{R}\|_{C^0([0,T];\mathcal{L}(H^\mu(\mathbb{T}))\cap \mathcal{L}(L^2(\mathbb{T})))}. \nonumber 
\end{align*}
Notice that $a\leq \|(W,\mathscr{R})\|_{r,T}+\|\mathscr{R}\|_{C^0([0,T];\mathcal{L}(H^\mu(\mathbb{T})))}$.

By \eqref{eq:EstControlTech2}, we get
\begin{equation*}
\|w'-\Lambda^\mu w\|_{C^0([0,T];L^2(\mathbb{T}))} \leq C \, \bigg[ \|f_1\|_{H^{\mu-1}(\mathbb{T})} + a \|f_1\|_{H^\mu(\mathbb{T})} \bigg] ,
\end{equation*}
which implies at $t=0$ that 
\begin{equation*}
	\|w'(0)-\Lambda^\mu w_0 \|_{L^2(\mathbb{T})} \leq C \, \bigg[ \|f_1\|_{H^{\mu-1}(\mathbb{T})} + a \|f_1\|_{H^\mu(\mathbb{T})} \bigg] .
\end{equation*}
Now, by definition, $w'(0)=S\Lambda^\mu f_1$ while $\Lambda^\mu w(0)=\Lambda^\mu w_{0}$. Therefore,
\begin{equation} \label{eq:EstControlTech3}
\|S\Lambda^\mu f_1\|_{L^2(\mathbb{T})} \leq \|w_{0}\|_{H^\mu(\mathbb{T})} +\|f_1\|_{H^{\mu-1}(\mathbb{T})} + a\|f_1\|_{H^\mu(\mathbb{T})}.
\end{equation}
For any $\mu \in [0,1]$, we have $\|f_1\|_{H^{\mu-1}(\mathbb{T})} \leq \|f_1\|_{L^2(\mathbb{T})} \leq C \, \|w_{0}\|_{L^2(\mathbb{T})}$, hence
\begin{equation} \label{eq:EstControlTech4}
\|S\Lambda^\mu f_1\|_{L^2} \leq C \, \bigg[ \|w_{0}\|_{H^\mu(\mathbb{T})}+a\|f_1\|_{H^\mu(\mathbb{T})} \bigg].
\end{equation}
Plugging this bound into \eqref{eq:Estf1} yields $\|f_1\|_{H^\mu} \leq C \, \bigg[  \|w_{0}\|_{H^\mu(\mathbb{T})}+a\|f_1\|_{H^\mu(\mathbb{T})} \bigg]$. By taking $a$
sufficiently small, we can deduce that
\begin{equation*}
\|f_1\|_{H^\mu (\mathbb{T})} \leq C \, \|w_{0}\|_{H^\mu(\mathbb{T})}.
\end{equation*}
For any $\mu\in(1,\frac{3}{2}]$, by considering \eqref{eq:EstControlTech3} we can deduce that $\|f_1\|_{H^{\mu-1}(\mathbb{T})} \leq C \, \|w_{0}\|_{H^{\mu-1}(\mathbb{T})}$. Hence \eqref{eq:EstControlTech4} holds, and we obtain the same conclusion as above. This
completes the proof of statement \emph{iii)}.

\end{proof}

\section{Proof of Theorem \ref{thm:main} } \label{sec:NashMoser}

In this section we deduce Theorem \ref{thm:main} using the Nash--Moser--H\"ormander Theorem \ref{thm:NashMoser} stated in Appendix \ref{sec:NM}, assuming the estimates and reductions proved in the previous sections. 

Fix $T>0$, a non-empty open set $\Omega \subset \mathbb{T}$, $R>0$, and let $s \geq s_0$ be sufficiently large. Let $\beta \in \mathcal{G} \subset B_{H^{s+1/2}}(R)$ satisfy the condition \eqref{eq:StrConnAss}, where $\mathcal{G}$ is the set appearing in the statement of Lemma \ref{lem:LowFreq}.

Given small initial and final data
\begin{equation*}
(\eta_0,\psi_0),\ (\eta_f,\psi_f) \in H^{s+1/2}_0(\mathbb{T}) \times H^s(\mathbb{T}),
\end{equation*}
we first pass to the symmetrized variable \eqref{eq:TransfGU}, defining
\begin{align*}
u_0 \coloneqq U(\eta_0,\psi_0), &\qquad u_f \coloneqq U(\eta_f,\psi_f), \\
U(\eta,\psi) = T_{\mathtt{p}} \omega - \mathrm{i} T_{\mathtt{q}} \eta, &\qquad \omega = \psi - T_{B(\eta,\beta)}\psi\eta ,
\end{align*}
see \eqref{eq:GoodUnB}. By Lemma \ref{lem:InvertUn}, if $\| (\eta,\psi) \|_{H^{s+1/2}(\mathbb{T}) \times H^s(\mathbb{T})}$ is sufficiently small, this change of variables is locally invertible from $H^{s+1/2}(\mathbb{T}) \times H^s(\mathbb{T})$ into
\begin{equation*}
\tilde{H}^s(\mathbb{T};\mathbb{C})
= \left\{ u \in H^s(\mathbb{T};\mathbb{C}) : \int_{\mathbb{T}} \mathrm{Im} \, u \, \mathrm{d}x = 0 \right\} .
\end{equation*}
Moreover, the inverse $Y_{\beta}$ satisfies tame bounds
\begin{equation*}
\|\eta\|_{H^{s+1/2}(\mathbb{T})} + \|\psi\|_{H^s(\mathbb{T})} \leq K( \|\beta\|_{H^{s+1/2}(\mathbb{T})} ) \|u\|_{H^s(\mathbb{T})} ,
\end{equation*}
which we use in order to solve the controllability problem for the scalar equation in the unknown $u$. Corollary \ref{cor:ParalinFull} gives the equation \eqref{eq:uEqNew}, namely
\begin{equation*}
	u_t + T_{V(u,\beta)} u_x + \mathrm{i} A_0^{1/2}\left(T_{c(u,\beta)} A_0^{1/2} \left[ 1+G_{00}^{-1} DL(\beta) \right] u\right) + R(u,\beta,\kappa)u
	= T_{\mathtt{p}(u,\beta)}\mathscr{P}_{ext} .
\end{equation*}
The estimate \eqref{eq:EstVubeta} shows that $V(u,\beta)$ is small when $u$ is small, and by Remark \ref{rem:StrucRem} we have that the remainder $R(\underline{u},\beta,\kappa)$ is of order zero, and is linear in $u$.

Now choose a real cutoff $\chi_\Omega \in C^\infty(\mathbb{T})$, with compact support $\mathrm{supp} \, \chi_\Omega \subset \Omega$. We look for a control of the form
\begin{equation*}
\mathscr{P}_{ext}(t,x) = \chi_\Omega(x)\,\mathrm{Re} \, f(t,x),
\end{equation*}
where $f$ is complex-valued. This is the form used in the linear control problem for the paradifferential operator $P_{0,\beta}$ given by \eqref{eq:OpP0beta}, where the controlled equation is given by \eqref{eq:CauchyProb}, namely
\begin{equation*}
P_{0,\beta}v = T_{\mathtt{p}} \chi_\Omega \mathrm{Re} \, f, \qquad v(0)=v_0, \qquad v(T)=0.
\end{equation*}
Proposition \ref{prop:ParaControl} states that, under the smallness assumptions on $V,c,R$, the above problem is solvable in $\tilde{H}^s(\mathbb{T})$, with
\begin{align*}
\|f\|_{C^0( [0,T];H^s(\mathbb{T}) )} &\leq K\|v_0\|_{\tilde{H}^s(\mathbb{T})} .
\end{align*}
Notice that the smallness assumptions in Proposition \ref{prop:ParaControl} are imposed only in a fixed low norm.

We now formulate the nonlinear problem as an implicit equation. Let
\begin{align*}
%X_a &\coloneqq \left\{ u \in C^0( [0,T];\tilde{H}^a(\mathbb{T};\mathbb{C}) ) : u \text{ solves } \eqref{eq:uEqNew} \right\}, \\
X_a &\coloneqq C^0( [0,T];\tilde{H}^a(\mathbb{T};\mathbb{C}) ), \\
Y_a &\coloneqq C^0( [0,T];H^a(\mathbb{T}) ),
\end{align*}
and set
\begin{equation*}
E_a \coloneqq X_a \times Y_a, \qquad
F_a \coloneqq C^0( [0,T];H^{a-\mu}(\mathbb{T}) ) \times \tilde{H}^a(\mathbb{T};\mathbb{C}) \times \tilde{H}^a(\mathbb{T};\mathbb{C}),
\end{equation*}
where $\mu>0$ is chosen such that it is larger than the loss produced by the leading order principal part and by the tame estimates. The exact value of $\mu$ is not important, provided $s_0$ is large enough in order to satisfy the inequalities in Appendix \ref{sec:NM}.

Let $(u,f)$ belong to a small ball around the origin in $E_\mu$, define
\begin{equation} \label{eq:PhiMap}
	\Phi(u,f) \coloneqq \left( P(u)u - T_{\mathtt{p}(u,\beta)}\chi_\Omega \mathrm{Re} \, f,\ u(0),\ u(T) \right),
\end{equation}
where
\begin{equation*}
P(u)u \coloneqq u_t + T_{V(u,\beta)} u_x
+ \mathrm{i} A_0^{1/2}\left(T_{c(u,\beta)} A_0^{1/2} \left[ 1+G_{00}^{-1}DL(\beta) \right] u \right) + R(u,\beta,\kappa)u .
\end{equation*}
Solving the nonlinear control problem is equivalent to solving
\begin{equation} \label{eq:EqNM}
	\Phi(u,f) = (0,u_0,u_f) .
\end{equation}
Since the initial and final data are small, the right-hand side of \eqref{eq:EqNM} is small in the relevant norm.

Now we have to check the assumptions of Theorem \ref{thm:NashMoser} in Appendix \ref{sec:NM}. First, the scale $E_a,F_a$ is equipped with the usual Fourier cutoffs $S_j$; hence, the smoothing estimates \eqref{S1}--\eqref{2705.4} hold. Secondly, the map $\Phi$ is $C^2$ tame on a small ball: this follows from the analyticity and tame bounds for the Dirichlet--Neumann operator (see Proposition \ref{prop:AnalOpG} and Proposition \ref{prop:DNOpara}), the paralinearization of Section \ref{sec:symm}, and the paradifferential calculus estimates in Appendix \ref{sec:Paradiff}. In particular, for $0 \leq a \leq a_2-\mu$ we obtain
\begin{equation} \label{eq:estPhi''}
	\|\Phi''(u)[h,k]\|_a
	\leq M(a)\left(
	\|h\|_{a+\mu}\|k\|_{a_0}
	+ \|h\|_{a_0}\|k\|_{a+\mu}
	+ \|u\|_{a+\mu}\|h\|_{a_0}\|k\|_{a_0}
	\right),
\end{equation}
which is the estimate required in \eqref{Phi sec}. 

It remains to construct a tame right inverse for $\Phi'(u,f)$. Let $(g,g_0,g_T)\in F_\infty$ be given. The linearized equation has the form
\begin{equation} \label{eq:Linearized}
	L_u v - T_{\mathtt{p}(u,\beta)}\chi_\Omega \mathrm{Re} \, \phi = g + \text{l.o.t.}, \qquad v(0)=g_0, \qquad v(T)=g_T,
\end{equation}
where $\mathrm{l.o.t.}$ denotes lower-order tame terms, and where $L_u$ has the same principal structure as $P_{0,\beta}$, namely
\begin{equation*}
L_u = \partial_t + T_{V(u,\beta)}\partial_x
+ \mathrm{i} A_0^{1/2} \left(T_{c(u,\beta)} A_0^{1/2} \left[ 1+G_{00}^{-1}DL(\beta)\right] \,\cdot \right) + \text{l.o.t.} ,
\end{equation*}
where the lower order terms are operators of order zero. 

If $\|u\|_{a_1}$ is sufficiently small, then the coefficients $V(u,\beta)$, $c(u,\beta)-1$, their time derivatives, and the order-zero remainder satisfy the smallness assumptions of Proposition \ref{prop:ParaControl}. Hence the controlled linear problem is solvable. The endpoint condition $v(T)=g_T$ is reduced to the null-controllability problem in Proposition \ref{prop:ParaControl} by subtracting a smooth lift $\xi(t)$ satisfying $\xi(0)=g_0$, $\xi(T)=g_T$; the equation for $v-\xi$ has the same form, with an additional source term. Applying Proposition \ref{prop:ParaControl} gives a solution $(v,\phi)$ satisfying
\begin{equation} \label{eq:TameRightInv}
	\|(v,\phi)\|_a
	\leq L(a)\left[
	\|(g,g_0,g_T)\|_{a+\mathtt{b}-\mathtt{a}}
	+ \|u\|_{a+\mathtt{b}}\|(g,g_0,g_T)\|_0
	\right] ,
\end{equation}
for $a_1 \leq a \leq a_2$. This is the tame right-inverse estimate \eqref{tame in NM}. The loss in \eqref{eq:TameRightInv} is finite, and it is absorbed by choosing $s_0,a_1,a_2,\mathtt{a},\mathtt{b}$ satisfying the numerical restrictions \eqref{ineq 2016}. The linear control input used here is precisely the result produced in Secs. \ref{sec:Red}--\ref{sec:Control}: Secs. \ref{sec:Red}--\ref{sec:Obs} reduce the problem to an observable constant-coefficient equation, and Sec. \ref{sec:Control} converts observability into controllability by HUM.

Therefore, we can apply Theorem \ref{thm:NashMoser} of Appendix \ref{sec:NM} to $\Phi$. The theorem assumes the smoothing scale, the second derivative tame estimate \eqref{Phi sec}, and a tame right inverse satisfying \eqref{tame in NM}; as a result, we obtain, for sufficiently small data $g$, a solution $U=(u,f)$ to $\Phi(U)=\Phi(0)+g$. Its higher-regularity conclusion also gives $U\in E_s$ whenever the data are in the corresponding $F_s$.

Applying this with
\begin{equation*}
g = (0,u_0,u_f) - \Phi(0)
\end{equation*}
yields a pair
\begin{equation*}
u \in C^0([0,T];\tilde{H}^s(\mathbb{T};\mathbb{C})), \qquad
f \in C^0([0,T];H^s(\mathbb{T})),
\end{equation*}
such that \eqref{eq:EqNM} holds true; equivalently, we have that
\begin{equation*}
P(u)u = T_{\mathtt{p}(u,\beta)}\chi_\Omega \mathrm{Re} \, f, \qquad u(0)=u_0, \qquad u(T)=u_f .
\end{equation*}
Setting
\begin{align*}
\mathscr{P}_{ext} &\coloneqq \chi_\Omega \mathrm{Re} \, f
\end{align*}
gives
\begin{equation*}
\mathscr{P}_{ext} \in C^0([0,T];H^s(\mathbb{T})), \qquad \mathrm{supp} \, \mathscr{P}_{ext}(t,\cdot) \subset \Omega .
\end{equation*}
Finally, define
\begin{equation*}
(\eta(t),\psi(t)) \coloneqq Y_{\beta}(u(t)).
\end{equation*}
By the inverse bounds of Lemma \ref{lem:InvertUn} and by the smallness of $u$, we have
\begin{equation*}
(\eta,\psi) \in C^0([0,T];H^{s+1/2}_0(\mathbb{T})\times H^s(\mathbb{T})).
\end{equation*}
The equivalence between the symmetrized equation and the original water-waves system, established in Corollary \ref{cor:ParalinFull}, shows that $(\eta,\psi)$ solves \eqref{eq:WWsys}. Moreover,
\begin{equation*}
(\eta(0),\psi(0))=(\eta_0,\psi_0), \qquad (\eta(T),\psi(T))=(\eta_f,\psi_f).
\end{equation*}
If $\varepsilon_0$ is sufficiently small, then the solution remains in the small neighbourhood where $Y_{\beta}$ is defined, the coefficient smallness assumptions remain valid, and the fluid domain $D_{\eta,\beta}(t)$ remains strictly connected. This proves Theorem \ref{thm:main}.

\section{The case with constant vorticity and flat bottom} \label{sec:constvort}

In this section we prove Theorem \ref{thm:mainConstVort}, highlighting the main differences compared to the proof of Theorem \ref{thm:main}.

Using \eqref{eq:G0Irr}, by introducing the Fourier multiplier
\begin{align} \label{eq:AgOp}
  A_{\gamma} &\coloneqq \frac{1}{2} G_{00} \left[ \frac{\gamma}{2} D^{-1} + \left( \left( \frac{\gamma}{2} D^{-1} \right)^2 +  4 (g-\kappa \partial_x^2) G_{00}^{-1} \right)^{1/2} \right] , \;\; G_{00} \coloneqq D \tanh(h D), 
\end{align}
and introducing the variable
\begin{align} \label{eq:uNewVarGamma}
u &\coloneqq \psi -\gamma \partial_x^{-1}\eta - \mathrm{i} A_{\gamma} \, G_{00}^{-1} \eta ,
\end{align}
we obtain
\begin{align}  \label{eq:uSymmLinGamma}
u_t + \mathrm{i} A_{\gamma} u  &= \mathscr{P}_{ext} ,
\end{align}
which has the same structure as in the irrotational case. For the next result, we introduce the symbols
\begin{align} 
\mathtt{A}_{\gamma}(\xi) &\coloneqq \frac{1}{2} \mathtt{G}_{00} \left[ \frac{\gamma}{2} \xi^{-1} + \left( \left( \frac{\gamma}{2} \xi^{-1} \right)^2 +  4 (g+\kappa |\xi|^2) \mathtt{G}_{00}^{-1} \right)^{1/2} \right] ,  \label{eq:AgSymb} \\
\mathtt{G}_{00}(\xi) &\coloneqq |\xi| \tanh(h |\xi|) , \nonumber \\
B(\eta) \psi &\coloneqq \frac{1}{1+\eta_x^2} \left[ G^{DN}(\eta)\psi + \eta_x \psi_x \right] , \nonumber \\
V(\eta) \psi &\coloneqq \psi_x - B(\eta)\psi . \nonumber
\end{align}

\begin{proposition} \label{prop:symmConstVort}

Let $T >0$. There exists $s_0 \gg 1$ such that the following holds true for any $s\geq s_0$. Let $(\eta,\psi) \in C( [0,T]; H^{s+1/2}_0(\mathbb{T}) \times H^s( \mathbb{T} ) )$ be a solution of \eqref{eq:WWsysV} such that \eqref{eq:StrConnectedFlat} holds true, and let us define 
\begin{align} 
\omega &\coloneqq \psi - T_{B(\eta)\psi} \eta \in C( [0,T] ; H^s(\mathbb{T}) ), \label{eq:GoodUn} \\
c &\coloneqq (1+\eta_x^2)^{-3/4} , \nonumber \\
\mathtt{p} &\coloneqq c^{-1/3} - \mathrm{i} \frac{5}{18} \chi(\xi) \frac{ \mathtt{A}_{\gamma,\xi}(\xi) }{ \mathtt{A}_{\gamma}(\xi) } c^{-4/3} c_x , \nonumber \\
\mathtt{q} &\coloneqq \chi(\xi) \left[ c^{2/3} \frac{ \mathtt{A}_{\gamma}(\xi) }{ \mathtt{G}_{00}(\xi) } - \mathrm{i} \partial_x(c^{2/3}) \frac{ \mathtt{A}_{\gamma}(\xi) }{ \xi \mathtt{G}_{00}(\xi) }\right] , \nonumber 
\end{align}
where the symbols $\mathtt{A}_{\gamma}$ and $\mathtt{G}_{00}$ are given by \eqref{eq:AgSymb}, and where $\chi \in C^{\infty}(\mathbb{R})$ satisfies $\chi \equiv 1$ for $|\xi| \geq \frac{2}{3}$ and $\chi \equiv 0$ for $|\xi| \leq \frac{1}{2}$.

Then the unknown 
\begin{align*} 
u &\coloneqq T_{\mathtt{p}} \omega - \mathrm{i} T_{\mathtt{q}} \eta
\end{align*}
solves
\begin{align} \label{eq:SymmConstVort}
u_t + T_{V(\eta)\psi - \gamma \eta} u_x + \mathrm{i} A_{\gamma}^{1/2} ( T_c A_{\gamma}^{1/2} u ) + R(\eta,\psi,\gamma,\kappa) &= T_{\mathtt{p}} \mathscr{P}_{ext} ,
\end{align}
where the remainder $R(\eta,\psi,\gamma,\kappa)$ satisfies 
\begin{align}
R(\eta,\psi,\gamma,\kappa) &= R_1(\eta,\gamma)\psi + R_2(\eta,\gamma,\kappa)\eta , \nonumber \\
\| R_1(\eta,\gamma) \psi \|_{H^s(\mathbb{T})} &\leq C \left( \|\eta\|_{H^{s+1/2}(\mathbb{T})} , |\gamma| \right) \, \|\eta\|_{H^{s+1/2}(\mathbb{T})}^{\theta} \, \|\psi\|_{H^{s}(\mathbb{T})} , \nonumber \\
\| R_2(\eta_1,\gamma,\kappa) \eta_2 \|_{H^s(\mathbb{T})} &\leq C \left( \|\eta_1 \|_{H^{s+1/2}(\mathbb{T})} , |\gamma| , \kappa \right) \, \|\eta_1 \|_{H^{s+1/2}(\mathbb{T})}^{\theta} \, \| \eta_2 \|_{H^{s+1/2}(\mathbb{T})} , \label{eq:EstRemConstVort}
\end{align}
for some $\theta \in (0,1]$.

\end{proposition}

\begin{proof}

By paralinearizing the system \eqref{eq:WWsysV} (see Proposition 4.4 of \cite{pasquali2026two}) and by using the fact that $G_{00} - |D|$ is a smoothing operator (see proof of Proposition 2.4 in \cite{alazard2018control}), we obtain 
\begin{equation} \label{eq:WWsysVpara}
\begin{cases}
\eta_t + T_{V(\eta)\psi - \gamma \eta} \eta_x - G_{00} \omega &= f_1 \\
\omega_t + T_{V(\eta)\psi - \gamma \eta} \omega_x + \kappa T_{\mathtt{h}} \eta &= f_2 + \mathscr{P}_{ext}
\end{cases}
,
\end{equation} 
where 
\begin{align*}
\mathtt{h} = \mathtt{h}^{(2)} + \mathtt{h}^{(1)}, &\;\; \mathtt{h}^{(2)} = \frac{ \xi^2 }{ (1+\eta_x^2)^{3/2} }  , \;\; \mathtt{h}^{(1)} = - \frac{\mathrm{i}}{2} \, (\partial_x \, \partial_{\xi}) \mathtt{h}^{(2)} , \\
f_1 \in L^{\infty}([0,T]; H^{s+1/2}(\mathbb{T}) ) , &\;\; f_2 \in L^{\infty}([0,T]; H^{s}(\mathbb{T}) ) , \\
\| (f_1,f_2) \|_{ L^{\infty}([0,T]; H^{s+1/2}(\mathbb{T}) \times H^{s}(\mathbb{T}) ) } &\leq  C \left( \| (\eta,\psi) \|_{ L^{\infty}( [0,T]; H^{s+1/2}_0(\mathbb{T}) \times H^{s}(\mathbb{T}) )  }  \right) ,
\end{align*}
for some non-decreasing function $C$ depending on $\eta_0$, $h$, $h_0$, $|\gamma|$, $\kappa$.

Then we set $\theta \coloneqq T_{\mathtt{p}} \omega$, $\zeta \coloneqq T_{\mathtt{q}} \eta$. Therefore \eqref{eq:WWsysVpara} takes the form
\begin{equation} \label{eq:WWsysVpara2}
\begin{cases}
\zeta_t + T_{V(\eta)\psi - \gamma \eta} \zeta_x - T_q G_{00} \omega &= \tilde{f}_1 \\
\theta_t + T_{V(\eta)\psi - \gamma \eta} \theta_x + \kappa T_{ \mathtt{p} \mathtt{h} } \eta &= \tilde{f}_2 + T_{ \mathtt{p} } \mathscr{P}_{ext}
\end{cases}
,
\end{equation} 
where
\begin{align*}
\tilde{f}_1 &\coloneqq T_{\mathtt{q}} f_1 + T_{ \mathtt{q}_t } \eta + \left[ T_{V(\eta)\psi - \gamma \eta} \partial_x , T_{\mathtt{q}} \right] \, \eta , \\
\tilde{f}_2 &\coloneqq T_{\mathtt{p}} f_2 + T_{ \mathtt{p}_t } \omega + \left[ T_{V(\eta)\psi - \gamma \eta} \partial_x , T_{\mathtt{p}} \right] \, \omega +\kappa \left( T_{ \mathtt{p} \mathtt{h} } - T_{ \mathtt{p} } T_{ \mathtt{h} } \right) \eta .
\end{align*}

Choosing $\mathtt{q}$ and $\mathtt{p}$ as in the statement of the proposition and by using Proposition 2.4 of \cite{alazard2018control}, we can deduce that there exists $\theta \in (0,1]$ such that 
\begin{align*}
\| \tilde{f}_j \|_{ H^s(\mathbb{T}) } &\leq C \left( \| \eta \|_{ H^{s+1/2}(\mathbb{T}) } ,|\gamma|,\kappa \right) \, \| \eta \|_{ H^{s+1/2}(\mathbb{T}) }^{\theta} \, \left[ \| \psi \|_{ H^{s}(\mathbb{T}) } + \| \eta \|_{ H^{s+1/2}(\mathbb{T}) }  \right] ,
\end{align*}
for $j=1,2$. In order to deduce the thesis, we need to prove that 
\begin{align}
& \| T_{\mathtt{q}} G_{00}\omega - A_{\gamma}^{1/2}  T_c A_{\gamma}^{1/2} T_{\mathtt{p}} \omega \|_{ H^s(\mathbb{T}) } \leq C \left( \|\eta\|_{H^{s+1/2}(\mathbb{T})} , |\gamma| \right) \, \|\eta\|_{H^{s+1/2}(\mathbb{T})}^{\theta} \, \| \omega \|_{H^{s}(\mathbb{T})} , \nonumber \\
& \| \kappa T_{\mathtt{p} \mathtt{h}} \eta - A_{\gamma}^{1/2}  T_c A_{\gamma}^{1/2} T_{\mathtt{p}} \eta \|_{ H^s(\mathbb{T}) } \leq C \left( \|\eta\|_{H^{s+1/2}(\mathbb{T})} , |\gamma| , \kappa \right) \, \|\eta\|_{H^{s+1/2}(\mathbb{T})}^{\theta} \, \| \eta \|_{H^{s+1/2}(\mathbb{T})} , \label{eq:SymmTechEst21}
\end{align}

%We introduce the following notation: given two operators $A$ and $B$, we write $A \sim B$ if for any $\mu \in \mathbb{R}$ there exists a positive constant $C \left( \| \eta \|_{ H^{s+1/2}(\mathbb{T}) } \right)$ such that
%\begin{align*}
%\| (A-B) f \|_{ H^{\mu}(\mathbb{T}) } &\leq C \left( \| \eta \|_{ H^{s+1/2}(\mathbb{T}) } \right) \| \eta \|_{ H^{s+1/2}(\mathbb{T}) } \| \eta \|_{ H^{\mu}(\mathbb{T}) } .
%\end{align*}

In order to prove \eqref{eq:SymmTechEst21}, we need to show that
\begin{align} \label{eq:SymmTechEst22}
T_{\mathtt{q}} G_{00} \sim A_{\gamma}^{1/2}  T_c A_{\gamma}^{1/2} \chi(D) T_{\mathtt{p}} , &\quad \kappa T_{\mathtt{p} \mathtt{h}} \chi(D) \sim A_{\gamma}^{1/2}  T_c A_{\gamma}^{1/2} \chi(D) T_{\mathtt{q}}  ,
\end{align}
where 
\begin{align*}
\mathtt{q} &= \mathtt{q}^{(1/2)} + \mathtt{q}^{(-1/2)}, \quad \mathtt{q}^{(1/2)} \in \Gamma^{1/2}_2, \quad \mathtt{q}^{(-1/2)} \in \Gamma^{-1/2}_1, \\
\mathtt{p} &= \mathtt{p}^{(0)} + \mathtt{p}^{(-1)}, \quad \mathtt{p}^{(0)} \in \Gamma^{0}_2, \quad \mathtt{p}^{(-1)} \in \Gamma^{-1}_1 .
\end{align*}

From symbolic calculus we have 
\begin{align*}
A_{\gamma}^{1/2}  T_c A_{\gamma}^{1/2} \chi(D)  &\sim T_{\mathtt{g}}, \quad \mathtt{g} = \chi c \mathtt{A}_{\gamma} - \mathrm{i} \chi \, \partial_{\xi} ( \mathtt{A}_{\gamma}^{1/2} ) \, \mathtt{A}_{\gamma}^{1/2} \, c_x ,
\end{align*}
while from \eqref{eq:SymmTechEst22} we obtain 
\begin{align*}
T_{\mathtt{g}} T_{\mathtt{p}} &\sim T_{\mathtt{g}_1} , \quad \mathtt{g}_1 = \mathtt{g} \mathtt{p}^{(0)} + \chi c \mathtt{A}_{\gamma} \mathtt{p}^{(-1)} - \mathrm{i} \chi c \, \mathtt{A}_{\gamma,\xi} \mathtt{p}^{(0)}_x ,
\end{align*}
hence we choose
\begin{align} \label{eq:qformulaVort}
\mathtt{q}^{(1/2)} \coloneqq \chi c \mathtt{p}^{(0)} \frac{ \mathtt{A}_{\gamma} }{ \mathtt{G}_{00} }, &\quad 
\mathtt{q}^{(-1/2)} \coloneqq -\mathrm{i} \chi \frac{ \mathtt{A}_{\gamma,\xi} }{ \mathtt{G}_{00} } \left[ \frac{1}{2} c_x \mathtt{p}^{(0)} + c \mathtt{p}^{(0)}_x \right] + \chi c \mathtt{p}^{(-1)} \frac{ \mathtt{A}_{\gamma} }{ \mathtt{G}_{00} } .
\end{align}

Similarly, we have $T_{\mathtt{g}} T_{\mathtt{q}} \sim T_{\mathtt{g}_2}$, where
\begin{align*}
\mathtt{g}_2 &= \mathtt{g} \mathtt{q}^{(1/2)} - \mathrm{i} \chi c \mathtt{A}_{\gamma,\xi} \mathtt{q}^{(1/2)}_x + \chi c \mathtt{A}_{\gamma} \mathtt{q}^{(-1/2)} \\
&= \chi \left[  c \, \mathtt{A}_{\gamma} \mathtt{q}^{(1/2)} -\mathrm{i} \chi \frac{ \partial_{\xi}( \mathtt{A}_{\gamma}^2 ) }{ \mathtt{G}_{00} } \left( c \, c_x \mathtt{p}^{(0)} + c^2 \mathtt{p}^{(0)}_x \right) + \chi c^2 \mathtt{p}^{(-1)} \frac{ \mathtt{A}_{\gamma}^2 }{ \mathtt{G}_{00} } \right] .
\end{align*}
Notice that
\begin{align*}
\frac{ \partial_{\xi}( \mathtt{A}_{\gamma}^2 ) }{ \mathtt{G}_{00} } = 3 \kappa \xi + r_1 , &\quad
\frac{  \mathtt{A}_{\gamma}^2  }{ \mathtt{G}_{00} } = \kappa \xi^2 + \frac{\gamma}{2} \kappa^{1/2} |\xi|^{1/2} + r_2,
\end{align*}
where $r_1$ and $r_2$ have order $0$. Observe that in the formula for $\mathtt{g}_2$ the terms $r_1 \left( c \, c_x \mathtt{p}^{(0)} + c^2 \mathtt{p}^{(0)}_x \right)$ and $c^2 \mathtt{p}^{(-1)} \left(  \frac{\gamma}{2} \kappa |\xi|^{1/2} + r_2 \right)$ can be regarded as remainders, so that
\begin{align*}
 A_{\gamma}^{1/2}  T_c A_{\gamma}^{1/2} T_{\mathtt{q}}  &\sim T_{\mathtt{g}_3} , \quad \mathtt{g}_3 = \chi \left[  c \, \mathtt{A}_{\gamma} \mathtt{q}^{(1/2)} -3 \mathrm{i} \chi \xi \left( c \, c_x \mathtt{p}^{(0)} + c^2 \mathtt{p}^{(0)}_x \right) + \chi c^2 \mathtt{p}^{(-1)} \xi^2 \right] .
\end{align*}

Similarly, we have
\begin{align*}
\kappa T_{\mathtt{p} \mathtt{h}} \chi(D) &\sim  \kappa T_{\chi \mathtt{p} ( \mathtt{h}^{(2)} + \mathtt{h}^{(1)} ) }  ,
\end{align*}
where 
\begin{align*}
\kappa \mathtt{h}^{(2)} = \kappa c^2 \xi^2 &= c^2 \left( \frac{ \mathtt{A}_{\gamma}^2 }{ \mathtt{G}_{00} } - \frac{\gamma}{2} \kappa^{1/2} |\xi|^{1/2} - r_1 \right) ,
\end{align*}
hence by \eqref{eq:qformulaVort}
\begin{align*}
& \chi \mathtt{p} c^2 \left( \frac{ \mathtt{A}_{\gamma}^2 }{ \mathtt{G}_{00} } - \frac{\gamma}{2} \kappa^{1/2} |\xi|^{1/2} - r_1 \right) \\
&= c \, \mathtt{A}_{\gamma} \mathtt{q}^{(1/2} + \chi \left( \mathtt{p}^{(-1)} c^2 \frac{ \mathtt{A}_{\gamma}^2 }{ \mathtt{G}_{00} } - \frac{\gamma}{2} \kappa^{1/2} \mathtt{p} c^2 |\xi|^{1/2} - \mathtt{p} c^2 r_1 \right) ,
\end{align*}
and
\begin{align*}
\kappa T_{\mathtt{p} \mathtt{h}} \chi(D) &\sim T_{ c \, \mathtt{A}_{\gamma} \mathtt{q}^{(1/2)} + \chi \left( \mathtt{p}^{(-1)} c^2 \frac{ \mathtt{A}_{\gamma}^2 }{ \mathtt{G}_{00} } - \frac{\gamma}{2} \kappa^{1/2} \mathtt{p} c^2 |\xi|^{1/2} - \mathtt{p} c^2 r_1 \right) } .
\end{align*}
Recalling the formulae for $\mathtt{g}_2$ and $\mathtt{g}$, we obtain
\begin{align*}
\mathtt{g}_1 &=  \chi c \mathtt{A}_{\gamma} \mathtt{q}^{(1/2)} - \mathrm{i} \chi \, \partial_{\xi} ( \mathtt{A}_{\gamma}^{1/2} ) \, \mathtt{A}_{\gamma}^{1/2} \, c_x   \mathtt{q}^{(1/2)} - \mathrm{i} \chi c \mathtt{A}_{\gamma,\xi} \mathtt{q}^{(1/2)}_x + \chi c \mathtt{A}_{\gamma} \mathtt{q}^{(-1/2)} ,
\end{align*}
but
\begin{align*}
& \partial_{\xi} ( \mathtt{A}_{\gamma}^{1/2} ) \, \mathtt{A}_{\gamma}^{1/2} \, c_x   \mathtt{q}^{(1/2)} +  c \mathtt{A}_{\gamma,\xi} \mathtt{q}^{(1/2)}_x \\
%&= \partial_{\xi}( \mathtt{A}_{\gamma}^{1/2} ) \mathtt{A}_{\gamma}^{1/2} c_x \chi c \mathtt{p}^{(0)} \frac{ \mathtt{A}_{\gamma} }{ \mathtt{G}_{00} } \\
%&\quad + c \left[ \chi_x c \mathtt{p}^{(0)} \frac{ \mathtt{A}_{\gamma} }{ \mathtt{G}_{00} } + \chi c_x \mathtt{p}^{(0)} \frac{ \mathtt{A}_{\gamma} }{ \mathtt{G}_{00} } + \chi c \mathtt{p}^{(0)}_x \frac{ \mathtt{A}_{\gamma} }{ \mathtt{G}_{00} }+ \chi c \mathtt{p}^{(0)} \partial_x \left( \frac{ \mathtt{A}_{\gamma} }{ \mathtt{G}_{00} } \right) \right] \\
&= 3 \xi \left( c \, c_x \mathtt{p}^{(0)} + c^2 \mathtt{p}^{(0)}_x \right) + r_3 ,
\end{align*}
where $r_3$ has order $0$, hence we can deduce the second identity in \eqref{eq:SymmTechEst22}.

Finally, from \eqref{eq:qformulaVort} we have
\begin{align*}
\mathtt{q} &= \chi \left[ c \mathtt{p}^{(0)} \frac{ \mathtt{A}_{\gamma} }{ \mathtt{G}_{00} } -\mathrm{i}  \frac{ \mathtt{A}_{\gamma,\xi} }{ \mathtt{G}_{00} } \left( \frac{1}{2} c_x \mathtt{p}^{(0)} + c \mathtt{p}^{(0)}_x \right) +  c \mathtt{p}^{(-1)} \frac{ \mathtt{A}_{\gamma} }{ \mathtt{G}_{00} } \right] ,
\end{align*}
where
\begin{align*}
\frac{1}{2} c_x \mathtt{p}^{(0)} + c \mathtt{p}^{(0)}_x &= \frac{1}{6} c^{-1/3} c_x ,
\end{align*}
and if we set
\begin{align*}
\mathtt{p}^{(-1)} &\coloneqq -\frac{5}{18} \mathrm{i} \chi \, \frac{ \mathtt{A}_{\gamma,\xi} }{ \mathtt{A}_{\gamma} } c^{-4/3} c_x ,
\end{align*}
we can deduce the thesis.

\end{proof}

Now, using Proposition \ref{prop:symmConstVort} we reduced the water waves system \eqref{eq:WWsysV} to the single equation \eqref{eq:SymmConstVort}, namely
\begin{align*} 
u_t + T_{V(\eta)\psi - \gamma \eta} u_x + \mathrm{i} A_{\gamma}^{1/2} ( T_c A_{\gamma}^{1/2} u ) + R(\eta,\psi,\gamma,\kappa) &= T_{\mathtt{p}} \mathscr{P}_{ext}   ,
\end{align*}
where $c$ depends on the unknown $\eta$, and where $V$ depends on the unknowns $(\eta,\psi)$. We want to prove that $V-\gamma \eta$ and $c$ depend on $u$ and $\gamma$ only; in order to do so, we just need to write the unknowns $(\eta,\omega)$ in terms of $u$. In the following, we denote $W_{\gamma} \coloneqq V-\gamma \eta$; arguing as in Lemma \ref{lem:InvertUn}, we can deduce the following result.

\begin{corollary} \label{cor:ParalinFullVort}

Let $T>0$, $\gamma \in \mathbb{R}$. There exists $s_0>0$ such that for any $s \geq s_0$ the following holds true: let $(\eta,\psi)$ be a solution of \eqref{eq:WWsysV} such that
\begin{equation*}
(\eta,\psi) \in C([0,T]; H^{s+1/2}_0(\mathbb{T}) \times H^s(\mathbb{T}) ) ,
\end{equation*}
and such that \eqref{eq:StrConnectedFlat} holds true. Then the unknown $u \in C([0,T] ; \tilde{H}^s(\mathbb{T}) ; \mathbb{C})$ defined in Proposition \ref{prop:symmConstVort} satisfies
\begin{align} \label{eq:SymmConstVort2}
u_t + T_{W_{\gamma}(u)} u_x + \mathrm{i} A_{\gamma}^{1/2} ( T_{c(u)} A_{\gamma}^{1/2} u ) + R(u,\gamma,\kappa) &= T_{\mathtt{p}(u)} \mathscr{P}_{ext} .
\end{align}

\end{corollary}

We now fix $\underline{u} \in C( [0,T]; \tilde{H}^{s}(\mathbb{T} ; \mathbb{C}) )$, and we set $W_{\gamma} = W_{\gamma}(\underline{u})$, $c = c(\underline{u})$, $R_{\gamma}=R(\underline{u},\gamma,\kappa)$ and $\mathtt{p}=\mathtt{p}(\underline{u})$, and we consider the linear operator 
\begin{align} \label{eq:OpPgamma}
	P_{\gamma} &\coloneqq \partial_t + T_{W_{\gamma}} \partial_x + \mathrm{i} A_{\gamma}^{1/2} ( T_{c} A_{\gamma}^{1/2} \cdot ) + R_{\gamma} .
\end{align}

We now assume that the following conditions hold true.
\begin{itemize}
	\item[$\mathrm{(A1)}_{\gamma}$] Let $s_0$ be sufficiently large, and let $W_{\gamma},c \in C^0( [0,T] ; H^{s_0}(\mathbb{T}) )$, where $c$ is bounded from below by $\frac{1}{2}$. Moreover, the symbol $\mathtt{p}$ is given by
	\begin{equation*}
		c^{-1/3} - \mathrm{i} \frac{5}{18} \chi(\xi) \frac{ \mathtt{A}_{\gamma,\xi}(\xi) }{ \mathtt{A}_{\gamma}(\xi) } c^{-4/3} c_x ,
	\end{equation*}
	and $\| c - 1 \|_{ W^{3/2,\infty}(\mathbb{T}) }$ is sufficiently small.
	\item[$\mathrm{(A2)}_{\gamma}$]If $P_{\gamma}u$ is a real-valued function, then $\frac{\mathrm{d}}{\mathrm{d}t} \int_{\mathbb{T}} \mathrm{Im} \, u(t,x) \mathrm{d}x = 0$.
\end{itemize}

We now fix a non-empty open domain $\Omega \subset \mathbb{T}$, and let us denote by $\chi_{\Omega}$ be a cut-off function such that $\chi_{\Omega} \equiv 1$ on $\Omega$. We consider the following problem: given $v_0 \in \tilde{H}^s(\mathbb{T};\mathbb{C})$, we want to determine whether there exists $f \in C^0( [0,T] ; H^s(\mathbb{T}) )$ such that the unique solution to the Cauchy problem
\begin{equation} \label{eq:CauchyProbGamma}
	\begin{cases}
		P_{\gamma}v &= T_{\mathtt{p}} \chi_{\Omega} \mathrm{Re} f , \\
		v(0,\cdot) &= v_0 ,
	\end{cases}
\end{equation}
satisfies $v(T,\cdot) =0$.

The statement analogous to Proposition \ref{prop:ParaControl} is the following one.

\begin{proposition} \label{prop:ParaControlGamma}
	
	There exists $s_0 \gg 1$ such that for all $T \in (0,1]$, for all $s \geq s_0$, if assumptions $\mathrm{(A1)}_{\gamma}-\mathrm{(A2)}_{\gamma}$ hold true, then there exist two positive constants $\delta=\delta(T,s,|\gamma|)$ and $K=K(T,s,|\gamma|)$ such that if
	\begin{equation} \label{eq:SmallParaControlVort}
		\begin{cases}
			\| W_{\gamma} \|_{ C^0( [0,T] ; H^{s_0}(\mathbb{T}) ) } + \| c-1 \|_{ C^0( [0,T] ; H^{s_0}(\mathbb{T}) ) } &\leq \delta , \\
			\| \partial_t^k W_{\gamma} \|_{ C^0( [0,T] ; H^{1}(\mathbb{T}) ) } + \| \partial_t^k c \|_{ C^0( [0,T] ; H^{1}(\mathbb{T}) ) } &\leq \delta , \quad k=1,2,3, \\
			\| R_{\gamma} \|_{ C^0 \left( [0,T] ; L( H^{s}(\mathbb{T}) ) \right) } &\leq \delta ,
		\end{cases}
	\end{equation}
	then for any $v_0 \in \tilde{H}^s(\mathbb{T};\mathbb{C})$ there exists $f \in C^0( [0,T] ; H^s(\mathbb{T}) )$ such that
	\begin{itemize}
		\item[i.] the unique solution to the Cauchy problem \eqref{eq:CauchyProbGamma} satisfies $v(T,\cdot) =0$;
		\item[ii.] $\| f \|_{ C^0( [0,T] ; H^s(\mathbb{T}) ) } \leq K \| v_0 \|_{ \tilde{H}^s(\mathbb{T};\mathbb{C}) }$.
	\end{itemize}
	
\end{proposition}

The reductions performed in Sec. \ref{sec:Red} can be done in a similar way: the result corresponding to Proposition \ref{prop:reduction} is the following one.

\begin{proposition} \label{prop:reductionGamma}
	
	Let $T\in (0,1]$, and consider an open non-empty set $\Omega \subset \mathbb{T}$.
	
	There exists $s_0 \gg 1$ such that for all $s \geq s_0$ the following holds true: consider an operator of the form
	\begin{align} \label{eq:tildePgamma}
		\tilde{P}_{\gamma} &\coloneqq \partial_t + W_{\gamma} \partial_x + \mathrm{i} A_{\gamma}^{1/2} \left( c A_{\gamma}^{1/2}  \cdot \right) + R_{2,\gamma} ,
	\end{align}
	then there exist two positive constants $\varepsilon=\varepsilon(T,|\gamma|)$ and $K=K(T,|\gamma|)$ 
	such that, if
	\begin{align*} 
		\| W_{\gamma} \|_{C^0([0,T],H^{s_0}(\mathbb{T}) )}+ \| c-1 \|_{C^0([0,T],H^{s_0}(\mathbb{T})  )} &\leq \varepsilon, \\
		\| \partial_t^k W_{\gamma} \|_{C^0( [0,T],H^1(\mathbb{T}) )}
		+\| \partial_t^k c \|_{ C^0( [0,T],H^1(\mathbb{T}) )} &\leq \varepsilon , k=1,2,3,\\
		\| R_{2,\gamma} \|_{ C^0( [0,T],\mathcal{L}(L^2(\mathbb{T})  )  )} &\leq \varepsilon ,
	\end{align*}
	then for any initial data $v_{0} \in L^2(\mathbb{T})$ there exists $f\in C^0([0,T];L^2(\mathbb{T}))$ such that: 
	\begin{itemize}
		\item 
		the unique solution $v$ to $\tilde{P}_{\gamma}v=\chi_{\Omega} \, \mathrm{Re} f,\quad v(t=0,\cdot)=v_{0}$ is such that $v(T)$ is an imaginary constant, namely 
		$\exists b \in \mathbb{R}$ such that $v(T,x)=\mathrm{i} b$ $\forall x \in \mathbb{T}$.
		
		\item  $\| f \|_{ C^0( [0,T],L^2(\mathbb{T}) )} \leq K \| v_{0} \|_{ L^2(\mathbb{T}) }$. 
	\end{itemize}
	
\end{proposition}

We can also perform the reduction to a constant-coefficient operator with a similar argument: namely, we first conjugate $\tilde{P}_{\gamma}$ to an operator of the form
\begin{align*}
	\partial_t + Z_{\gamma} \partial_x + \mathrm{i} A_\gamma + R_{4,\gamma}
\end{align*}
where $R_{4,\gamma}$ is a pseudodifferential operator of order $0$, and where $Z_{\gamma} = Z_{\gamma}(t,x)$ has zero average, and then we can prove the following result.

\begin{corollary} \label{cor:ReductionGamma}
	Assume that $s_0$ is sufficiently large, and that the quantity 
	\begin{align} 
		\mathcal{N}_{\gamma} &\coloneqq  \| Z_{\gamma} \|_{ C([0,T];H^{s_0}(\mathbb{T})) } + \| c-1 \|_{ C([0,T];H^{s_0}(\mathbb{T})) } \nonumber \\
		&\quad + \| c_t \|_{ C([0,T];H^1(\mathbb{T})) }	+ \| R_{2,\gamma} \|_{ C([0,T];\mathcal{L}(L^2(\mathbb{T}))) }, \label{eq:calNgamma}
	\end{align}
	is sufficiently small. Let us consider the operator
	\begin{align*}
		\mathcal{A} &\coloneqq \mathrm{Op}\big(q(t,x,\xi)e^{\mathrm{i}b(t,x)|\xi|^{1/2}}\big)
	\end{align*}
	with
	\begin{align*}
		b=b_0(t)+\frac{2}{3}\partial_x^{-1}Z_{\gamma} ,
	\end{align*}	
	where $b_0$ is determined by \eqref{eq:b0choice}, and $q=e^{\mathtt{g}}$ where $\mathtt{g}$ is given by \eqref{eq:ttg}. 
	Then 
	\begin{align*}
		\left[ \partial_t +Z_{\gamma}\partial_x+\mathrm{i} A_{\gamma} +R_{4,\gamma} \right] \mathcal{A} &= \mathcal{A} \left[ \partial_t + \mathrm{i} A_{\gamma} + R_{5,\gamma} \right] ,
	\end{align*}
	where	
	\begin{align*}
		\| R_{5,\gamma} \|_{C([0,T];\mathcal{L}(L^2(\mathbb{T})) )}\leq C \mathcal{N} .
	\end{align*}
	
\end{corollary}

The rest of the proof of Theorem \ref{thm:mainConstVort} follows along the lines of Secs. 6-10 of \cite{alazard2018control}. Indeed, due to the fact that $A_{\gamma}$ in \eqref{eq:AgOp} is a constant coefficient operator which is a lower-order perturbation of $A_0$, one can prove Ingham's type estimates as for the irrotational case, since the analogue of Proposition \ref{prop:HighFreq} can be proved with a similar argument, while the analogue of Proposition \ref{prop:InghamLowFinal} can be proved as Proposition 6.5 of \cite{alazard2018control}.

Next, arguing as in Corollary \ref{cor:ObsFull} we can prove an observability result for equations of the form 
\begin{equation*}
	w_t+Z_{\gamma} w_x+\mathrm{i}A_{\gamma}w+\mathscr{R}w=0,
\end{equation*}
where $\mathscr{R}$ is an operator of order $0$.

Using that $\eta \in H^s_0(\mathbb{T})$, one can prove the analogue of Proposition \ref{prop:ControlLin} for 
\begin{align} 
	Q_{\gamma} w &= \chi_\omega \operatorname{Re} f,\qquad w(0)=w_{0} , \label{eq:ControlEqGamma} \\
	Q_{\gamma} &\coloneqq \partial_t + Z_{\gamma}\partial_x + \mathrm{i} A_{\gamma} + \mathscr{R} . \nonumber 
\end{align}

Finally, arguing as in Sec. \ref{sec:NashMoser}, we can deduce Theorem \ref{thm:mainConstVort}.

\section*{Acknowledgements}

The authors would like to thank Thomas Alazard for suggesting the setting of the problem.

S.P. is supported by the European Union ERC CONSOLIDATOR GRANT 2023 GUnDHam, Project Number: 101124921. The views and opinions expressed are, however, those of the
authors only and do not necessarily reflect those of the European Union or the European Research Council. Neither the European Union nor the granting authority can be held responsible for them. 

Both authors were supported by PRIN 2022HSSYPN ``Turbulent effects vs Stability in Equations from Oceanography'' and by INdAM-GNAMPA.

\begin{appendix}

\section{Paradifferential calculus} \label{sec:Paradiff}

Here we recall some standard rules of paradifferential calculus (see Appendix A in \cite{alazard2018control} and Sec. 4 of \cite{alazard2009paralinearization}). We recall that we defined the paradifferential operators in \eqref{eq:ParadiffOp}: given a symbol $a \in \Sigma^m_\varrho(\mathbb{T})$, we define the paradifferential operator $T_a$ by
\begin{align*}
\mathscr{F}(T_au)(\xi) &:= (2\pi)^{-1} \sum_{\eta \in \mathbb{Z}} \chi(\xi-\eta,\eta) \; \; \mathscr{F}a(\xi-\eta,\eta) \; \; \mathscr{F}u(\eta) , 
\end{align*}
where $\mathscr{F}a$ is the Fourier transform of $a$ with respect to the first variable.  \\

Observe that, due to \eqref{eq:chiSymm}, if $a$ and $u$ are real-valued functions, so is $T_a u$.

\begin{definition} \label{def:OrdOp}
Let $m \in \mathbb{R}$, then an operator $T$ is of order $m$ if, for all $s \in \mathbb{R}$, it is bounded from $H^{s}(\mathbb{T})$ to $H^{s-m}(\mathbb{T})$.
\end{definition}

\begin{proposition} \label{prop:OrdParadiffOp}
Let $m \in \mathbb{R}$ and let $a \in \Gamma^m_0(\mathbb{T})$, then $T_a$ is of order $m$. Moreover, for any $s \in \mathbb{R}$ there exists $K > 0$ such that
\begin{align*}
\| T_a \|_{L( H^s(\mathbb{T}) ,  H^{s-m}(\mathbb{T}) )} &\leq K \, M^m_0(a) .
\end{align*}
\end{proposition}

Recall the following properties
\begin{proposition} \label{prop:CompParadiffOp}
Let $0 <\varrho \leq 1$, $m_1,m_2 \in \mathbb{R}$ and let $a \in \Gamma^{m_1}_{\varrho}(\mathbb{T})$, $b \in \Gamma^{m_2}_{\varrho}(\mathbb{T})$, then $T_a T_b - T_{ab}$ is of order $m_1+m_2-\varrho$. Moreover, for any $s \in \mathbb{R}$ there exists $K > 0$ such that
\begin{align*}
\| T_a T_b - T_{ab} \|_{L( H^s(\mathbb{T}) ,  H^{s-m_1-m_2+\varrho}(\mathbb{T}) )} &\leq K \, M^{m_1}_{\varrho}(a) \, M^{m_2}_{\varrho}(b).
\end{align*}
\end{proposition}

\begin{proposition} \label{prop:AsympParadiffOp}
Let $\varrho >0$, $m_1,m_2 \in \mathbb{R}$ and let $a \in \Gamma^{m_1}_\varrho(\mathbb{T})$, $b \in \Gamma^{m_2}_\varrho(\mathbb{T})$. Set
\begin{align*}
a \sharp b (x,\xi) &:= \sum_{|\alpha| < \varrho} \frac{ (-\mathrm{i})^\alpha }{\alpha!} \; \; \pd_{\xi}^\alpha a(x,\xi) \; \; \pd_{x}^\alpha b(x,\xi) \in \sum_{j < \varrho} \Gamma^{m_1+m_2-j}_{\varrho-j}(\mathbb{T}),
\end{align*}
then $T_a T_b - T_{a \sharp b}$ is of order $\leq m_1+m_2-\varrho$. Moreover, for any $s \in \mathbb{R}$ there exists $K > 0$ such that
\begin{align*}
\| T_a T_b - T_{a \sharp b} \|_{L( H^s(\mathbb{T}) ,  H^{s-m_1-m_2+\varrho}(\mathbb{T}) )} &\leq K \, M^{m_1}_{\varrho}(a) \, M^{m_2}_{\varrho}(b).
\end{align*}
\end{proposition}

If $a=a(x)$ does not depend on the variable $\xi$, then the paradifferential operator $T_a$ is called \emph{paraproduct}. From Proposition \ref{prop:OrdParadiffOp} we have that if $b \in H^\beta(\mathbb{T})$ with $\beta > \frac{1}{2}$, then $T_b$ is of order $0$. Moreover, we recall the following properties

\begin{lemma} \label{lem:Paraprod1}
Let $m_1 \in \mathbb{R}$, $m_2 < 1/2$, $a \in H^{m_1}(\mathbb{T})$, $b \in H^{m_2}(\mathbb{T})$. Then $T_b a \in H^{m_1+m_2-1/2}(\mathbb{T})$.

Moreover, if $a \in L^\infty(\mathbb{T})$, then $T_a$ is an operator of order $\leq 0$, and there exists $C>0$ such that
\begin{align*}
\| T_a u \|_{ H^s(\mathbb{T}) } &\leq C \, \|a\|_{L^\infty(\mathbb{T})} \, \|  u \|_{ H^s(\mathbb{T}) } , \; \; \forall u \in H^s(\mathbb{T}), \; \; \forall s \in \mathbb{R}.
\end{align*}
\end{lemma}

\begin{lemma} \label{lem:Paraprod2}
Let $m > 1/2$, $a \in H^{m}(\mathbb{T})$, $F \in C^{\infty}(\mathbb{R})$. Then 
\begin{equation*}
F(a)-F(0)-T_{F'(a)} a \in H^{2m-1/2}(\mathbb{T}) . 
\end{equation*}

Moreover, let $m_1,m_2 \in \mathbb{R}$ be such that $m_1 + m_2 >0$, and let $a \in H^{m_1}(\mathbb{T})$, $b \in H^{m_2}(\mathbb{T})$. Then 
\begin{align*}
R(a,b) &:= ab - T_a b - T_b a \in H^{m_1 + m_2 - 1/2}(\mathbb{T}) ,
\end{align*}
and there exists $K>0$ such that
\begin{align*}
\| R(a,b) \|_{ H^{m_1 + m_2 - 1/2}(\mathbb{T}) } &\leq K \, \| a \|_{H^{m_1}(\mathbb{T})} \, \| b \|_{H^{m_2}(\mathbb{T})} .
\end{align*}

\end{lemma}

\begin{lemma} \label{lem:Paraprod3}
Let $m_1,m_2 > \frac{1}{2}$. If $a \in H^{m_1}(\mathbb{T})$, $b \in H^{m_2}(\mathbb{T})$, then $T_a T_b - T_{ab}$ is of order $ - \left( \min(m_1,m_2) - \frac{1}{2} \right)$.  Moreover, for any $s \in \mathbb{R}$ there exists $K > 0$ such that
\begin{align*}
\| T_a T_b - T_{a \sharp b} \|_{L( H^s(\mathbb{T}) ,  H^{s+\min(m_1,m_2) - \frac{1}{2}}(\mathbb{T}) )} &\leq K \, \|a\|_{H^{m_1}(\mathbb{T})} \,  \|b\|_{H^{m_2}(\mathbb{T})} .
\end{align*}
\end{lemma}

\begin{lemma} \label{lem:Paraprod4}
Let $m,s \in \mathbb{R}$ be such that $m+s>0$. If $\ell \in \mathbb{R}$ satisfies
\begin{align*}
\ell \leq m , &\;\; \ell < m+s-\frac{1}{2},
\end{align*}
then there exists $K > 0$ such that for all $a \in H^m(\mathbb{T})$ and all $u \in H^s(\mathbb{T})$ we have
\begin{align*}
\| a u - T_a u \|_{ H^{\ell}(\mathbb{T}) } &\leq K \, \|a\|_{H^{m}(\mathbb{T})} \,  \|u\|_{H^{s}(\mathbb{T})} .
\end{align*}
\end{lemma}

\section{A Nash--Moser--H\"ormander theorem} 
\label{sec:NM}

Here we state a Nash--Moser--H\"ormander theorem (see Theorem 2.1 of \cite{baldi2017nash} for a slightly more general statement).

Let $(E_a)_{a \geq 0}$ be a decreasing family of Banach spaces with continuous injections  
$E_b \hookrightarrow E_a$, 
\begin{equation} \label{S0}
\| u \|_a \leq \| u \|_b \quad \text{for} \  a \leq b.	
\end{equation}
Set $E_\infty = \cap_{a\geq 0} E_a$ with the weakest topology making the 
injections $E_\infty \hookrightarrow E_a$ continuous. 
Assume that $S_j : E_0 \to E_\infty$ for $j = 0,1,\ldots$ are linear operators 
such that, with constants $C$ bounded when $a$ and $b$ are bounded, 
and independent of $j$,
\begin{alignat}{2}
\label{S1} 
\| S_j u \|_a 
& \leq C \| u \|_a 
&& \text{for all} \ a;
\\
\label{S2} 
\| S_j u \|_b 
& \leq C 2^{j(b-a)} \| S_j u \|_a 
&& \text{if} \ a<b; 
\\
\label{S3} 
\| u - S_j u \|_b 
& \leq C 2^{-j(a-b)} \| u - S_j u \|_a 
&& \text{if} \ a>b; 
\\ 
\label{S4} 
\| (S_{j+1} - S_j) u \|_b 
& \leq C 2^{j(b-a)} \| (S_{j+1} - S_j) u \|_a 
\quad && \text{for all $a,b$.}
\end{alignat}
Set 
\begin{equation}  \label{new.24}
R_0 u := S_1 u, \qquad 
R_j u := (S_{j+1} - S_j) u, \quad j \geq 1.
\end{equation}
Thus 
\begin{equation} \label{2705.3}
\| R_j u \|_b \leq C 2^{j(b-a)} \| R_j u \|_a \quad \text{for all} \ a,b.
\end{equation}
Bound \eqref{2705.3} for $j \geq 1$ is \eqref{S4}, 
while, for $j=0$, it follows from \eqref{S0} and \eqref{S2}.

We also assume that 
\begin{equation} \label{2705.4}
\| u \|_a^2 \leq C \sum_{j=0}^\infty \| R_j u \|_a^2	\quad \forall a \geq 0,
\end{equation}
with $C$ bounded for $a$ bounded.

Now let us suppose that we have another family $F_a$ of decreasing Banach spaces with smoothing operators having the same properties as above. We use the same notation also for the smoothing operators. 

\begin{theorem} \label{thm:NashMoser}
Let $a_1, a_2, \mathtt{a}, \mathtt{b}, a_0, \mu$ be real numbers with 
\begin{equation} \label{ineq 2016}
0 \leq a_0 \leq \mu \leq a_1, \qquad 
a_1 + \frac{\mathtt{b}}{2} \, < \mathtt{a} < a_1 + \mathtt{b} , \qquad 
2\mathtt{a} < a_1 + a_2. 
\end{equation}
Let $V$ be a convex neighborhood of $0$ in $E_\mu$. 
Let $\Phi$ be a map from $V$ to $F_0$ such that $\Phi : V \cap E_{a+\mu} \to F_a$ 
is of class $C^2$ for all $a \in [0, a_2 - \mu]$, with 
\begin{align} 
\|\Phi''(u)[v,w] \|_a 
& \leq M(a) \big( \| v \|_{a+\mu} \| w \|_{a_0} + \| v \|_{a_0} \| w \|_{a+\mu} +  \| u \|_{a+\mu}  \| v \|_{a_0} \| w \|_{a_0} \big)
\label{Phi sec}
\end{align}
for all $u \in V \cap E_{a+\mu}$, $v,w \in E_{a+\mu}$,
where $M: [0, a_2 - \mu] \to \mathbb{R}$, is a positive and increasing function. 
Assume that $\Phi'(v)$, for $v \in E_\infty \cap V$ 
belonging to some ball $\| v \|_{a_1} \leq \delta_1$,
has a right inverse $\Psi(v)$ mapping $F_\infty$ to $E_{a_2}$, and that
\begin{equation}  \label{tame in NM}
\| \Psi(v)g \|_a \leq 
L(a) \big( \|g\|_{a + \mathtt{b} - \mathtt{a}} + 
 \| v \|_{a + \mathtt{b}} + \| g \|_0 \big)
\quad \forall a \in [a_1, a_2],
\end{equation}
where $L: [a_1, a_2] \to \mathbb{R}$ is a positive and increasing function.

Then for all $A > 0$ there exists $\delta > 0$ such that, 
for every $g \in F_\mathtt{b}$ satisfying
\begin{equation} \label{2705.1}
\sum_{j=0}^\infty \| R_j g \|_\mathtt{b}^2 \leq A^2 \| g \|_\mathtt{b}^2, \quad
\| g \|_\mathtt{b} \leq \delta,
\end{equation}
there exists $u \in E_\mathtt{a}$ solving $\Phi(u) = \Phi(0) + g$.
The solution $u$ satisfies 
\[
\| u \|_\mathtt{a} \leq C  (1 + A) \| g \|_\mathtt{b}, 
\]
where $C$ is a constant depending on $a_1, a_2, \mathtt{a}, \mathtt{b}$. 

Moreover, let $c > 0$
and assume that \eqref{Phi sec} holds for all $a \in [0, a_2 + c - \mu]$,
$\Psi(v)$ maps $F_\infty$ to $E_{a_2 + c}$, 
and \eqref{tame in NM} holds for all $a \in [a_1, a_2 + c]$. 
If $g$ satisfies \eqref{2705.1} and, in addition, $g \in F_{\mathtt{b}+c}$ with
\begin{equation} \label{0406.1}
\sum_{j=0}^\infty \| R_j g \|_{\mathtt{b}+c}^2 \leq A_c^2 \| g \|_{\mathtt{b}+c}^2 
\end{equation}
for some $A_c$, then the solution $u$ belongs to $E_{\mathtt{a} + c}$, 
with 
\begin{equation} \label{0211.10}	
\| u \|_{\mathtt{a}+c} \leq C_{1,c} \big(  \| g \|_\mathtt{b} +  \| g \|_{\mathtt{b}+c} \big) ,
\end{equation}
where $C_{1,c}$ is a positive constant which depends on $a_1, a_2, \mathtt{a}, \mathtt{b}, c$, $A$ and $A_c$.
\end{theorem}

\end{appendix}

\bibliography{Water_Waves_Control}
\bibliographystyle{alpha}

\end{document}